\documentclass{amsart}
\usepackage{amssymb}
\usepackage[all]{xy}
\usepackage{graphicx}
\usepackage[T1]{fontenc}
\usepackage[bitstream-charter]{mathdesign}
\usepackage{hyperref}
\newtheorem{theorem}{Theorem}[section]
\newtheorem{prop}[theorem]{Proposition}
\newtheorem{lemma}[theorem]{Lemma}
\newtheorem{cor}[theorem]{Corollary}

\theoremstyle{remark}
\newtheorem{dfn}[theorem]{Definition}
\newtheorem{remark}[theorem]{Remark}

\def\co{\colon\thinspace}
\def\delbar{\bar{\partial}}
\def\barmu{\bar{\mu}}

\def\ep{\epsilon}
\def\AR{\mathbb{A}_{R_M}}
\def\EB{\mathbb{E}_{\mathcal{B}}}
\def\R{\mathbb{R}}
\def\Z{\mathbb{Z}}
\def\S{Spec\,}
\def\p{\frak{p}}
\def\ev{\mathbf{ev}}

\title[Deformed Hamiltonian Floer theory]{Deformed Hamiltonian Floer theory, capacity estimates, and Calabi quasimorphisms}
\author{Michael Usher}
\date{\today}\address{Department of Mathematics\\University of Georgia\\Athens, GA 30602}\email{usher@math.uga.edu}
\begin{document}
\begin{abstract}
We develop a family of deformations of the differential and of the pair-of-pants product on the Hamiltonian Floer complex of a symplectic manifold $(M,\omega)$ which upon passing to homology yields ring isomorphisms with the \emph{big} quantum homology of $M$.  By studying the properties of the resulting deformed version of the Oh--Schwarz spectral invariants, we obtain a Floer-theoretic interpretation of a result of Lu which bounds the Hofer--Zehnder capacity of $M$ when $M$ has a nonzero Gromov--Witten invariant with two point constraints, and we produce a new algebraic criterion for $(M,\omega)$ to admit a Calabi quasimorphism and a symplectic quasi-state.  This latter criterion is found to hold whenever $M$ has generically semisimple quantum homology in the sense considered by Dubrovin and Manin (this includes all compact toric $M$), and also whenever $M$ is a point blowup of an arbitrary closed symplectic manifold.
\end{abstract}
\maketitle

\section{Introduction}

The three-point genus zero Gromov--Witten invariants of a closed symplectic manifold $(M,\omega)$ can be organized in such a way as to give a product $\ast_0$ on the homology $H_*(M;\Lambda_{\omega})$ of $M$ where $\Lambda_{\omega}$ is a certain Novikov ring.  The resulting ring $(H_*(M,\Lambda_{\omega}),\ast_0)$, called the (undeformed) quantum homology of $(M,\omega)$, has become a fundamental tool of modern symplectic topology, in part due to the existence of a ring isomorphism \cite{PSS} from the quantum homology to Hamiltonian Floer homology with its pair of pants product.

From the early days of physicists' conception of quantum homology (see \cite[p. 323]{Wi90},\cite[Section 3]{Wi91}) it was anticipated that $\ast_0$ would be just one member of a whole family of quantum products $\ast_{\eta}$, with the deformation parameter $\eta$ varying in the homology of $M$,\footnote{our convention will be that, if $\dim M=2n$, $\eta$ is an element of $\oplus_{i=0}^{n-1}H_{2i}(M;\Lambda_{\omega}^{0})$ for a certain subring $\Lambda_{\omega}^{0}<\Lambda_{\omega}$ containing $\mathbb{C}$ (see Section \ref{conv}(c) below).} each of which has its structure constants given by an $\eta$-dependent formula involving counts of genus-zero pseudoholomorphic curves.  As demonstrated in \cite{KM}, the composition law which was used in \cite{RT} to prove the existence and associativity of the small quantum product $\ast_0$ also establishes the corresponding facts for all of the $\ast_{\eta}$; the family of rings $\{(H_*(M;\Lambda_{\omega}),\ast_{\eta})\}$ (or some other structure carrying equivalent information) is known in the literature as the big quantum homology of $(M,\omega)$. 

While much interesting work has been done relating to the rich algebraic and geometric structure intrinsic to the big quantum homology (see for instance \cite{Man}), the big quantum homology does not seem to have had many external applications to problems in symplectic topology.  In this paper we provide some such applications.  The applications have in common the following outline: one shows that a symplectic manifold satisfies some desirable property if there is some $\eta$ such that the ring $(H_*(M,\Lambda_{\omega}),\ast_{\eta})$ obeys a certain condition, and then argues that the condition is satisfied for some $\eta$.  Since the rings $(H_*(M,\Lambda_{\omega}),\ast_{\eta})$ are typically not mutually  isomorphic, the freedom to vary $\eta$ allows one to prove more results than would be available if one confined oneself to using undeformed $(\eta=0)$ quantum homology.

Our applications concern problems in Hamiltonian dynamics, and we bring big quantum homology to bear on these problems by connecting it to Hamiltonian Floer theory. In particular, we prove:
\begin{theorem}\label{floermain} Let $\eta\in \oplus_{i=0}^{n-1}H_*(M;\Lambda_{\omega}^{0})$.  Then for generic Hamiltonians $H\co S^1\times M\to\mathbb{R}$ we may construct an $\mathbb{R}$-filtered, $\eta$-deformed Hamiltonian Floer complex \[ \left(CF(H)=\cup_{\lambda\in\mathbb{R}}CF^{\lambda}(H),\partial^{\eta,H}\right)\] with the following properties:
\begin{itemize} \item[(i)] $(\partial^{\eta,H})^{2}=0$, and where $HF(H)_{\eta}=H_*(CF(H),\partial^{\eta,H})$ is the resulting homology there is an isomorphism $\underline{\Phi}_{\eta,H}^{PSS}\co H_*(M,\Lambda_{\omega})\to HF(H)_{\eta}$.
\item[(ii)] For generic pairs $(H,K)$, where $H\lozenge K\co S^1\times M\to\mathbb{R}$ denotes the concatenation of $H$ and $K$ (Section \ref{conv}(g)), there is a pair-of-pants product $\ast_{\eta}^{Floer}\co CF(H)\otimes CF(H)\to CF(H\lozenge K)$ which is a chain map with respect to the differentials $\partial^{\eta,\cdot}$, restricts for each $\lambda,\mu\in\mathbb{R}$ to a map $CF^{\lambda}(H)\otimes CF^{\mu}(H)\to CF^{\lambda+\mu}(H\lozenge K)$, and whose induced map on homology $\underline{\ast}^{Floer}_{\eta}$ fits into a commutative diagram 
\[  \xymatrix{ H_*(M;\Lambda_{\omega})\otimes H_*(M;\Lambda_{\omega})\ar[d]_{\ast_{\eta}}\ar[rr]^{\underline{\Phi}^{PSS}_{\eta,H}\otimes \underline{\Phi}^{PSS}_{\eta,K}}& & HF(H)_{\eta}\otimes HF(K)_{\eta} \ar[d]^{\underline{\ast}^{Floer}_{\eta}} \\ H_*(M;\Lambda_{\omega})\ar[rr]^{\underline{\Phi}^{PSS}_{\eta,H\lozenge K}} & & HF(H\lozenge K)_{\eta}
} \] \end{itemize}
\end{theorem}
For proofs of the various statements in Theorem \ref{floermain}, see Corollary \ref{delsquared} and Propositions \ref{prodfilt} and \ref{ringiso}.
 
 Theorem \ref{floermain} allows one to define deformed versions $\rho(a;H)_{\eta}$ of the Oh-Schwarz spectral invariants (\cite{Sc00},\cite{Oh06})
in a standard way: where $a\in H_*(M;\Lambda_{\omega})$ and $H$ is suitably nondegenerate $\rho(a;H)_{\eta}$ is the infimal filtration level of a chain representing the homology class $\underline{\Phi}_{\eta,H}^{PSS}a$. One can show that these invariants satisfy a number of properties similar to familiar properties from the undeformed case; see Proposition \ref{rhoprop} for a list of several of these.  Among these properties we mention in particular that $\rho(a;\cdot)_{\eta}$ extends by continuity to arbitrary continuous functions $H\co S^1\times M\to\mathbb{R}$, and that we have a triangle inequality \begin{equation}\label{inttri} \rho(a\ast_{\eta}b;H\lozenge K)_{\eta}\leq \rho(a;H)_{\eta}+\rho(b;K)_{\eta}.\end{equation} 

The spectral invariants of most Hamiltonians are difficult to compute, so (\ref{inttri}) is an important tool in understanding them; since the sort of information conveyed by (\ref{inttri}) changes whenever the quantum multiplication $\ast_{\eta}$ changes, we can begin to see how it may be useful to study the spectral invariants for all values of the deformation parameter $\eta$, rather than just those for $\eta=0$ as has been done in the past.

Let us now discuss our applications.
\subsection{Hofer--Zehnder capacity}  The $\pi_1$-sensitive Hofer--Zehnder capacity of the closed symplectic manifold $(M,\omega)$ may be expressed as \[ c_{HZ}^{\circ}(M,\omega)=\sup\left\{\max H-\min H\left|\begin{array}{cc}H\co M\to\mathbb{R},\mbox{ all contractible periodic orbits of the}\\ \mbox{Hamiltonian vector field $X_H$ of period $\leq 1$ are constant}\end{array}\right.\right\}.\]  While it often holds that $c_{HZ}^{\circ}(M,\omega)=\infty$, the case when $c_{HZ}^{\circ}(M,\omega)<\infty$ is of considerable interest, for instance because it implies  the Weinstein conjecture for contact (or indeed stable) hypersurfaces of $M$.  Using the deformed Hamiltonian Floer complexes, we give a new proof of the following theorem of Lu (see also \cite{HV},\cite{LiuT00} for earlier related results).

\begin{theorem}[{\cite[Corollary 1.19]{Lu06}}] \label{gwhz} Suppose that $(M,\omega)$ admits a nonzero Gromov--Witten invariant of the form \[ \langle [pt],a_0,[pt],a_1,\ldots,a_k\rangle_{0,k+3,A},\] where $A\in H_2(M;\mathbb{Z})/torsion$ and $a_0,\ldots,a_k$ are rational homology classes of even degree.  Then \[ c_{HZ}^{\circ}(M,\omega)\leq \langle[\omega],A\rangle.\]\end{theorem}

We prove this Theorem below as Corollary \ref{mainhzcor}.  Strictly speaking Lu's result is (at least superficially) a bit more general, for instance because it allows for descendant insertions and does not require the $a_i$ to have even degree; however I am not aware of any examples of manifolds satisfying Lu's hypotheses but not those of Theorem \ref{gwhz}.  As seen in \cite[Theorem 1.21]{Lu06}, the estimate provided by Theorem \ref{gwhz} is sharp for $M$ equal to a product $\prod_{i=1}^{k}\mathbb{C}P^{n_i}$ with any split symplectic form which is a multiple of the standard one on each factor.

Our proof of Theorem \ref{gwhz} brings it into the purview of Hamiltonian Floer theory.  In fact, we deduce Theorem \ref{gwhz} from part (ii) of the following Theorem \ref{hzestintro} about the deformed spectral invariants, which is new even in the undeformed case $\eta=0$ (though an important special case appears in \cite{U10a}).  As notation, if $H\co S^1\times M\to\mathbb{R}$ is a smooth function, $\bar{H}\co S^1\times M\to\mathbb{R}$ denotes the time-reversed Hamiltonian $\bar{H}(t,m)=-H(t,\phi_{H}^{t}(m))$, whose flow is inverse to that of $H$.

\begin{theorem}\label{hzestintro}  Fix $C>0$ and $\eta\in \oplus_{i=0}^{n-1}H_{2i}(M;\Lambda_{\omega})$ and suppose that one of the following two conditions holds: \begin{itemize} \item[(i)] There are $a,b\in H_*(M;\Lambda_{\omega})\setminus\{0\}$ such that, for all autonomous $H\co M\to\mathbb{R}$, \[ (\rho(a,H)_{\eta}-\nu_{[M]}(a))+(\rho(b,\bar{H})_{\eta}-\nu_{[M]}(b))\leq C;\] \emph{or} \item[(ii)] Where $[pt]$ is the standard generator of $H_0(M;\mathbb{C})$, for all $H\co M\to\mathbb{R}$, \[ \rho([pt];H)_{\eta}+\rho([pt];\bar{H})_{\eta}\geq -C.\] \end{itemize} Then the $\pi_1$-sensitive Hofer-Zehnder capacity of $(M,\omega)$ obeys the bound \[ c_{HZ}^{\circ}(M)\leq C.\]
\end{theorem}

Theorem \ref{hzestintro} is proven below as Corollary \ref{hzest}; we refer there for the definition of the quantity $\nu_{[M]}(a)\in \mathbb{R}\cup\{-\infty\}$, noting here only that $\nu_{[M]}(a)$ is finite if and only if $a\in H_*(M;\Lambda_{\omega})$ has a nontrivial component in $H_{2n}(M;\Lambda_{\omega})$.  Thus whenever $a,b\in H_*(M;\Lambda_{\omega})$ each have nontrivial components in $H_{2n}(M;\Lambda_{\omega})$ and there holds a universal bound $\rho(a;H)_{\eta}+\rho(b;\bar{H})_{\eta}\leq C'$ we obtain an explicit finite bound for the Hofer--Zehnder capacity of $(M,\omega)$.

We obtain Theorem \ref{gwhz} from Theorem \ref{hzestintro} by using the assumed nonzero Gromov--Witten invariant to find a value (indeed an open dense set of values) $\eta\in \oplus_{i=0}^{n-1}H_{2i}(M;\mathbb{C})$ so that $[pt]\ast_{\eta} a_0$ has a nontrivial component in $H_{2n}(M;\Lambda_{\omega})$; the triangle inequality and a standard duality property (Proposition \ref{rhoprop}(vii)) of the spectral invariants allow one to use this to obtain an estimate as in (ii) of Theorem \ref{hzestintro}.  Here it is essential to consider the deformed products $\ast_{\eta}$ and not just $\ast_0$. This point is obvious when the number $k+3$ of insertions in the assumed Gromov--Witten invariant is larger than $3$, since in that case the invariant is unseen by $\ast_0$ whereas it does contribute to $\ast_{\eta}$ for most values of $\eta$. In fact it is useful to allow $\eta$ to vary even in the case that $k+3=3$, since if one considered only $\eta=0$ the possibility would remain that the invariant would be cancelled by invariants coming from other classes in $H_{2}(M;\mathbb{Z})$, whereas  one can show that there are always some values of $\eta$ for which such cancellation does not occur.

\subsection{Calabi quasimorphisms}

Beginning with \cite{EP03}, work of Entov--Polterovich and others has shown that, on some symplectic manifolds, the asymptotic spectral invariant $\tilde{\phi}_H\mapsto \lim_{k\to\infty}\frac{\rho(e;H^{\lozenge k})_{0}}{k}$ (for $H$ normalized and $e$ a well-chosen idempotent with respect to the quantum multiplication $\ast_{0}$) defines (after multiplication by a constant) a so-called Calabi quasimorphism $\mu_{e,0}$ on the universal cover $\widetilde{Ham}(M,\omega)$ of the Hamiltonian diffeomorphism group.\footnote{In this paper we only discuss Calabi quasimorphisms on $\widetilde{Ham}(M,\omega)$; we do not address the interesting question of whether these can be arranged to descend to $Ham(M,\omega)$.  Thus in what follows a ``Calabi quasimorphism on $M$'' should be understood as being defined on $\widetilde{Ham}(M,\omega)$.}  Thus $\mu_{e,0}$ is a homogeneous quasimorphism (i.e. a map which fails to be a homomorphism by a uniformly bounded amount and is a homomorphism when restricted to cyclic subgroups) which, on the subgroup given by the flows generated by Hamiltonians supported in any given displaceable open subset, coincides with the classical Calabi homomorphism of \cite[Proposition II.4.1]{Ban}.  Moreover, in this case the formula $\zeta_{e,0}(H)=\lim_{k\to\infty}\frac{\rho(e;kH)_{0}}{k}$ defines a function $\zeta_{e,0}\co C(M)\to\mathbb{R}$ satisfying the axioms of a symplectic quasi-state, as defined in \cite{EP06}.  These constructions have had many interesting applications, \emph{e.g.}, to the structure of $\widetilde{Ham}(M,\omega)$ and to rigidity properties of subsets of $M$; see for instance \cite{EP03},\cite{EPZ},\cite{EP09}.

For  $\eta\in \oplus_{i=0}^{n-1}H_{2i}(M;\Lambda_{\omega}^{0})$ write $QH(M,\omega)_{\eta}$ for the commutative algebra $(H_{ev}(M;\Lambda_{\omega}),\ast_{\eta})$ given by restricting $\eta$-deformed quantum multiplication to even-dimensional homology.

After establishing basic properties of the deformed spectral invariants as described in Proposition \ref{rhoprop}, we extend the arguments of \cite{EP03},\cite{EP06},\cite{EP08},\cite{U10b} to obtain the following theorem, which summarizes results of Section \ref{Calabi}:

\begin{theorem}\label{qmintro} Let $(M,\omega)$ be a closed symplectic manifold, $\eta\in \oplus_{i=0}^{n-1}H_{2i}(M;\Lambda_{\omega}^{0})$, and $e\in QH(M,\omega)_{\eta}$ with $e\ast_{\eta}e=e$.  \begin{itemize} \item[(i)] Suppose that, for some $C>0$ we have an estimate \begin{equation}\label{eeH} \rho(e;H)_{\eta}+\rho(e;\bar{H})_{\eta}\leq C\end{equation} for all $H\co S^1\times M\to\mathbb{R}$.  Then the functions $\mu_{e,\eta}\co \widetilde{Ham}(M,\omega)\to\mathbb{R}$ and $\zeta_{e,\eta}\co C(M)\to\mathbb{R}$ given by \[ \mu_{e,\eta}(\tilde{\phi}_H)=-vol(M)\frac{\rho(e;H^{\lozenge k})_{\eta}}{k}\quad \zeta_{e,\eta}(H)=\lim_{k\to\infty}\frac{\rho(e;kH)_{\eta}}{k}\] define respectively a Calabi quasimorphism and a symplectic quasi-state.  The quasimorphism $\mu_{e,\eta}$ obeys the stability property of \cite[(3)]{EPZ}.
\item[(ii)] Suppose that we have a direct sum decomposition (of algebras) \[ QH(M,\omega)_{\eta}=F\oplus A \] where $F$ is a field.  Then an estimate (\ref{eeH}) holds with $e$ equal to the identity in the field summand $F$.\end{itemize}
\end{theorem}

Note the similarity of the criteria denoted (i) in, respectively, Theorems \ref{hzestintro} and \ref{qmintro}.  In particular, if $e\in QH(M,\omega)_{\eta}$ is an idempotent which contains a nontrivial component in $H_{2n}(M;\Lambda_{\omega})$, we find that an estimate $\rho(e;H)_{\eta}+\rho(e;\bar{H})_{\eta}\leq C$ simultaneously gives rise to two deep and quite distinct conclusions about Hamiltonian dynamics on $(M,\omega)$: a bound on the $\pi_1$-sensitive Hofer--Zehnder capacity on the one hand, and the existence of a Calabi quasimorphism and a symplectic quasi-state on the other.  In light of Theorem \ref{qmintro}(ii), if $QH(M,\omega)_{\eta}$ is semisimple (i.e. decomposes as a finite direct sum of fields), then at least one of the multiplicative identities in the field summands will satisfy these conditions, since the sum of these multiplicative identities is $[M]$.

Of course, Theorem \ref{qmintro}(ii) begs the question of under what circumstances a decomposition  $QH(M,\omega)_{\eta}=F\oplus A$ with $F$ a field exists.  In the later sections of the paper we make a detailed study of this matter.  To any \emph{deformation class} $(M,\mathcal{C})$ of symplectic manifolds we introduce a ``universal quantum coefficient ring'' $R_M$ and an $R_M$-algebra $\AR$ and define (see Definition \ref{ssdfn} and the end of Section \ref{mult}) what it means for $(M,\mathcal{C})$ to have ``generically field-split big quantum homology'' in terms of the properties of this algebra $\AR$.  Since it is also of interest to know when $QH(M,\omega)_{\eta}$ satisfies the stronger condition of being semisimple (for instance results of \cite{EP09} show that semisimplicity leads to stronger rigidity results in some contexts than does the property of having a field direct summand), we also define a notion of $(M,\mathcal{C})$   having ``generically semisimple big quantum homology.''  Using methods from scheme theory, we show:

\begin{theorem}\label{algintro}  If $(M,\omega)$ is a closed symplectic manifold with $\omega$ belonging to the deformation class $\mathcal{C}$, the following are equivalent: \begin{itemize} \item[(i)] $(M,\mathcal{C})$ has generically field-split (resp. generically semisimple) big quantum homology.
\item[(ii)] There is $\eta\in \oplus_{i=0}^{n-1}H_{2i}(M;\Lambda_{\omega}^{0})$ such that the deformed quantum homology $QH(M,\omega)_{\eta}$ splits as a direct sum $F\oplus A$ where $F$ is a field (resp. splits as a finite direct sum of fields).
\item[(iii)] There is an open dense set $\mathcal{U}\subset \oplus_{i=0}^{n-1}H_{2i}(M;\mathbb{C})$ such that $QH(M,\omega)_{\eta}$ splits as a direct sum $F\oplus A$ where $F$ is a field (resp. splits as a finite direct sum of fields) whenever $\eta\in\mathcal{U}$.
\end{itemize}
\end{theorem}

See Theorem \ref{bigsymp} for a slightly more specific version of this theorem and its proof.

To illustrate that the equivalent conditions in Theorem \ref{algintro} are far from vacuous, we mention:
\begin{theorem}\label{examples}\begin{itemize} \item[(i)] The deformation class of any closed symplectic toric manifold has generically semisimple big quantum homology.
\item[(ii)] The deformation class of the blowup at a point of any closed symplectic manifold has generically field-split big (and also small) quantum homology.\end{itemize}
Consequently, any closed symplectic toric manifold, and any blowup of any closed symplectic manifold at a point, admits a Calabi quasimorphism and a symplectic quasi-state.
\end{theorem}

\begin{proof}
 As noted in Section \ref{lit}, part (i) follows directly from a theorem of \cite{Ir} and Theorem \ref{cvgthm}. Part (ii) is proven as Theorem \ref{blowup}. 
\end{proof}

Part (ii) of Theorem \ref{examples} generalizes a result of McDuff (which appears in \cite{EP08}) which produces a Calabi quasimorphism on the blowup of any non-uniruled symplectic manifold.  Meanwhile (i) generalizes results that have been proven for a variety of toric Fano manifolds; in particular in \cite{OT} and \cite{FOOO10} it is shown that if $M$ is toric and Fano then $QH(M,\omega)_{0}$ is semisimple for generic choices of the toric symplectic form $\omega$.  However Theorem \ref{examples} does not require any Fano condition, and also does not require the symplectic form to be chosen generically.

Let us emphasize some aspects of Theorem \ref{algintro}.  First of all, note that it shows that, as soon we can obtain a single $\eta$ giving rise to a Calabi quasimorphism $\mu_{e,\eta}$ via Theorems \ref{algintro} and \ref{qmintro}, there are in fact uncountably many such $\eta$.  It is natural to ask whether these quasimorphisms are the same.  We do not prove any results in this direction here, but note that recent work of Fukaya--Oh--Ohta--Ono \cite[Theorem 1.10]{FOOO} shows that, in certain toric cases, an uncountable subset of these
quasimorphisms can be seen to be linearly independent by means of their interactions with Lagrangian Floer theory.  Meanwhile it is also possible for the Calabi quasimorphisms $\mu_{e,\eta}$ to change as one varies the idempotent $e$: examples of such behavior (with $\eta=0$) have been found for certain symplectic forms on $\mathbb{C}P^2\#\overline{\mathbb{C}P^2}$ in \cite[Corollary F]{OT} and on $S^2\times S^2$ in \cite{ElP}. 

We also point out that our criterion for the existence of a Calabi quasimorphism on $\widetilde{Ham}(M,\omega)$ depends only on the deformation class of $\omega$.  If we had confined ourselves to the case $\eta=0$ we could not have expected to obtain a deformation-invariant criterion, since the algebraic structure of $QH(M,\omega)_0$ can be surprisingly sensitive to deformations of $\omega$.  For example, in \cite[Section 5]{OT} Ostrover--Tyomkin produce a toric Fano $4$-fold $M$ such that $QH(M,\omega)_0$ fails to be semisimple when $\omega$ is the monotone toric symplectic form but is semisimple for generic toric symplectic forms $\omega$.  Calculations in \cite[Proposition 3.5.5]{BaM} suggest that a similar phenomenon occurs for the manifold obtained by blowing up $\mathbb{C}P^2$ at $5$ or more  points.  However, by changing the question from, ``Is $QH(M,\omega)_0$ semisimple (or field-split)?'' to ``Does there exist $\eta$ such that $QH(M,\omega)_{\eta}$ is semisimple (or field-split)?'' we evade such issues.

Algebraic geometers  have been studying a notion of generically semisimple big quantum homology for some time (see, \emph{e.g.}, \cite{Dub}, \cite{Man}, \cite{Bay}), and we show in Section \ref{cvgsect} that our notion is equivalent to theirs.\footnote{As alluded to in Section \ref{cvgsect}, different authors use different coeffecient rings in the algebraic geometry literature when discussing quantum homology, and in some cases it is assumed that certain power series defining the structure constants converge; however in any event our results show that the resulting notions of generically semisimple big quantum homology are equivalent on their common domains of definition, a fact which seems to be familiar to algebraic geometers.}  This effectively answers a question raised in \cite[Remark 3.2]{EP08} as to whether semisimplicity in the sense studied there (which in our language amounts to the semisimplicity of $QH(M,\omega)_0$) can be deduced from generic semisimplicity in the algebraic geometry sense.  As stated, the answer is evidently negative, since the Ostrover-Tyomkin example has generically semisimple quantum homology but fails to have $QH(M,\omega)_0$ semisimple.  On the other hand, from the standpoint of applications this paper shows that generic semisimplicity suffices for the constructions for which Entov and Polterovich employed the semisimplicity of $QH(M,\omega)_0$.  

\subsection{Summary of the paper, with additional remarks about the deformed Hamiltonian Floer complexes}

In Section \ref{quant1}, we briefly review the definition of the deformed quantum rings $(H_*(M,\omega),\ast_{\eta})$.

Section \ref{hamsect} develops deformed Hamiltonian Floer theory.  Thus to a suitably nondegenerate Hamiltonian $H$ and a deformation parameter $\eta\in \oplus_{i=0}^{n-1}H_{2i}(M;\Lambda_{\omega}^{0})$ we associate an $\mathbb{R}$-filtered Floer complex $(CF(H),\partial^{\eta,H})$.  We also define continuation maps $(CF(H_-),\partial^{\eta,H_-})\to (CF(H_+),\partial^{\eta,H_+})$; PSS-type maps \cite{PSS} $\Phi_{\eta,H}^{PSS}\co CM(f;\Lambda_{\omega})\to CF(H)$ and $\Psi_{\eta,H}\co CF(H)\to CM(f;\Lambda_{\omega})$ where $CM(f;\Lambda_{\omega})$ is the Morse complex; and a pair of pants product $\ast_{\eta}^{Floer}\co CF(H)\otimes CF(K)\to CF(H\lozenge K)$.  The deformations are especially easy to describe when $\eta\in H_{2n-2}(M;\Lambda_{\omega}^{0})$.  For example, whereas matrix elements for the ordinary Floer differential $\partial^{0,H}$ count certain cylinders $u\co \mathbb{R}\times S^1\to M$ with rational weights $\ep(u)$, for $\eta\in H_{2n-2}(M;\Lambda_{\omega}^{0})$ we choose a closed $2$-form $\theta$ Poincar\'e dual to $\eta$ and form the deformed differential $\partial^{\eta,H}$ by instead counting these cylinders with weight $\ep(u)\exp(\int_{\mathbb{R}\times S^1}u^*\theta)$.  Similar modifications are made to obtain deformed versions of the standard continuation and PSS maps and the pair of pants product.  The various standard identities involving these maps (for example the facts that $(\partial^{\eta,H})^{2}=0$ and 
that the other maps are chain maps with respect to $\partial^{\eta,H}$) then follow from the standard arguments together with Stokes' theorem and the fact that the chosen $2$-form $\theta$ is closed.   This approach to ``twisting'' the Hamiltonian Floer complex, though quite simple and useful (it alone would suffice for our applications in many though not all cases---in particular it is enough to yield Calabi quasimorphisms on all point blowups and on all toric Fano manifolds), 
does not seem to have been used on closed symplectic manifolds before.  However, in the context of the symplectic cohomology of Liouville domains the approach has been used by Ritter \cite{Ri}, and it is also similar to the use of "B-fields" in Lagrangian Floer theory (see \cite{Fu}, \cite{Cho}).  Note that in these other contexts (and also in a similarly deformed version of Morse homology, where one integrates closed one-forms over gradient flowlines) the resulting homology does depend on the deformation parameter $\eta$, whereas in our case it does not, perhaps surprisingly at first glance.  This independence can be explained by the facts that, first, the $\eta$-deformed homology will be independent of the Hamiltonian $H$ by a standard continuation argument, and second, for a small, time-independent $H$, all of the gradient flow cylinders can be arranged to degenerate to lines, over which the $2$-form $\theta$ integrates trivially, so that $\theta$ does not affect the differential.  On the other hand $\theta$ does affect the pair-of-pants product, which unavoidably involves genuinely $2$-dimensional objects. 

For more general $\eta\in \oplus_{i=0}^{n-1}H_{2i}(M;\Lambda_{\omega}^{0})$, the differential and the other maps on $CF(H)$ are further deformed by the imposition of incidence conditions corresponding to the higher-codimension components of $\eta$.  Some subtleties arise here because the fact that the moduli spaces have codimension-one boundary leads one to need to use the same chains for the incidence conditions at each marked point (rather than just homologous ones as in Gromov--Witten theory) in order for the maps to satisfy the appropriate identities, which in turn causes potential transversality problems as the marked points approach each other.  However such issues are readily handled by using the same machinery (the fiber product construction for Kuranishi structures) that is used in \cite{FOOO09} to deal with similar (and indeed in some cases more difficult) matters in the construction of Lagrangian Floer theory with bulk deformations.  Our use of these sorts of deformations in Hamiltonian Floer theory appears to be new, though Fukaya--Oh--Ohta--Ono arrived at this technique independently, see \cite{FOOO}.  For the benefit of readers who lack fluency in \cite{FOOO09}, we also provide a construction of the deformed Floer complexes in the semipositive case which does not rely on Kuranishi structures; due to the aforementioned transversality problems for the relevant fiber products this does require some work, most of which is consigned to Appendix \ref{app}.  However, it is still our opinion that the approach based on \cite{FOOO09} is the most appropriate framework for these sorts of constructions.  The reader should see \cite{FOOO} for a more thorough development of this approach, including a relation to Lagrangian Floer theory.

In any event, after constructing these maps, we find that the $\eta$-deformed Hamiltonian Floer theory behaves in much the same way as the undeformed version, except that its multiplicative structure on homology is isomorphic via the appropriate PSS maps to the $\eta$-deformed quantum ring structure.  In particular we can construct spectral invariants in the usual way and prove that they satisfy standard properties as laid out in Section \ref{specsec}.

Section \ref{capsec} contains the applications discussed above which obtain Hofer--Zehnder capacity estimates by means of the behavior of the deformed spectral invariants.  We use arguments depending on the construction of $S^1$-equivariant Kuranishi structures on moduli spaces associated to the Floer differentials and the PSS maps (similarly to what is done in \cite{FO}, \cite{LiuT98} to prove the Arnold conjecture, see also the proof of \cite[Theorem IV]{Oh05} for a case closely related to ours);  the polyfold theory of  \cite{HWZ} should eventually provide a preferable framework for such arguments.  With current technology, as we explain in Remark \ref{nokur}, the use of $S^1$-equivariant Kuranishi structures can be avoided only under a rather strong assumption on the first Chern class of the manifold.  The results of Section \ref{capsec} are not used elsewhere in the paper.

Section \ref{Calabi} establishes Theorem \ref{qmintro}, thus reducing the existence of a Calabi quasimorphism on $\widetilde{Ham}(M,\omega)$ to the statement that, for some $\eta\in \oplus_{i=0}^{n-1}H_{2i}(M;\Lambda_{\omega}^{0})$, the commutative algebra $QH(M,\omega)_{\eta}$ satisfies an algebraic condition.

The last two sections of the body of the paper are much more algebraic, oriented toward the goal of finding equivalent conditions to that identified in Section \ref{Calabi}.  Section \ref{algsect} is pure commutative algebra, devoted mainly to a theorem about sets of equivalent conditions under which the spectral cover map $\S A\to \S R$ associated to a finite $R$-algebra $A$ either has an unramified point in its domain or else has a point in its range over which all points in the fiber are unramified.  The final Section \ref{quant2} connects this to the problem at hand, reformulating big quantum homology as an algebra $\AR$ over a ring $R_M$ (as alluded to earlier, $R_M$ and $\AR$ are, unlike $QH(M,\omega)_{\eta}$, symplectic deformation invariants).  For a deformation parameter $\eta$, the $\eta$-deformed quantum homology is then recovered as the tensor product $\AR\otimes_{R_M}\Lambda_{\omega}$ associated to an appropriate homomorphism $\phi_{\eta}\co R_M\to \Lambda_{\omega}$.  The sets of equivalent conditions identified in Section \ref{algsect} are seen to in turn be equivalent in this context to the existence of an $\eta$ such that $QH(M,\omega)_{\eta}$, respectively, has a field direct summand or is semisimple.  The nice behavior of our algebraic conditions under base change (Proposition \ref{basechange}) enables us to move between coefficient systems with ease, leading to results such as Theorem \ref{algintro}.  Throughout Section \ref{quant2} we discuss both big and small quantum homology;\footnote{In our language, the small quantum homology corresponds to the deformations obtained by allowing $\eta$ to only vary in $H_{2n-2}(M;\Lambda_{\omega}^{0})$; this differs from some references which have small quantum homology correspond just to $\eta=0$.} while this in principle does not lead to greater generality since if our algebraic conditions hold for small quantum homology then they also hold for big quantum homology, the small quantum homology is easier to compute since it only involves $3$-point Gromov--Witten invariants, and so our conditions can be easier to check for small quantum homology. At the end of Section \ref{quant2} we verify that one-point blowups always have generically field-split small (and hence also big) quantum homology, drawing on results of \cite{Ga},\cite{Hu},\cite{Bay} about the Gromov--Witten invariants of blowups. 

Finally, in Appendix \ref{app}, we outline the proof of the result from Section \ref{bigdefsect} which underlies our construction in the semipositive case of the deformed Floer complexes for classes $\eta\in \oplus_{i=0}^{n-1}H_{2i}(M;\Lambda_{\omega}^{0})$ which are not of codimension two.  This is included only to provide a more self-contained treatment for readers who prefer to avoid relying on the more sophisticated constructions from \cite{FOOO09}.

\subsection*{Acknowledgements} I am grateful to the organizers of the MSRI Workshop on Symplectic and Contact Topology and Dynamics (March 2010) and of the Chern Institute Conference on Symplectic Geometry and Physics (May 2010) for the opportunity to present some of this work as it was in preparation.  I also thank K. Fukaya, Y.-G. Oh, and K. Ono for telling me about their related (and independent) joint work with Ohta \cite{FOOO},  and L. Polterovich for interesting comments.   The Stacks Project \cite{Sta} was helpful to me as I prepared Section \ref{algsect}.

\subsection{Some conventions}\label{conv}  There are multitudes of different conventions in the literature used for signs, coefficient rings, and so forth for the various characters in our story; for this reason it seems prudent to collect some of our conventions in one place for the reader's convenience (though the reader will generally be reminded of these as they become relevant).\begin{itemize}\item[(a)] $(M,\omega)$ always denotes a closed symplectic manifold, and $S^1=\mathbb{R}/\mathbb{Z}$.
\item[(b)] A ``Hamiltonian'' is a smooth function $H\co S^1\times M\to\mathbb{R}$.  The Hamiltonian vector field of $H$ is the time dependent vector field $X_H=X_H(t,\cdot)$ defined by $\iota_{X_H(t,\cdot)}=-d(H(t,\cdot))$.  (N.B.: The negative sign is contrary to my usage in previous papers.)
\item[(c)]  The Novikov ring associated to $\omega$ is the generalized formal power series ring \[\Lambda_{\omega}=\left\{\left.\sum_{g\in \Gamma_{\omega}}a_gT^g\right| a_g\in\mathbb{C},(\forall C>0)(\#\{g|a_g\neq 0,g<C\}<\infty)\right\},\] where \[ \Gamma_{\omega}=Im(\langle \omega,\cdot\rangle\co \pi_2(M)\to\mathbb{R}).\]  Thus $\Lambda_{\omega}$ is ``completed upward.'' The subring $\Lambda_{\omega}^{0}$ is defined by \[ \left\{\left.\sum_{g\in\Gamma_{\omega}}a_gT^g\in \Lambda_{\omega}\right|a_g\neq 0 \Rightarrow g\geq 0\right\}.\]
\item[(d)] The Floer complex of a (suitably nondegenerate) Hamiltonian $H$ has underlying module \[ CF(H)=\left\{\left.\sum_{[\gamma,v]\in\tilde{P}(H)}a_{[\gamma,v]}[\gamma,v]\right|(\forall C<0)(\#\{[\gamma,v]|a_{[\gamma,v]}\neq 0,\mathcal{A}_H([\gamma,v])>C\}<\infty)\right\}.\]  Thus $CF(H)$ is, unlike $\Lambda_{\omega}$, ``completed downward.''
\item[(e)] Algebraic structures such as $CF(H)$ and the quantum rings $(H_*(M,\Lambda_{\omega}),\ast_{\eta})$ will be graded by $\mathbb{Z}/2\mathbb{Z}$, not $\mathbb{Z}$.  Correspondingly, there is no ``degree-shifting element'' such as that denoted $q$ in \cite{EP08}.
\item[(f)] Any ring whose name begins with the initials $QH$ (such as $QH(M,\omega)_{\eta}$) corresponds only to the \emph{even-degree} part of the quantum homology (and thus is commutative and, in light of (e) above, is ungraded).  When we wish to include both even and odd degree elements we will directly refer to   $H_*(M,\Lambda_{\omega})$.
\item[(g)] Hamiltonians $H,K\co S^1\times M\to\mathbb{R}$ are composed via the ``concatenation'' operation \[ H\lozenge K(t,m)=\left\{\begin{array}{ll}\chi'(t)H(\chi(t),m) & 0\leq t\leq 1/2 \\ \chi'(t-1/2)K(\chi(t),m) & 1/2\leq t\leq 1\end{array}\right.\] for a suitable smooth monotone surjection $\chi\co [0,1/2]\to [0,1]$.  In particular the time-one maps are related by \[ \phi_{H\lozenge K}^{1}=\phi_{K}^{1}\circ\phi_{H}^{1}.\]
\item[(h)]  A fixed parameter $\eta$ will determine a certain natural number $m$ and cycles $c_1,\ldots,c_m$ such that $c_i$ has dimension $2d(i)$, as explained in Section \ref{hamsect}. The letter $I$ will typically refer to an element $(i_1,\ldots,i_k)$ of $\{1,\ldots,m\}^k$ for some $k\in\mathbb{N}$.  For a previously-chosen set collection of elements $z_i\in \Lambda_{\omega}^{0}$ $(1\leq i\leq m)$ we will write \[ z_I=z_{i_1}\cdots z_{i_k}.\]  Also, \[ \delta(I)=\sum_{j=1}^{k}(2n-2-2d(i_j)) \]  (This latter number is interpreted as the codimension associated to a certain set of incidence conditions associated to $I$.)

\end{itemize}

\section{Quantum homology I} \label{quant1}

To fix notation, we now review the definition of the deformed quantum ring structures $\ast_{\eta}$ on a closed $2n$-dimensional symplectic manifold $(M,\omega)$.  These structures will be viewed from a different perspective in Section \ref{quant2}.

Where \[ \Gamma_{\omega}=\{\langle [\omega],A\rangle|A\in Im(\pi_2(M)\to H_2(M;\mathbb{Z}))\}\leq \mathbb{R},\]  the \emph{Novikov ring} associated to $(M,\omega)$ is \[ \Lambda_{\omega}=\left\{\left.\sum_{i=1}^{\infty}a_iT^{g_i}\right|a_i\in\mathbb{C},g_i\in \Gamma_{\omega},(\forall C>0)\left(\#\left(\{i|a_i\neq 0,g_i<C\}\right)<\infty\right)\right\}\] where $T$ is a formal variable. 
It is not difficult to check that $\Lambda_{\omega}$ is a field (see for instance \cite[Theorem 4.1]{HS}). 

The Novikov ring has a distinguished subring $\Lambda_{\omega}^{0}\leq \Lambda_{\omega}$, defined by \[ \Lambda_{\omega}^{0}=\left\{\left.\sum_i a_iT^{g_i}\in \Lambda_{\omega}\right|(\forall i)(g_i\geq 0)\right\}.\]  An easy exercise shows that there is a well-defined map \[ \exp\co \Lambda_{\omega}^{0}\to\Lambda_{\omega}^{0}\] defined by the usual Taylor series \[ \exp(a)=\sum_{k=0}^{\infty}\frac{a^k}{k!}\] and satisfying the identity \[ \exp(a+b)=\exp a\exp b\]  (on the other hand,  for a more general element $a\in \Lambda_{\omega}$, the formal sum $\sum \frac{a^k}{k!}$ typically does \emph{not} give a well-defined element of $\Lambda_{\omega}$).

For each $\eta\in \oplus_{i=0}^{n-1}H_{2i}(M;\Lambda_{\omega}^{0})$,  we will now define an $\eta$-deformed quantum multiplication $\ast_{\eta}$, which makes $H_*(M,\Lambda_{\omega})$ into a $\Lambda_{\omega}$-algebra.  Our reason for only considering even-dimensional classes $\eta$ is that it will allow us to obtain an algebra which is commutative by restricting $\ast_{\eta}$ to the even-dimensional part of $H_*(M,\Lambda_{\omega})$.  Additionally, orientation-related issues would make it difficult to obtain results such as Corollary \ref{delsquared} below if we did not require $\eta$ to be even-dimensional.

Given a class $A\in H_2(M;\mathbb{Z})$ and classes $a_1,\ldots,a_k\in H_*(M;\Lambda_{\omega}^{0})$ with $k\geq 3$, let \[ \langle a_1,\ldots,a_k\rangle_{0,k,A} \] denote the Gromov--Witten invariant \cite{RT},\cite{FO},\cite{LiT},\cite{R} enumerating (in the appropriate virtual sense) genus zero $J$-holomorphic curves $u\co S^2\to M$ representing $A$ with $k$ freely-varying marked points $z_1,\ldots,z_k$ such that $u(z_i)\in N_i$, where $J$ is a generic almost complex structure compatible with $\omega$ and the $N_i$ are generic cycles representing the $a_i$.   (Recall that this invariant is multilinear in the $a_i$, so may be expressed in terms of invariants in which each inserted homology class is homogeneous.)  Let $\{c_1,\ldots,c_r\}$ be a homogeneous basis for $H_*(M;\mathbb{Q})$, with Poincar\'e dual basis $\{c^1,\ldots,c^r\}$ (i.e. $c_i\cap c^j=\delta_{i}^{j}$ where $\cap$ is the Poincar\'e interection pairing).  The operation $\ast_{\eta}\co H_*(M;\Lambda_{\omega})\otimes H_*(M;\Lambda_{\omega})\to H_*(M;\Lambda_{\omega})$ is then obtained by extending linearly from the formula \begin{equation}\label{firstmult} x\ast_{\eta} y=\sum_{A\in H_2(M;\mathbb{Z})}\sum_{k=0}^{\infty}\frac{1}{k!}\sum_{j=1}^{r}\langle x,y,c_j,\underbrace{\eta,\ldots,\eta}_{k}\rangle_{0,k+3,A}T^{\langle [\omega],A\rangle}c^j.\end{equation}  The reader may verify the fact that the above formula does indeed give a well-defined element of $H_*(M;\Lambda_{\omega})$; in any event this follows from what is done later in Section \ref{quant2}.  As seen in \cite[Section 4]{KM}, the fact that $\ast_{\eta}$ is associative follows from properties of Gromov--Witten invariants (in particular the composition law) that are proven in \cite{RT} and elsewhere.

If we write $\eta=\eta_D+\eta'$ where $\eta_D\in H_{2n-2}(M;\Lambda_{\omega}^{0})$ and $\eta'\in \oplus_{i=0}^{n-2}H_{2i}(M;\Lambda_{\omega}^{0})$, 
the divisor axiom\footnote{Recall that the divisor axiom \cite[Proposition 7.5.6]{MS} asserts that if $c_{k+1}\in H_{2n-2}(M;\mathbb{Q})$ then one has the relation \[ \langle c_1,\ldots,c_k,c_{k+1}\rangle_{0,k+1,A}=(c_{k+1}\cap A)\langle c_1,\ldots,c_k\rangle_{0,k,A}.\]} for Gromov--Witten invariants shows that \[ x\ast_{\eta} y=\sum_{A\in H_2(M;\mathbb{Z})}\sum_{k=0}^{\infty}\frac{1}{k!}\sum_{i=1}^{r}\langle x,y,c_j,\underbrace{\eta',\ldots,\eta'}_{k}\rangle_{0,k+3,A}\exp(\eta_D\cap A)T^{\langle [\omega],A\rangle}c^j;\] in particular if $\eta\in H_{2n-2}(M;\Lambda_{\omega}^{0})$ we simply have \[ x\ast_{\eta} y=\sum_{A\in H_2(M;\mathbb{Z})}\sum_{j=1}^{r}\langle x,y,c_j\rangle_{0,3,A}\exp(\eta\cap A)T^{\langle [\omega],A\rangle}c^j.\]

\section{Hamiltonian Floer theory and its deformations}\label{hamsect}

Let us fix our notations and conventions for Hamiltonian Floer theory.  Let $H\co S^1 \times M\to \mathbb{R}$ be a smooth function (here and below the circle $S^1$ will be identified with $\mathbb{R}/\mathbb{Z}$).  Our convention will be that the (time-dependent) Hamiltonian vector field $X_H$ associated to $H$ is given by \[ \iota_{X_H(t,\cdot)}\omega=-d(H(t\cdot)).\]   This is consistent with the convention generally used in the papers of Entov-Polterovich, but opposite to that used by Oh and by the present author in earlier papers; various signs appearing below will accordingly differ from those in corresponding results from those papers.

The Hamiltonian isotopy generated by the vector field $X_H$ will be denoted $\{\phi_{H}^{t}\}_{t\in\mathbb{R}}$ (with $\phi_{H}^{0}$ equal to the identity).  The path $\{\phi_{H}^{t}\}_{0\leq t\leq 1}$ determines an element in $\widetilde{Ham}(M,\omega)$, denoted $\tilde{\phi}_H$.

Let $\mathcal{L}_0M$ be the space of contractible loops in $M$, and \[ \widetilde{\mathcal{L}_0M}=\frac{\{(\gamma,v)|\gamma\in \mathcal{L}_0M,v\co D^2\to M,v|_{\partial D^2}=\gamma\}}{(\gamma,v)\sim (\gamma',v')\mbox{ iff }\gamma=\gamma'\mbox{ and }\int_{D^2}v^*\omega=\int_{D^2}v'^*\omega}.\] 

(In this paper the Hamiltonian Floer complex will be treated as just $\mathbb{Z}_2$-graded; since we do not impose a $\mathbb{Z}$-grading, we will not incorporate the first Chern class of $M$ into the definition of  $\widetilde{\mathcal{L}_0M}$ as is often done).  

  Define \[ P(H)=\{\gamma\in\mathcal{L}_0M|\dot{\gamma}(t)=X_H(t,\gamma(t))\}\] and \[ \tilde{P}(H)=\{[\gamma,v]\in \widetilde{\mathcal{L}_0M}|\gamma\in P(H)\}.\]  

Thus $\tilde{P}(H)$ is the set of critical points of the action functional \[ \mathcal{A}_H([\gamma,v])=-\int_{D^2}v^*\omega+\int_{0}^{1}H(t,\gamma(t))dt.\]

Assume that $H$ is nondegenerate in the usual sense that, for each $p\in Fix(\phi_{H}^{1})$,  the linearization $d_p\phi_{H}^{1}\co T_pM\to T_pM$ does not have $1$ as an eigenvalue. We then have a map \[ \mu\co \tilde{P}(H)\to \{\pm 1\}\] defined by \[ \mu([\gamma,v])=sign \det (Id-d_p\phi_{H}^{1}).\]

As an abelian group, the Floer complex of $H$ is then defined by \[ CF(H)=\left\{\left.\sum_{[\gamma,v]\in\tilde{P}(H)}a_{[\gamma,v]}[\gamma,v]\right|a_{[\gamma,v]}\in\mathbb{C},(\forall C\in\mathbb{R})(\#\{[\gamma,v]|a_{[\gamma,v]}\neq 0,\mathcal{A}_H([\gamma,v])>C\}<\infty)\right\}.\]

We thus have \[ CF(H)=CF_{ev}(H)\oplus CF_{odd}(H) \] where\[ CF_{ev}(H)=\left\{\left.\sum_{[\gamma,v]\in\tilde{P}(H)}a_{[\gamma,v]}[\gamma,v]\in CF(H)\right|a_{[\gamma,v]}\neq 0\Rightarrow \mu([\gamma,v])=+1\right\} \] and \[ CF_{odd}(H)=\left\{\left.\sum_{[\gamma,v]\in\tilde{P}(H)}a_{[\gamma,v]}[\gamma,v]\in CF(H)\right|a_{[\gamma,v]}\neq 0\Rightarrow \mu([\gamma,v])=-1\right\}. \]

 For any $g\in \Gamma_{\omega}$ choose an arbitrary $A_g\in \pi_2(M)$ such that $\langle[\omega],A_g\rangle=g$.  Then $CF(H)$ has the structure of a $\Lambda_{\omega}$ module, via the action \[ \left(\sum a_iT^{g_i}\right)\cdot \left(\sum b_{[\gamma,v]}[\gamma,v]\right)=\sum a_ib_{[\gamma,v]}[\gamma,v\#A_g].\]  Note that $\Lambda_{\omega}$ and $CF(H)$ are completed in the ``opposite directions''; this is an artifact of the negative sign in the formula for $\mathcal{A}_H$.  

\subsection{Homotopy classes of cylinders}
For two nondegenerate Hamiltonians $H_{\pm}\co S^1\times M\to\mathbb{R}$ and  $\gamma^-,\gamma^+\in P(H_{\pm})$ define $\pi_2(\gamma^-,\gamma^+)$ to be the set of path components of the space of $C^0$ maps $u\co [-\infty,\infty]\times S^1\to M$ obeying $u(\{\pm\infty\} \times S^1)=\gamma^{\pm}$.  

Let $u\co \mathbb{R}\times S^1\to M$ be a $C^1$ map such that (where $s$ is the $\mathbb{R}$-variable on $\mathbb{R}\times S^1$) $u(s,\cdot)\to \gamma^{\pm}$ uniformly; $\frac{\partial u}{\partial s}(s,\cdot)\to 0$ in $L^2$; and $\frac{\partial u}{\partial t}(s,\cdot)\to X_{H_{\pm}}$ in $L^2$ as $s\to \pm\infty$. Then $u$ extends continuously to a map $u\co [-\infty,\infty]\times S^1\to M$ with $u(\{\pm\infty\} \times S^1)=\gamma^{\pm}$, so that $u$ represents some class $C\in \pi_2(\gamma^-,\gamma^+)$. 

Write \[\Omega^{2}_{cl}(M;\Lambda_{\omega}^{0})=\left\{\left.\sum_{i=1}^{\infty}\theta_iT^{g_i}\right| \theta_i\in\Omega^2(M;\mathbb{C}),\,d\theta_i=0,\,0\leq g_i\nearrow\infty,\,g_i\in\Gamma_{\omega}\right\}.\] 
It is a straightforward consequence of Stokes' theorem that, if $\theta\in \Omega^2(M;\mathbb{C})$ is any \emph{closed} complexified $2$-form on $M$, the quantity $\int_{\mathbb{R}\times S^1}u^*\theta$ (which is easily seen to be finite by our asymptotic assumptions) depends only on the homotopy class $C$ and not on the representative $u$.  Accordingly, for $C\in \pi_2(\gamma^-,\gamma^+)$ and for $\theta\in \Omega^2(M;\mathbb{C})$ with $d\theta=0$ we denote \[ \int_{C}\theta:=\int_{\mathbb{R}\times S^1}u^*\theta\] where $u$ is any representative of $C$ as described above.
Consequently, for any $\theta\in \Omega^2_{cl}(M,\Lambda_{\omega}^{0})$ and any $\gamma^{\pm}\in P(H)$, $C\in \pi_2(\gamma^-,\gamma^+)$, we have a well-defined value \[ \int_C\theta=\sum_{i=1}^{\infty}\left(\int_C\theta_i\right)T^{g_i}\in \Lambda_{\omega}^{0}.\]

Given three periodic orbits $\gamma_0,\gamma_1,\gamma_2$ we have a ``concatenation'' map \begin{align*} \pi_2(\gamma_0,\gamma_1)\times \pi_2(\gamma_1,\gamma_2)&\to \pi_2(\gamma_0,\gamma_2)   
\\ (C_-,C_+)&\mapsto C_-\#C_+\end{align*} where $C_-\#C_+$ is the equivalence class of a map $w\co [-\infty,-\infty]\times S^1\to M$ defined (for the sake of definiteness) as follows.  Let $u,v\co [-\infty,\infty]\times S^1\to M$ be arbitrary representatives of, respectively, $C_-$ and $C_+$, and define $w\co [-\infty,\infty]\times S^1\to M$ by \[ w(s,t)=\left\{\begin{array}{ll} u(-\log(-s),t) & s\leq 0 \\ v(\log(s),t) & s\geq 0\end{array}\right.\] 
(here we've extended the natural logarithm to a continuous map $\log\co [0,\infty]\to [-\infty,\infty]$; the definition is consistent at $s=0$ since by assumption $u(\infty,\cdot)=v(-\infty,\cdot)=\gamma_1$). 


Evidently, if $\theta\in \Omega^2(M;\Lambda_{\omega}^{0})$, $C_-\in \pi_2(\gamma_0,\gamma_1),C_+\in\pi_2(\gamma_1,\gamma_2)$, then \begin{equation}\label{concat} \int_{C_-\#C_+}\theta=\int_{C_-}\theta+\int_{C_+}\theta.\end{equation} 

Another important map on $\pi_2(\gamma^-,\gamma^+)$ is the \emph{Maslov map} \[ \bar{\mu}\co \pi_2(\gamma^-,\gamma^+)\to\mathbb{Z}.\]  To define it, let $C\in \pi_2(\gamma^-,\gamma^+)$ and let $u\co [-\infty,\infty]\times S^1$ be an arbitrary representative of $C$.  Since $\gamma^-$ is by assumption contractible, choose any disc $v\co D^2\to M$ with $v|_{\partial D^2}=\gamma^-$.  Glue the negative end of $u$ to the boundary of $v$ along $\gamma^-$ to form a new disc $u\#v\co D^2\to M$, now with $u\#v|_{\partial D^2}=\gamma^+$. Now define \[ \bar{\mu}(C)=\mu_{CZ}(\gamma^+,v\#u)-\mu_{CZ}(\gamma^-,v).\]  Here the Conley--Zehnder index  $\mu_{CZ}$ is defined as usual (see \cite[Section 2]{Sal}): for $\gamma\in P(H)$ and $v\co D^2\to M$ with $v|_{\partial D^2}=\gamma$, symplectically trivialize $v^*TM$ over $D^2$, and let $\mu_{CZ}(\gamma,v)$ be the Conley--Zehnder index (in the sense of \cite[Remark 5.4]{RS}) of the path of symplectic matrices which represent the linearizations $d_p\phi_{H}^{t}$ in terms of this trivialization.  The quantity $\bar{\mu}(C)$ is easily seen to be independent of the various choices involved and to obey \[ \bar{\mu}(C_-\#C_+)=\bar{\mu}(C_-)+\bar{\mu}(C_+)\quad (C_-\in \pi_2(\gamma_0,\gamma_1),C_+\in\pi_2(\gamma_1,\gamma_2)).\]  

\subsection{Differentials}

An $S^1$-family $J=\{J_t\}_{t\in S^1}$ of $\omega$-compatible almost complex structures on $M$ induces an $S^1$-family of metrics and hence an $L^2$ metric on $\widetilde{\mathcal{L}_0M}$ in a standard way.  The negative gradient flow equation for the action functional $\mathcal{A}_H$ with respect to this metric asks for a map $u\co \mathbb{R}\times S^1\to M$ obeying \begin{equation}\label{floereqn} \delbar_{J,H}u:=\frac{\partial u}{\partial s}+J(t,u(s,t))\left(\frac{\partial u}{\partial t}-X_H(t,u(s,t))\right)=0.\end{equation}  For a fixed $\gamma^-,\gamma^+\in P(H)$ and $C\in \pi_2(\gamma^-,\gamma^+)$ write\footnote{to avoid clutter we are omitting $J$ and $H$ from the notation for $\tilde{\mathcal{M}}(\gamma^-,\gamma^+,C)$; we hope this will not cause confusion} \[ \tilde{\mathcal{M}}(\gamma^-,\gamma^+,C)=\left\{u\co \mathbb{R}\times S^1\to M\left|\begin{array}{c}\delbar_{J,H}u=0,\int_{\mathbb{R}\times S^1}\left|\frac{\partial u}{\partial s}\right|^2<\infty,\\u(s,\cdot)\to\gamma^{\pm}\mbox{ as }s\to \pm \infty, \\  \, [u]=C\in\pi_2(\gamma^-,\gamma^+)\end{array}\right.\right\}.\]
Where $\mathbb{R}$ acts by $s$-translation, let \[ \mathcal{M}(\gamma^-,\gamma^+,C)=\tilde{\mathcal{M}}(\gamma^-,\gamma^+,C)/\mathbb{R}.\]

This latter space has a standard Gromov compactification $C\mathcal{M}(\gamma^-,\gamma^+,C)$ (by broken trajectories and sphere bubbles), and in \cite[Theorem 19.14]{FO} this compactification $C\mathcal{M}(\gamma^-,\gamma^+,C)$ is endowed with an oriented Kuranishi structure with corners, of dimension $\bar{\mu}(C)-1$; moreover as $C$ varies the various Kuranishi structures are compatible with gluing maps in an appropriate sense made precise in \cite[Addendum 19.16]{FO} (see also \cite[Section 4]{LiuT98} for an alternative construction).   In particular, if $\bar{\mu}(C)=1$ the zero set of a generic multisection $\frak{s}_C$ associated to this Kuranishi structure consists of finitely many points each with a corresponding rational multiplicity, and the sum of these rational mutliplicities gives a value which we denote \[ |\frak{s}_{C}^{-1}(0)|.\]  

(See \cite[Appendix A]{FOOO09} for a general introduction to Kuranishi structures.  Very briefly, a Kuranishi structure on a compact space $Z$ identifies a neighborhood $U_p$ of each $p\in Z$ with a quotient by a finite group $\Gamma_p$ of the zero locus of some $\Gamma_p$-equivariant smooth map $s_p\co V_p\to E_p$ where the smooth manifold with corners $V_p$ and the finite-dimensional vector space $E_p$ both carry $\Gamma_p$-actions.  The multisection $\frak{s}_C$ amounts to the  data of suitably compatible multi-valued perturbations $s'_{p_i}$ of $s_{p_i}$  for certain specially-chosen  $p_i$ (as in \cite[Lemma A1.11]{FOOO09}) such that the $U_{p_i}$ cover $Z$, with each $s'_{p_i}$ transverse to zero (so that $s'_{p_i}$ vanishes at just finitely many points when the virtual dimension is zero).  \cite[Definition A1.27]{FOOO09} then prescribes multiplicities at each point in the vanishing locus, and what we denote by    $|\frak{s}_{C}^{-1}(0)|$ is the sum of these multiplicities.)

Readers who do not like Kuranishi structures but are willing to restrict themselves to the case where $(M,\omega)$ is strongly semipositive (\emph{i.e.}, for any $A\in \pi_2(M)$ with $2-n\leq \langle c_1(TM),A\rangle< 0$ we have $\int_{A}\omega\leq 0$) can instead follow the constructions of \cite{HS}.  Provided that the pair $(J,H)$  belongs to a suitable residual set the space  $\mathcal{M}(\gamma^-,\gamma^+,C)$ will be a an oriented manifold of dimension $\barmu(C)-1$, consisting in the case that $\barmu(C)=1$ of finitely many points; what we denote by $|\frak{s}_{C}^{-1}(0)|$ should then be interpreted as the signed number of points in $\mathcal{M}(\gamma^-,\gamma^+,C)$. (The section $\frak{s}_C$ can in this context be regarded as the restriction of the operator $\delbar_{J,H}$ to the space of cylinders representing $C\in \pi_2(\gamma^-,\gamma^+)$.)  

The standard Floer differential is the operator $\partial_0\co CF(H)\to CF(H)$ obtained by extending linearly from the formula \[ \partial_0[\gamma^-,v^-]=\sum_{\gamma^+\in P(H)}\sum_{\substack{ C\in \pi_2(\gamma^-,\gamma^+),\\  \bar{\mu}(C)=1}}|\frak{s}_{C}^{-1}(0)|[\gamma^+,v^-\#C].\]

(Here and below the notation $[\gamma^+,v^-\#C]$ is to be interpreted in the obvious way using concatenation; of course since the equivalence relation defining elements $[\gamma,v]$ of $\tilde{P}(H)$ is much weaker than homotopy the assignment $C\mapsto [\gamma^+,v^-\#C]$ will be many-to-one.)

\subsubsection{Small deformations} 
We consider additionally a whole family of Floer differentials, $\partial_{\theta}$, where $\theta$ takes values in the  set $\Omega^2_{cl}(M;\Lambda_{\omega}^{0})$ of closed Novikov-ring valued $2$-forms. For  $\theta=\sum_i\theta_iT^{g_i}\in \Omega^2_{cl}(M,\Lambda_{\omega}^{0})$ and any $\gamma^{\pm}\in P(H)$, $C\in \pi_2(\gamma^-,\gamma^+)$, we have a well-defined value \[ \int_C\theta=\sum_{i=1}^{\infty}\left(\int_C\theta_i\right)T^{g_i}\in \Lambda_{\omega}^{0}.\]

For any such $\theta$, we then set \[ \partial_{\theta}[\gamma^-,v^-]=\sum_{\gamma^+\in P(H)}\sum_{\substack{ C\in \pi_2(\gamma^-,\gamma^+),\\  \bar{\mu}(C)=1}}|\frak{s}_{C}^{-1}(0)|\exp\left(\int_{C}\theta\right)[\gamma^+,v^-\#C].\]

Said differently, for each cylinder $u$ counted by the Floer boundary operator, if $u$ is ordinarily (with respect to the standard boundary operator $\partial_0$) counted with weight $\ep(u)$, we instead count $u$ with weight $\ep(u)\exp{\int_{\mathbb{R}\times S^1}u^*\theta}\in \Lambda_{\omega}^{0}$. 

\begin{prop} If $\theta\in\Omega^{2}_{cl}(M,\Lambda_{\omega}^{0})$, then $\partial_{\theta}^{2}=0$.
\end{prop}

\begin{proof} This follows from the standard argument (together with Stokes' theorem as manifested in the formula (\ref{concat})).  Indeed, for $[\gamma^-,v^-]\in \tilde{P}(H)$, $\partial_{\theta}^{2}[\gamma^-,v^-]$ is a formal sum of generators $[\gamma^+,v^+]$ where $\gamma^+\in P(H)$ and $\mu_{CZ}(\gamma^+,v^+)-\mu_{CZ}(\gamma^-,v^-)=2$.   The coefficient on $[\gamma^+,v^+]$ is equal to \[ \sum_{\substack{C\in \pi_2(\gamma^-,\gamma^+),\\ \int_C\omega=\int_{D^2}v^{+*}\omega-\int_{D^2}v^{-*}\omega  }}\sum_{\gamma\in P(H)}\sum_{\substack{C_-\in \pi_2(\gamma^-,\gamma),\\ C_+\in \pi_2(\gamma,\gamma^+),\\ C_-\#C_+=C,\\ \bar{\mu}(C_{\pm})=1}}|\frak{s}_{C_-}^{-1}(0)||\frak{s}_{C_+}^{-1}(0)|\exp\left(\int_{C_-}\theta\right)\exp\left(\int_{C_+}\theta\right),\] which in turn is, by (\ref{concat}), equal to \begin{equation}\label{dts} \sum_{\substack{C\in \pi_2(\gamma^-,\gamma^+),\\ \int_C\omega=\int_{D^2}v^{+*}\omega-\int_{D^2}v^{-*}\omega   }}\exp\left(\int_{C}\theta\right)\sum_{\gamma\in P(H)}\sum_{\substack{C_-\in \pi_2(\gamma^-,\gamma),\\ C_+\in \pi_2(\gamma,\gamma^+),\\ C_-\#C_+=C,\\ \bar{\mu}(C_{\pm})=1}}|\frak{s}_{C_-}^{-1}(0)||\frak{s}_{C_+}^{-1}(0)|.\end{equation}  
But as in \cite[Lemma 20.2]{FO}, for each $C\in \pi_2(\gamma^-,\gamma^+)$ appearing in the sum, the multisection $\frak{s}_C$ associated to the $1$-dimensional oriented  Kuranishi structure on $C\mathcal{M}(\gamma^-,\gamma^+,C)$ has \[ \partial \frak{s}_{C}^{-1}(0)=\bigcup_{\substack{\gamma,C_-,C_+:\\C_-\#C_+=C}}\frak{s}_{C_-}^{-1}(0)\times \frak{s}_{C_+}^{-1}(0).\]  So since the sum of the multiplicities of the points of $\partial \frak{s}_{C}^{-1}(0)$ is zero it follows that the sum in (\ref{dts}) is zero.

(In the semipositive case where one does not use Kuranishi structures the same argument works, as in \cite[Theorem 5.1]{HS}: for $C\in \pi_2(\gamma^-,\gamma^+)$ the space $\mathcal{M}(\gamma^-,\gamma^+,C)$ has a compactification which is an oriented manifold with boundary equal to \[ \bigcup_{\substack{\gamma,C_-,C_+:\\C_-\#C_+=C}}\mathcal{M}(\gamma^-,\gamma,C_-)\times \mathcal{M}(\gamma,\gamma^+,C_+),\] which again causes the sum in (\ref{dts}) to be zero.)
\end{proof}

All of the boundary operators $\partial_{\theta}$ as $\theta$ varies are different; however one can show\footnote{using arguments similar to those in the proof of \cite[Lemma 3.8]{U09}---we omit the proof since we will not need to use the statement} that if $\theta,\theta'$ are cohomologous then the complexes built from $\partial_{\theta},\partial_{\theta'}$ are isomorphic in the category of $\mathbb{R}$-filtered chain complexes.  Thus we have introduced a way of deforming the standard Floer complex $(CF(H),\partial_0)$ by a class in $H^2(M;\Lambda_{\omega}^{0})$, or equivalently (by Poincar\'e duality) by a class in $H_{2n-2}(M;\Lambda_{\omega}^{0})$.  In what follows the boundary operator denoted $\partial_{\theta}$ above will instead be called $\partial^{PD[\theta],H}$.  Presently, these deformed boundary operators will be generalized to the case where $PD[\theta]$ is replaced by a more general even-dimensional homology class.

\subsubsection{Big deformations} \label{bigdefsect}
 This generalization entails counting Floer trajectories which obey certain incidence conditions, but requires more subtle technical arguments than one might initially imagine.   Conceptually, the relevant technical hurdles form a proper subset of the technical hurdles overcome in the construction of  $A_{\infty}$-algebras associated to appropriate Lagrangian submanifolds in \cite{FOOO09}, so we could just appeal to that work.  In the interests of not limiting the audience of this part of the paper to those with a working knowledge of \cite{FOOO09}, though, we provide a more self-contained description which in the case where $(M,\omega)$ is strongly semipositive does not make any use of Kuranishi structures.  If $(M,\omega)$ is not strongly semipositive then one of course will need Kuranishi structures or something similar to deal with multiple covers of spheres of negative Chern number, but at least in principle one could use constructions just at the level of \cite{FO} rather than \cite{FOOO09}.

Let us first give an indication of the basic strategy, with some explanations of where the novel technical issues arise.  Suppose for simplicity that we wish to deform the boundary operator by a class $\eta=f_*[N]$ where $f\co N\to M$ is a smooth map of a closed, $2d$-dimensional manifold $N$, where $2d\leq 2n-4$.  The plan is to construct \[ \partial^{\eta,H}=\sum_{k=0}^{\infty}\frac{1}{k!}\partial^{\eta,H,k},\] where, as a first (inaccurate) approximation, each $\partial^{\eta,k}$ counts (modulo reparametrization) solutions $u$ to (\ref{floereqn}) having $k$ freely-varying marked points $z_1,\ldots,z_k$ with incidence conditions $u(z_i)\in f(N)$: in fiber product notation, we wish to count elements of the spaces \begin{equation}\label{motfp} \left((\tilde{\mathcal{M}}(\gamma^-,\gamma^+;C)\times (\mathbb{R}\times S^1)^k)_{(ev_1,\ldots,ev_k)}\times_{(f,\ldots,f)}N^k\right)/\mathbb{R}\end{equation}  where $ev_i$ denotes evaluation at the $i$th marked point: $ev_i(u,z_1,\ldots,z_k)=u(z_i)$. The astute reader will already have noticed a transversality problem in (\ref{motfp}) occurring when $k\geq 2$ and when two of the marked points $z_i$ collide.  One way of dealing with such transversality problems is to use the fiber product construction for Kuranishi structures, as described in \cite[Section A1.2]{FOOO09}, but we will also describe a more direct approach.  

Note that in Gromov--Witten theory, one can avoid these sorts of transversality problems by, instead of requiring that $u(z_i)\in f(N)$ for all $i$, requiring that $u(z_i)\in f_i(N)$ where the various $f_i(N)$ are in general position with respect to each other.  However in the Floer-theoretic context this strategy will not work, essentially because Floer theory involves moduli spaces have that boundary of codimension $1$ rather than $2$, so that varying the constraint cycle $N$ in a one-parameter family will typically cause the operations to change. If we apply the typical gluing argument to a chain of two Floer trajectories each obeying one incidence constraint on the cycle $N$, then we are led to considering Floer trajectories with two incidence conditions \emph{both on $N$}, which is \emph{not} equivalent to considering Floer trajectories with incidence conditions on different cycles homologous to $N$.  
This suggests that, in order to ensure that $\partial^{\eta,H}\circ\partial^{\eta,H}=0$, the operator $\partial^{\eta,H,2}$ should enumerate elements of a space such as the $k=2$ version of (\ref{motfp}); however, as noted earlier, this space suffers from transversality problems.

Thus, in defining the operators $\partial^{\eta,H,k}$, there are two competing issues to address: on the one hand we wish the $\partial^{\eta,H,k}$ to be suitably compatible with the gluing arguments that are typically used to show that Floer differentials have square zero, while on the other hand we need to maintain transversality as marked points approach each other.  While the preceding two paragraphs suggest that these goals are in tension, it is reasonable to expect that the problem can be solved: the former issue relates to behavior when marked points are far away from each other, while the latter relates to behavior when marked points are near to each other, so one can hope to resolve both issues simultaneously.  We now set about fulfilling this hope.

For any positive integer $k$, we have an evaluation map \[ \mathbf{ev}_k\co C^0(\R\times S^1,M)\times (\R\times S^1)^k\to M^k\] 
which figures in (\ref{motfp}).  Our basic strategy is to (simultaneously for all $k$) modify these maps $\ev_k$ to maps $\tilde{\ev}_k$ with respect to which the fiber products as in (\ref{motfp}) can be made transverse, while preserving the appropriate compatibility under gluing of cylinders.  We construct $\tilde{\ev}_{k}$ along the following lines.  
For a smooth positive function $\beta\co \R\to (0,\infty)$ with $\beta(s)\to 0$ as $s\to\pm\infty$, for $1\leq i\leq k$ define $\tau_{\beta,i}\co(\R\times S^1)^k\to \R$ by \begin{equation} \label{taubetadef} \tau_{\beta,i}((s_1,t_1),\ldots,(s_k,t_k))=\sum_{j=1}^{i-1}\beta(s_i-s_j).\end{equation}

Then where $\{\psi^{\tau}_{V}\}_{\tau\in\R}$ denotes the flow of a suitable vector field $V$, we will set, for a map $u\co \R\times S^1\to M$ and an element $\vec{z}=(z_1,\ldots,z_k)=((s_1,t_1),\ldots,(s_k,t_k))\in (\R\times S^1)^k$, \begin{equation}\label{evscheme} \tilde{\ev}_k\left(u,\vec{z}\right)=\left(u(z_1),\psi^{\tau_{\beta,2}(\vec{z})}_{V}(u(z_2)),\ldots, \psi^{\tau_{\beta,k}(\vec{z})}_{V}(u(z_k))\right).\end{equation}  Thus the evaluation map at the $i$th marked point is modified by flowing along the vector field $V$ by an amount dependent on the locations of all previous marked points relative to the $i$th one. (Since $\tau_{\beta,1}(\vec{z})=0$ we have omitted it from the notation above.) 

\begin{remark}\label{mainrmk}  Here is an intuitive description of why this approach can be expected to address the issues described above.  Consider first the issue mentioned earlier of non-transversality as marked points approach each other.  Notice that the definition (\ref{taubetadef}) of the functions $\tau_{\beta,i}$ ensures that, if $\vec{z}\in (\R\times S^1)^k$ has $z_{i_1}=z_{i_2}$ where $i_1<i_2$, then $\tau_{\beta,i_2}(\vec{z})>\tau_{\beta,i_1}(\vec{z})$.  Thus whereas the original fiber product (\ref{motfp}) would have had an excess-dimensional stratum in which the (equal) marked points $z_1$ and $z_2$ are both mapped to the same point on $f(N)$, by replacing $\ev_k$ with $\tilde{\ev}_k$ we force $z_1$ and $z_2$ to be connected to different points on $f(N)$ by prescribed-length flowlines of the vector field $V$ as $z_1$ and $z_2$ approach each other (at least as long as they are not mapped to zeros of  $V$, as we will arrange), and this will (for generic choices of the auxiliary data) prevent the diagonals in $(\R\times S^1)^k$ from contributing problematic additional strata.

At the same time,  these perturbations of the evaluation maps are compatible with the gluing maps that arise in the standard proofs that the Floer differential has square zero.  This is essentially because of the fact that $\beta(s)\to 0$ as $s\to \pm\infty$, so that distinct marked points on the domain do not influence each other in the limit as the distance between them approaches infinity.
  Somewhat more precisely, a typical end of one of these moduli spaces involves a sequence of Floer trajectories $u^{(n)}\co \R\times S^1\to M$ splitting into two trajectories $u_{\pm}\co \R\times S^1\to M$, with the marked points
 $z_{1}^{(n)}=(s_{1}^{(n)},t_{1}^{(n)}),\ldots,z_{k}^{(n)}=(s_{k}^{(n)},t_{k}^{(n)})$  distributing themselves  among the two components; write
$i_{1}^{-}<\ldots<i_{l_-}^{-}$ for the indices in $\{1,\ldots,k\}$ of the marked points which limit to points on the domain of
 $u_-$, and $i_{1}^{+}<\ldots <i_{l_+}^{+}$ for those indices corresponding to marked points which limit to points on the domain of
  $u_+$. Thus for some $T_{+}^{(n)},T^{(n)}_{-}\in \R$ with $T_{+}^{(n)}-T^{(n)}_{-}\to\infty$ there will be \[ z_{i_{1}^{-}},\ldots,z_{i_{l_-}^{-}},z_{i_{1}^{+}},\ldots,z_{i_{l_+}^{+}}\in\R\times S^1\] such that, where for $T\in\R$ $\sigma_T\co \R\times S^1\to \R\times S^1$ denotes the map $(s,t)\mapsto (s+T,t)$, we have $u^{(n)}\circ \sigma_{T_{\pm}^{n}}\to u_{\pm}$ uniformly on compact subsets, and $\sigma_{T_{\pm}^{(n)}}^{-1}(z_{i_{j}^{\pm}}^{(n)})\to z_{i_{j}^{\pm}}$.  

Note in particular that $s_{{i_j}^{+}}^{(n)}- s_{{i_l}^{-}}^{(n)}\to\infty$ as $n\to\infty$ for each $j,l$.  Consequently, in view of the definition (\ref{taubetadef}), we will have for each $j=1,\ldots,l_{\pm}$, \[ \tau_{\beta,j}(z_{i_{1}^{\pm}},\ldots,z_{i_{l_{\pm}}^{\pm}})=\lim_{n\to\infty}\tau_{\beta,i_{j}^{\pm}}(z_{1}^{(n)},\ldots,z_{k}^{(n)}).\]  Again, this statement can roughly be interpreted as meaning that our perturbations to the incidence conditions behave compatibly under gluing of broken trajectories, as is required for the usual statement that one-dimensional spaces of unparametrized Floer trajectories have boundaries given by unions of products of zero-dimensional spaces of trajectories and that therefore the boundary operator squares to zero.

\end{remark}

We now begin to make these statements more precise.  
By \cite[Th\'eor\`eme II.29]{Th}, we may fix once and for all a basis for $\oplus_{j=0}^{n-2}H_{2j}(M;\mathbb{Q})$, say with $m$ elements, each member of which may be represented by an embedding $f_i\co N_i\to M$ where $N_i$ is a closed smooth oriented manifold, say of dimension $2d(i)$ ($1\leq i\leq m$).   We should emphasize that we consider the $f_i$  to be chosen at the very outset of the construction (before, for instance, we have chosen a Hamiltonian or an almost complex structure), and we will not really make a full investigation of the extent to which our constructions are independent of $f_i$ up to appropriate isomorphism. At least at the level of the homology ring, our main results do imply such independence; however we do not address (and do not need to address) whether a different choice of the  $f_i$ would give rise to different deformed spectral invariants.\footnote{In the version of this construction that appears in \cite{FOOO}, a result stating that the deformed spectral invariants depend only on the homology class $\eta$ appears as Theorem 7.7(2).}

\begin{dfn}\label{sndeg} A Hamiltonian $H\co S^1\times M\to\R$ is \emph{strongly nondegenerate} if, for each $p\in Fix(\phi_{H}^{1})$, \begin{itemize} \item the linearization $d_p\phi_{H}^{1}\co T_p M\to T_pM$ does not have 1 as an eigenvalue, and \item the orbit $\{\phi_{H}^{t}(p)|t\in [0,1]\}$ is disjoint from each submanifold $f_i(N_i)$.\end{itemize}
\end{dfn}

A standard argument shows that the space of strongly nondegenerate $H$ is residual in the $C^l$-topology for any $2\leq l\leq \infty$. 

 Let $\mathcal{J}^l$ denote the space of $S^1$-families of $\omega$-compatible almost complex structures of class $C^l$. Let $g\co M\to\R$ be a Morse function $g$ all of whose critical points are disjoint from the various $f_i(N_i)$.  For a positive  integer $l$ let $\mathcal{V}^l$ denote the space of $C^l$ gradient-like vector fields $V$ for $g$---in other words, those $V$ for which there are coordinate charts near each critical point of $g$ in terms of which $V$ vanishes linearly, while $dg(V)< 0$ on the complement of the critical points of $g$.  Thus if $V\in\mathcal{V}^l$ then $V$ has no nontrivial periodic orbits, and zero locus of $V$ is precisely the set $Crit(g)$ of critical points of $g$, which in particular is disjoint from each $f_i(N_i)$.  For $V\in\mathcal{V}^l$ let $\{\psi^{\tau}_{V}\}_{\tau\in\R}$ denote the flow generated by $V$.
  
 Meanwhile fix a strongly nondegenerate Hamiltonian $H_0$, and let $\mathcal{H}^l$ denote a small ball around $H_0$ in the space of $C^{l+1}$ Hamiltonians $H\co S^1\times M\to\R$ which coincide to order two with $H_0$ at the $1$-periodic orbits of $X_{H_0}$.  Finally  choose a Banach space $(\tilde{\mathcal{B}},\|\cdot\|)$ of functions $\mathbb{R}\to \R$ which is dense in $L^2$ and all of whose elements belong to $C^{\infty}(\R;\R)$ (for definiteness let us follow \cite[Section 5]{Fl} and use for $\tilde{\mathcal{B}}$ the completion of $C^{\infty}_{0}(\R;\R)$ with respect to a norm of the form $\sum_{k=0}^{\infty}\ep_k\|\cdot\|_{C_k}$ for sufficiently rapidly decreasing $\ep_k$, with $\ep_0=1$) and let $\mathcal{B}$ denote the space of functions of the form $\beta(s)=e^{-s^2}(1+f(s))$ for $f\in \tilde{\mathcal{B}}$ with $\|f\|<1$.  Thus $\mathcal{B}$ has an obvious identification with an open set in a Banach space and so is a Banach manifold; moreover all of its elements are smooth, positive functions $\beta\co \R\to (0,2)$ which decay in Gaussian fashion at $\pm \infty$.
 
  The data of $V\in \mathcal{V}^l$ and $\beta\in\mathcal{B}$ give rise to functions $\tau_{\beta,i}\co (\R\times S^1)^k\to \R$ and $\tilde{\ev}_k\co C^0(\R\times S^1)^k\times (\R\times S^1)^k\to M^k$ as in
(\ref{taubetadef})  and (\ref{evscheme}) above.

Write $\mathcal{A}^l=\mathcal{J}^l\times \mathcal{H}^l\times \mathcal{J}^l\times\mathcal{B}$, equipped with the topology coming from the $C^l$ topology on the first three factors and the topology mentioned in the previous paragraph on the last factor, and let $\mathcal{A}=\cap_{l=2}^{\infty}\mathcal{A}^l$, with the topology coming from the $C^{\infty}$-topology on the first three factors.  Thus each $\mathcal{A}^l$ is a Banach manifold and $\mathcal{A}$ is a Frechet manifold.

Given $\frak{a}=(J,H,V,\beta)\in \mathcal{A}^l$; $\gamma^-,\gamma^+\in P(H)$; $C\in \pi_2(\gamma^-,\gamma^+)$; and $I=(i_1,\ldots,i_k)\in \{1,\ldots,m\}^k$ (where $k\in\Z_{\geq 0}$), 
let 
 \[ \tilde{\mathcal{M}}^{\frak{a}}(\gamma^-,\gamma^+,C;N_I)=\left\{(u,\vec{z},n_1,\ldots,n_k)\in \tilde{\mathcal{M}}(\gamma^-,\gamma^+,C)\times(\R\times S^1)^k\times\prod_{j=1}^{k}N_{i_j}\left|\begin{array}{c}\tilde{\ev}_k(u,z_1,\ldots,z_k)=\\ \,(f_{i_1}(n_1),\ldots,f_{i_k}(n_k))\end{array}\right.\right\}.\]  Again, the map $\tilde{\ev}_k$ is determined by $\beta$ and $V$ via (\ref{evscheme}).

Note that since the functions $\tau_{\beta,i}\co(\R\times S^1)^k\to \R$ of (\ref{taubetadef}) obey \[ \tau_{\beta,i}(\sigma_T(z_1),\ldots,\sigma_T(z_k))=\tau_{\beta,i}(z_1,\ldots,z_k),\] the spaces $\tilde{\mathcal{M}}^{\frak{a}}(\gamma^-,\gamma^+,C;N_I)$ admit $\R$-actions induced by translation of the domain.  These $\R$-actions are free except in the degenerate case where $\gamma^-=\gamma^+$, $k=0$, and $C\in \pi_2(\gamma^-,\gamma^-)$ is the trivial homotopy class. Let \[ \mathcal{M}^{\frak{a}}(\gamma^-,\gamma^+,C;N_I)=\tilde{\mathcal{M}}^{\frak{a}}(\gamma^-,\gamma^+,C;N_I)/\R.\]

Recalling that the dimension of the manifold $N_i$ is $2d(i)$ where $d(i)\leq n-2$, for $I=(i_1,\ldots,i_k)\in\{1,\ldots,m\}^k$ let \[ \delta(I)=\sum_{j=1}^{k}(2n-2-2d(i_j)).\]

We have:

\begin{prop}\label{pertmain}  Assume that $(M,\omega)$ is strongly semipositive.  Then there is a residual subset $\mathcal{A}^{reg}\subset\mathcal{A}$ such that for each $\frak{a}=(J,H,V,\beta)\in\mathcal{A}^{reg}$ the following holds:
\begin{itemize}
\item[(i)] Each moduli space $\tilde{\mathcal{M}}^{\frak{a}}(\gamma^-,\gamma^+,C;N_I)$ is a union of a smooth oriented manifold of dimension $\barmu(C)-\delta(I)$ together with spaces which are contained in smooth manifolds of dimension at most $\barmu(C)-\delta(I)-2$.
\item[(ii)] When $\barmu(C)-\delta(I)=1$, the quotient $\mathcal{M}^{\frak{a}}(\gamma^-,\gamma^+,C;N_I)$ consists of finitely many points; we denote the oriented number of these points by $|\frak{s}_{C,I}^{-1}(0)|$.
\item[(iii)] When $\barmu(C)-\delta(I)=2$, the quotient $\mathcal{M}^{\frak{a}}(\gamma^-,\gamma^+,C;N_I)$  has a compactification which is a smooth oriented $1$-manifold with oriented boundary $\partial \mathcal{M}^{\frak{a}}(\gamma^-,\gamma^+,C;N_I)$ given by \begin{equation}\label{bdryfib}
\coprod_{\substack{\gamma,C_-,C_+:\\C_-\#C_+=C}}\left(\coprod_{S\subset\{1,\ldots,k\}}\mathcal{M}^{\frak{a}}(\gamma^-,\gamma,C_-;N_{J_-(I,S)})\times \mathcal{M}^{\frak{a}}(\gamma,\gamma^+,C_+;N_{J_+(I,S)})\right).\end{equation}   Here for $I=(i_1,\ldots,i_k)\in \{1,\ldots,m\}^{k}$ and $S\subset \{1,\ldots,k\}$ (say with $\#S=l$) we denote by $J_-(I,S)$ the element of $\{1,\ldots,m\}^{l}$ consisting of the entries $i_m$ with $m\in S$ (taken in order) and by $J_+(I,S)$ the element of $\{1,\ldots,m\}^{k-l}$ consisting of the entries $i_m$ with $m\in \{1,\ldots,k\}\setminus S$.\end{itemize}
\end{prop}    

The proof of Proposition \ref{pertmain} is outlined in Appendix \ref{app}; the basic scheme of the proof is fairly standard but there are some tricky details to address regarding various strata in the compactifications of the moduli spaces.  To connect (iii) above to Remark \ref{mainrmk}, the set $S$ corresponds to what would be denoted there by $\{i_{1}^{-},\ldots,i_{l_-}^{-}\}$; thus $S$ corresponds to the set of marked points which fall onto the first component of a broken trajectory, while $\{1,\ldots,k\}\setminus S$ corresponds to those which fall onto the last component.

To generalize to the non-semipositive case (in which one has problems arising from the bubbling of multiply-covered spheres of negative Chern number), one can put Kuranishi structures with boundary and corners on the Gromov-Floer compactifications of the moduli spaces $\mathcal{M}^{\frak{a}}(\gamma^-,\gamma^+,C;N_I)$ using the techniques of \cite{FO}.  Alternately, and we would say somewhat more naturally (though it requires more machinery), one can use the constructions of Kuranishi structures on (unperturbed) fiber products from \cite{FOOO09}.  Choose any strongly nondegenerate Hamiltonian $H$ and an $S^1$-family $J$ of $\omega$-compatible almost complex structures.  For $\gamma^-,\gamma^+\in P(H)$ and $C\in \pi_2(\gamma^-,\gamma^+)$, denote \[ \tilde{\mathcal{M}}_k(\gamma^-,\gamma^+,C)= \tilde{\mathcal{M}}(\gamma^-,\gamma^+,C)\times (\mathbb{R}\times S^1)^k.\]  Where again $\mathbb{R}$ acts on each factor by translation of the first variable, we have a quotient \[ \mathcal{M}_k(\gamma^-,\gamma^+,C)=
\tilde{\mathcal{M}}_k(\gamma^-,\gamma^+,C)/\mathbb{R}\] whose compactification $C\mathcal{M}_k(\gamma^-,\gamma^+,C)$, carries  an oriented Kuranishi structure with corners (this can be shown by adapting the construction of \cite[Section 7.1]{FOOO09} from the Lagrangian to the Hamiltonian context).  Moreover, we have evaluation maps $ev_1,\ldots,ev_k\co C\mathcal{M}_k(\gamma^-,\gamma^+,C)\to M$ ($ev_i$ is given by evaluating at the $i$th marked point; in other words, for $[u,z_1,\ldots,z_k]$ in the dense subset $\mathcal{M}_k(\gamma^-,\gamma^+,C)$ we have $ev_i([u,z_1,\ldots,z_k])=u(z_i)$) which are weakly submersive (see \cite[Definition A1.13]{FOOO09} for the definition), in view of which, for any $I=(i_1,\ldots,i_k)$, there is  (as explained in \cite[Section A1.2]{FOOO09}) an oriented Kuranishi structure with corners on the fiber product \[ \mathcal{M}(\gamma^-,\gamma^+,C;N_I):=C\mathcal{M}_k(\gamma^-,\gamma^+,C)_{(ev_1,\ldots,ev_k)}\times_{(f_{i_1},\ldots,f_{i_k})}(N_{i_1}\times\cdots\times N_{i_k}).\]  Where $I=(i_1,\ldots,i_k)$, this Kuranishi structure has dimension \[ \dim \mathcal{M}(\gamma^-,\gamma^+,C;N_{i_1},\ldots,N_{i_k})=\bar{\mu}(C)-1-\delta(I),\] and just as in Proposition \ref{pertmain}(iv) the codimension-one stratum of its boundary consists of the interior of \begin{equation} \coprod_{\substack{\gamma,C_-,C_+:\\C_-\#C_+=C}}\left(\coprod_{S\subset\{1,\ldots,k\}}\mathcal{M}(\gamma^-,\gamma,C_-;N_{J_-(I,S)})\times \mathcal{M}(\gamma,\gamma^+,C_+;N_{J_+(I,S)})\right),\end{equation}  where the notation $J_{\pm}(I,S)$ has the same meaning as in (\ref{bdryfib}).

In the case that $\bar{\mu}(C)=1+\delta(I)$, let $\frak{s}_{C,I}$ denote the multisection associated to the Kuranishi structure on $\mathcal{M}(\gamma^-,\gamma^+,C;N_I)$, and let  $|\frak{s}_{C,I}^{-1}(0)|$ denote the sum of the rational multiplicities of the points of the (zero-dimensional) vanishing locus of $\frak{s}_{C,I}$.  

With this preparation regarding the relevant moduli spaces, we can now explain how to deform the standard Floer differential by a general homology class of even codimension.  Recall that at the outset of the construction we have specified smooth embeddings $f_i\co N_i\to M$ ($i=1,\ldots,m$) of smooth $2d(i)$-dimensional manifolds $N_i$ such that the classes $c_i=f_{i*}[N_i]$ represent a basis for $\oplus_{i=0}^{n-2}H_{2i}(M;\mathbb{Q})$.  For any $\eta\in \oplus_{i=0}^{n-1}H_{2i}(M;\Lambda_{\omega}^{0})$, let $\theta\in \Omega_{cl}^{2}(M;\Lambda_{\omega}^{0})$ be Poincar\'e dual to the degree-$(2n-2)$ component of $\eta$, and write the lower-dimensional part of $\eta$ as \[ \eta-PD[\theta]=\sum_{i=1}^{m}z_i c_i\in\oplus_{i=0}^{n-2}H_{2i}(M;\Lambda_{\omega}^{0}) \]  (so $z_i\in \Lambda_{\omega}^{0}$). For any $I=(i_1,\ldots,i_k)\in \{1,\ldots,m\}^k$, let \[ z_I=z_{i_1}\cdots z_{i_m}.\] 

Either choose $(J,H,V,\beta)\in\mathcal{A}^{reg}$ as in Proposition \ref{pertmain}, or choose a strongly nondegenerate Hamiltonian $H$ and an $S^1$-family of almost complex structures $J$ for input into the Kuranishi structure construction described after Proposition \ref{pertmain}; in either case this gives rise to curve counts $|\frak{s}_{C,I}^{-1}(0)|$.
Given two elements $[\gamma^-,v^-],[\gamma^+,v^+]\in\tilde{P}(H)$,  consider the expression \begin{equation}\label{defmatrix} \langle \partial^{\eta,H}[\gamma^-,v^-],[\gamma^+,v^+]\rangle=\sum_{k=0}^{\infty}\left(\sum_{\substack{C\in \pi_2(\gamma^-,\gamma^+): \\  \,[\gamma^+,v^+]=[\gamma^+,v^-\#C]}}\sum_{\substack{I\in\{1,\ldots,m\}^k:\\ \bar{\mu}(C)=\delta(I)+1}}\frac{1}{k!}|\frak{s}_{C,I}^{-1}(0)|\exp\left(\int_C\theta\right)z_I   \right).\end{equation}

\begin{prop} Only finitely many nonzero terms appear in the sum defining (\ref{defmatrix}).  Consequently we may define an endomorphism \[ \partial^{\eta,H}\co CF(H)\to CF(H) \] by extending $\Lambda_{\omega}$-linearly from \[ \partial^{\eta,H}[\gamma^-,v^-]=\sum_{[\gamma^+,v^+]\in\tilde{P}(H)}
 \langle \partial^{\eta,H}[\gamma^-,v^-],[\gamma^+,v^+]\rangle[\gamma^+,v^+].\]
\end{prop}

\begin{proof}   The condition that $[\gamma^+,v^+]=[\gamma^+,v^-\#C]$ is equivalent to the statement that $\int_C\omega=\int_{D^2}v^{+*}\omega-\int_{D^2}v^{-*}\omega$. A standard computation shows that  any element of $\mathcal{M}(\gamma^-,\gamma^+,C)$ has energy precisely equal to $\mathcal{A}_H([\gamma^-,v^-])-\mathcal{A}_H([\gamma^+,v^+])=\int_C\omega+\int_{0}^{1}(H(t,\gamma^-(t))-H(t,\gamma^+(t)))dt$. Gromov-Floer compactness then implies that there are just finitely many homotopy classes $C\in \pi_2(\gamma^-,\gamma^+)$ with $\int_C\omega=\int_{D^2}v^{+*}\omega-\int_{D^2}v^{-*}\omega$ which have $\mathcal{M}(\gamma^-,\gamma^+,C)\neq\varnothing$.  Thus there are only finitely many $C$ which contribute nonzero terms in (\ref{defmatrix}).

Let $\mu_0$ be the maximal value of $\bar{\mu}(C)$ over these finitely many $C$.  For any $k$ and $I\subset\{1,\ldots,m\}^k$ contributing to (\ref{defmatrix}) we have $\delta(I)\leq \mu_0-1$.  But we also have $\delta(I)\geq 2k$ since each $d(i_j)\leq n-2$, so all nonzero terms appearing in  (\ref{defmatrix}) have $k\leq \frac{1}{2}(\mu_0-1)$.  Since the union of the sets $\{1,\ldots,m\}^k$ for these values of $k$ is finite, this shows that there are only finitely many possible choices of $I$ which can contribute a nonzero term to (\ref{defmatrix}).  This completes the proof that the right hand side of (\ref{defmatrix}) is a finite sum.

It quickly follows that $\partial^{\eta,H}[\gamma^-,v^-]$ is a well-defined element of $CF(H)$; the relevant finiteness condition is satisfied by the standard Gromov-Floer compactness argument, owing to the fact that an element of $\mathcal{M}(\gamma^-,\gamma^+,C)$ has energy $\mathcal{A}_H([\gamma^-,v^-])-\mathcal{A}_H([\gamma^+,v^+])$.  It is easy to see that we have $\partial^{\eta,H}(T^g[\gamma^-,v^-])=T^g\partial^{\eta,H}[\gamma^-,v^-]$ for any $g\in \Gamma_{\omega}$, from which it follows that $\partial^{\eta,H}$ may be extended to a $\Lambda_{\omega}$-linear operator on $CF(H)$.
\end{proof}

\begin{prop}\label{combins}
Given $\gamma^-,\gamma^+\in P(H)$,  $C\in \pi_2(\gamma^-,\gamma^+)$, and $k\in\mathbb{N}$, we have \begin{equation}\label{uglysum} \sum_{k_-=0}^{k}\sum_{\gamma\in P(H)}\sum_{\substack{ C_-\in \pi_2(\gamma^-,\gamma),\\C_+\in \pi_2(\gamma,\gamma^+),\\C_-\#C_+=C}} \sum_{\substack{J_-\in\{1,\ldots,m\}^{k_-},\\J_+\in \{1,\ldots,m\}^{k-k_-},\\\bar{\mu}(C_{\pm})=\delta(J_{\pm})+1}}\frac{z_{J_-}z_{J_+}}{k_-!(k-k_-)!}|\frak{s}_{C_-,J_-}^{-1}(0)||\frak{s}_{C_+,J_+}^{-1}(0)|=0.\end{equation}
\end{prop}

\begin{proof} Since the number of subsets of cardinality $k_-$ in $\{1,\ldots,k\}$ is $\frac{k!}{k_-!(k-k_-)!}$, the sum is equivalent to \[ \frac{1}{k!}\sum_{\gamma,C_-,C_+}\sum_{J_-,J_+,S}z_{J_-}z_{J_+}|\frak{s}_{C_-,J_-}^{-1}(0)||\frak{s}_{C_+,J_+}^{-1}(0)|\] where $\gamma,C_-,C_+$ satisfy the same constraints as in (\ref{uglysum}), and where $J_-,J_+$ are required to be tuples of elements of $\{1,\ldots,m\}$ of combined length $k$ and $S$ is a subset of $\{1,\ldots,k\}$ of cardinality equal to the length of the tuple $J_-$.

Now the assignment $(I,S)\mapsto (J_-(I,S),J_+(I,S),S)$ (notation as in (\ref{bdryfib})) is a bijection from the set \[ \{(I,S)|I\subset\{1,\ldots,m\}^k,S\subset \{1,\ldots,k\}\}\]  to the set of data $(J_-,J_+,S)$ as in the inner sum above, and if we have $J_{\pm}(I',S')=J_{\pm}(I,S)$ then \[ z_{I}=z_{J_{+}(I,S)}z_{J_-(I,S)}=z_{I'}.\]    So the sum reduces to \[ \frac{1}{k!}\sum_{I\in\{1,\ldots,m\}^k}z_I\sum_{S\subset\{1,\ldots,k\}}\sum_{\gamma,C_-,C_+}|\frak{s}_{C_-,J_-(I,S)}^{-1}(0)|
|\frak{s}_{C_+,J_+(I,S)}^{-1}(0)|.\]  But by (\ref{bdryfib}) the inner two sums, for any given $I$, enumerate with multiplicity the points on the boundary of $\frak{s}_{C,I}^{-1}(0)$, and so for each $I$ the coefficient on $z_I$ is equal to zero.
\end{proof}

\begin{cor}\label{delsquared}
For any  $\theta\in\Omega^{2}_{cl}(M;\Lambda_{\omega}^{0})$ and any $c=\sum z_ic_i$ as above we have \[ (\partial^{\eta,H})^{2}=0.\]
\end{cor}

\begin{proof}Denote the left hand side of (\ref{uglysum}) by $r_k(\gamma_-,\gamma_+,C)$.  It is straightforward to see that, in the sum defining $(\partial^{\eta,H})^{2}[\gamma^-,v^-]$, the coefficient on $[\gamma^+,v^+]$ is \[ \sum_{k=0}^{\infty}\frac{1}{k!}\sum_{\substack{C\in \pi_2(\gamma^-,\gamma^+),\\\int_{C}\omega=\int_{D^2}v^{+*}\omega-\int_{D^2}v^{*-}\omega }}\exp\left(\int_C\theta\right)r_k(\gamma_-,\gamma_+,C).\]
So the corollary follows directly from the fact  that, by Proposition \ref{combins}, each $r_k(\gamma^-,\gamma^+,C)=0$.
\end{proof}

\subsection{Other maps on the Floer complexes}

For $\eta,H$ as in the previous section (together with appropriate auxiliary data which are suppressed from the notation), the data \[ \frak{c}_{\eta,H}=(\tilde{P}(H)\to P(H),\mathcal{A}_H,\omega,\partial^{\eta,H}) \] comprise the structure of a ``filtered Floer--Novikov complex'' in the sense of \cite{U08},\cite{U10b}, with Floer chain complex equal to $CF(H)$ and boundary operator $\partial^{\eta,H}$ (since we only consider a $\mathbb{Z}_2$-grading, the grading and degree in the definition in \cite{U10b} can just be set equal to zero).  Following those references, for any $c=\sum c_{[\gamma,v]}[\gamma,v]\in CF(H)$ we set \begin{equation}\label{elldef} \ell(c)=\max\{\mathcal{A}_H([\gamma,v])|c_{[\gamma,v]}\neq 0\}.\end{equation} Also, for $\lambda\in\mathbb{R}$, let \[ CF^{\lambda}(H)=\{c\in CF(H)|\ell(c)\leq\lambda\}.\]

In the standard case $\eta=0$, there are a variety of maps that can be defined on the Floer complexes which formally count solutions to appropriate modifications of the Floer equation (\ref{floereqn}).  These maps have straightforward analogues when the deformation parameter $\eta$ is nontrivial; loosely speaking, the maps simply need to be modified in the same way that the standard differential is modified.  Our discussion will use the framework of Kuranishi structures on fiber products of \cite{FOOO09}. In the strongly semipositive case one can instead use the approach described in Section \ref{bigdefsect} to achieve transversality for the relevant fiber products in a more direct way; we omit the details of this, as the arguments are essentially identical to those in Section \ref{bigdefsect}.  At least after restricting to appropriate residual subsets of spaces of auxiliary data, this approach would give rise in the strongly semipositive case to ``Kuranishi-structure-free'' proofs of all of the results in the rest of Section \ref{hamsect} except for Theorem \ref{slowthm} (regarding Theorem \ref{slowthm}, which is used only for the results of Section \ref{capsec} and not for the other main results of the paper, see Remark \ref{nokur}).

\subsubsection{Continuation maps}\label{contsect}
For example, given $\eta$, let $H_-,H_+\co S^1\times M\to \mathbb{R}$ be two strongly nondegenerate Hamiltonians, let $\tilde{H}\co\mathbb{R}\times S^1\times M\to\mathbb{R}$ be a smooth function such that $\tilde{H}|_{\{s\}\times S^1\times M}=H_-$ for $s<-1$ and $\tilde{H}|_{\{s\}\times S^1\times M}=H_+$ for $s>1$, and also (for reasons that will become apparent later) choose an additional smooth function $K\co \mathbb{R}\times S^1\times M\to\mathbb{R}$ having support in $[-1,1]\times S^1\times\mathbb{R}$.

Where $\{J(s,t)\}_{(s,t)\in\mathbb{R}\times S^1}$ is a family of $\omega$-compatible almost complex structures with $\frac{\partial J}{\partial s}=0$ for $|s|$ large, for $C\in \pi_2(\gamma^-,\gamma^+)$ let $\mathcal{N}(\gamma^-,\gamma^+,C,\tilde{H},K)$ denote the space of finite energy solutions $u\co \mathbb{R}\times S^1\to M$ which represent $C$ and obey \begin{equation}\label{continuation} \left(\frac{\partial u}{\partial s}-X_K(s,t,u(s,t))\right)+J(s,t,u(s,t))\left( \frac{\partial u}{\partial t}-X_{\tilde{H}}(s,t,u(s,t))  \right)=0 \end{equation}

As in \cite[Section 20]{FO},\cite[Chapter 7]{FOOO09}, the compactification $C\mathcal{N}(\gamma^-,\gamma^+,C,\tilde{H},K)$ admits a Kuranishi structure with corners, as does the compactification  $C\mathcal{N}_k(\gamma^-,\gamma^+,C,\tilde{H},K)$ of the space $\mathcal{N}_k(\gamma^-,\gamma^+,C,\tilde{H},K)=\mathcal{N}(\gamma^-,\gamma^+,C,\tilde{H},K)\times (\mathbb{R}\times S^1)^k$, and the evaluation maps\newline  $ev_1,\ldots,ev_k\co C\mathcal{N}_k(\gamma^-,\gamma^+,C,\tilde{H},K)\to M$ are weakly submersive. 
Recalling the maps $f_i\co N_i\to M$ that were fixed above Definition \ref{sndeg}, it follows that for any $I=(i_1,\ldots,i_k)\in \{1,\ldots,m\}^k$ the fiber product \[ 
\mathcal{N}_k(\gamma^-,\gamma^+,C,\tilde{H},K;N_I)=C\mathcal{N}_k(\gamma^-,\gamma^+,C,\tilde{H},K)_{(ev_1,\ldots,ev_k)}\times_{(f_{i_1},\ldots,f_{i_k})}(N_{i_1}\times\cdots\times N_{i_k})\] has a Kuranishi structure with corners, with dimension $\bar{\mu}(C)-\delta(I)$.  Letting $\frak{s}_{\tilde{H},K,C,I}$ be the associated multisection, define \[ \Phi_{\eta, \tilde{H},K}\co CF(H_-)\to CF(H_+) \] by extending $\Lambda_{\omega}$-linearly from \begin{equation} \label{defcont}  \Phi_{\eta, \tilde{H},K}[\gamma^-,v^-]=\sum_{\gamma^+\in P(H^+)}\sum_{C\in \pi_2(\gamma_-,\gamma_+)}\sum_{k=0}^{\infty}\frac{1}{k!}\sum_{\substack{I\in\{1,\ldots,m\}^k,\\ \bar{\mu}(C)=\delta(I)}}|\frak{s}_{\tilde{H},K,C,I}^{-1}(0)|\exp\left(\int_C\theta\right)z_I[\gamma^+,v^-\#C].\end{equation}

Some words are in order regarding why the formal sum right hand side of (\ref{defcont}) validly defines an element of $CF(H)$.  Remark first of all that a direct computation shows that, where $E(u)=\int_{\mathbb{R}\times S^1}\left|\frac{\partial u}{\partial s}-X_K(s,t,u(s,t))\right|_{J_{s,t}}^{2}dsdt$ and $\{\tilde{H},K\}=\omega(X_{\tilde{H}},X_K)$, any solution $u$ to (\ref{continuation}) with $[\gamma^+,v^-\#u]=[\gamma^+,v^+]$ obeys \begin{align}\label{contest}
\mathcal{A}_{H_-}&([\gamma^-,v^-])-\mathcal{A}_{H_+}([\gamma^+,v^+]) \nonumber\\&=E(u)-\int_{\mathbb{R}\times S^1}\left(\frac{\partial \tilde{H}}{\partial s}(s,t,u(s,t))-\frac{\partial K}{\partial t}(s,t,u(s,t))-\{\tilde{H},K\}(s,t,u(s,t))\right)dsdt.
\end{align}
Write \[ C^+(\tilde{H},K)=\int_{-\infty}^{\infty}\int_{0}^{1}\max_{\{(s,t)\}\times M}\left(\frac{\partial \tilde{H}}{\partial s}(s,t,\cdot)-\frac{\partial K}{\partial t}(s,t,\cdot)-\{\tilde{H},K\}(s,t,\cdot)\right)dsdt \] 
(note that the integrand is supported in $[-1,1]\times S^1$ by the construction of $\tilde{H}$ and $K$).  
Then (\ref{contest}) shows that, given $[\gamma^-,v^-],[\gamma^+,v^+]$, any cylinder $u$ solving (\ref{continuation}) with $u(s,\cdot)\to \gamma^{\pm}$ as $s\to\pm\infty$ and $[\gamma^+,v^-\#u]=[\gamma^+,v^+]$ obeys $E(u)\leq \mathcal{A}_{H_-}([\gamma^-,v^-])-\mathcal{A}_{H_+}([\gamma^+,v^+])+C^+(\tilde{H},K)$.  Gromov-Floer compactness consequently shows that there can be just finitely many homotopy classes $C\in \pi_2(\gamma^-,\gamma^+)$ which contribute to the coefficient of $[\gamma^+,v^+]$ in the formula (\ref{defcont}) for    $\Phi_{\eta, \tilde{H},K}[\gamma^-,v^-]$.  So, as was the case with the boundary operator, the fact that $\delta(I)\geq 2k$ shows that, since the only pairs $(C,I)$ appearing in (\ref{defcont}) have $\bar{\mu}(C)=\delta(I)$, the coefficient on any given $[\gamma^+,v^+]$ in the expression for $\Phi_{\eta, \tilde{H},K}[\gamma^-,v^-]$ is a sum over just finitely many $C,k,I$ and so is a well-defined element of $\Lambda_{\omega}^{0}$.  Moreover, for any given $m\in\mathbb{R}$ and $[\gamma^-,v^-]\in\tilde{P}(H)$, any cylinder $u$  which contributes to a nonzero coefficient in $\Phi_{\eta, \tilde{H},K}[\gamma^-,v^-]$ on some $[\gamma^+,v^+]$ with $\mathcal{A}_{H_+}([\gamma^+,v^+])\geq m$ necessarily has $E(u)\leq 
C^+(\tilde{H},K)+\mathcal{A}_{H_-}([\gamma^-,v^-])-m$.  Consequently Gromov-Floer compactness implies that there can be just finitely many $[\gamma^+,v^+]$ with $\mathcal{A}_{H_+}([\gamma^+,v^+])\geq m$ which appear with nonzero coefficient in $\Phi_{\eta, \tilde{H},K}[\gamma^-,v^-]$.  This establishes the Novikov finiteness condition for the right hand side of (\ref{continuation}), and so justifies the definition of the $\Lambda_{\omega}$-linear map $\Phi_{\eta, \tilde{H},K}\co CF(H_-)\to CF(H_+)$.

Moreover, the map $\Phi_{\eta, \tilde{H},K}$ is a chain map (with respect to the differentials $\partial^{\eta,H_{\pm}}$); this follows from a consideration of the boundaries of one-dimensional moduli spaces $\mathcal{N}_k(\gamma^-,\gamma^+,C,\tilde{H},K;N_I)$ in much the same way as in the proof that $(\partial^{\eta,H})^2=0$.

Observe also that, since we always have $E(u)\geq 0$, the estimate (\ref{contest}) shows that, for all $c\in CF(H_-)$,  \begin{equation}\label{filtcont} \ell(\Phi_{\eta, \tilde{H},K}c)\leq \ell(c)+C^+(\tilde{H},K);\end{equation} thus, on the filtered subcomplexes $CF^{\lambda}(H_-)$ of $CF(H_-)$, $\Phi_{\eta,\tilde{H},K}$ restricts as a map \[ \Phi_{\eta,\tilde{H},K}\co CF^{\lambda}(H_-)\to CF^{\lambda+C^+(\tilde{H},K)}(H_+).\]  

A noteworthy special case is that in which the two Hamiltonians $H_{\pm}$ are normalized (i.e. $\int_M H_{\pm}(t,\cdot)\omega^n=0$ for all $t$) and induce paths $\{\phi_{H_{\pm}}^{t}\}_{t\in [0,1]}$ in $Ham(M,\omega)$ which are homotopic rel endpoints.  In this case, as is done in \cite[p. 20]{U09}, we may choose $\tilde{H}\co \mathbb{R}\times S^1\times M\to\mathbb{R}$ so that the $\{\phi_{\tilde{H}(s,\cdot)}^{t}\}$  ($-1\leq s\leq 1$) realize a homotopy from $\{\phi_{H_-}^{t}\}$ to $\{\phi_{H_+}^{t}\}$ with each $\tilde{H}(s,\cdot)$ normalized.  If we then define $K(s,t,\cdot)\co M\to\mathbb{R}$ to be the unique mean zero function so that $\frac{d}{ds}(\phi_{\tilde{H}(s,\cdot)}^{t}(p))=X_{K(s,t,\cdot)}(\phi_{\tilde{H}(s,\cdot)}^{t}(p))$ for all $p$, it will hold that \[ 
\frac{\partial \tilde{H}}{\partial s}(s,t,\cdot)-\frac{\partial K}{\partial t}(s,t,\cdot)-\{\tilde{H},K\}(s,t,\cdot)=0 \]  (this is a well-known consequence of \cite[Proposition I.1.1]{Ban} and the normalization condition on  $H$ and $K$).

So in this special case we have $C^+(\tilde{H},K)=0$, and so $\Phi_{\eta,\tilde{H},K}\co CF^{\lambda}(H_-)\to CF^{\lambda}(H_+)$.  We can equally well construct a similar map in the opposite direction, and then an exact reproduction of the proof of \cite[Lemma 3.8]{U09} shows the following:

\begin{prop}\label{filtiso}  Given a closed symplectic manifold $(M,\omega)$ and $\eta\in \oplus_{i=0}^{n-1}H_{2i}(M;\Lambda_{\omega}^{0})$, let $\tilde{\phi}\in \widetilde{Ham}(M,\omega)$ be represented by strongly nondegenerate, normalized Hamiltonians $H_{\pm}\co S^1\times M\to\mathbb{R}$.  Then for a suitable choice of $\tilde{H},K$, the chain map \[ \Phi_{\eta,\tilde{H},K}\co (CF(H_-),\partial^{\eta,H_-})\to (CF(H_+),\partial^{\eta,H_+})\] is an \emph{isomorphism} of chain complexes which, for each $\lambda\in\mathbb{R}$, restricts as an isomorphism \[ \xymatrix{ CF^{\lambda}(H_-)\ar[r]^{\sim}& CF^{\lambda}(H_+).    }\]
\end{prop} 

\subsubsection{PSS isomorphisms}

The Piunikhin-Salamon-Schwarz isomorphism \cite{PSS} from Morse homology to Floer homology can likewise be adapted to the deformed complexes $(CF(H),\partial^{\eta,H})$.  Given $\eta\in \oplus_{i=0}^{n-1}H_{2i}(M;\Lambda_{\omega}^{0})$, choose a suitably generic pair consisting of a Morse function $f\co M\to\mathbb{R}$ and a Riemannian metric $g$ on $M$ (with the stable and unstable manifolds of $f$ transverse both to each other and to the maps $f_i\co N_i\to M$).  The data $(f,g)$ determine a Morse complex $CM(f;\Lambda_{\omega})=\sum_{p\in Crit(f)}\Lambda_{\omega}\langle p\rangle$, with differential $\partial_f$ counting negative gradient flowlines in the standard way. 

If $H$ is a strongly nondegenerate Hamiltonian, the usual PSS map $\Phi^{PSS}_{0,H}\co CM(f;\Lambda_{\omega})\to CF(H)$ enumerates ``spiked planes'' comprising a half-infinite negative gradient flowline $\zeta\co (-\infty,0]\to M$ for $f$ and a finite-energy map $v\co \mathbb{C}\to M$ obeying, with respect to the standard polar coordinates $r,\theta$ on $\mathbb{C}$, \begin{equation}
\label{psseq} r\frac{\partial v}{\partial r}+J(re^{i\theta},v(re^{i\theta}))\left(\frac{\partial v}{\partial \theta}-\frac{\beta(r)}{2\pi}X_H(\theta/2\pi,v(re^{i\theta}))\right)=0\end{equation} with $\zeta(0)=v(0)$ where $\beta\co [0,\infty)\to [0,1]$ is a smooth monotone function with $\beta(r)=0$ for $r<1/2$ and $\beta(r)=1$ for $r>1$.  A straightforward modification of this map along the lines of the construction of $\partial^{\eta,H}$ produces a deformed PSS map $\Phi^{PSS}_{\eta,H}\co CM(f;\Lambda_{\omega})\to CF(H)$, as follows.

Given $\gamma\in P(H)$, write $\pi_2(\gamma)$ for the set of relative homotopy classes of discs $u\co D^2\to M$ with $v|_{\partial D^2}=\gamma$.  If $C\in \pi_2(\gamma)$, write $\bar{\mu}(C)=\mu_{CZ}(\gamma,u)$ where $u$ is any disc representing $C$.  A finite-energy map $v\co \mathbb{C}\to M$ which obeys (\ref{psseq}) and has $v(re^{i\theta})\to \gamma(\theta/2\pi)$ as $r\to\infty$ determines a class $[v]\in \pi_2(\gamma)$ (by rescaling the radial coordinate to $[0,1)$ and compactifying).  For $C\in \pi_2(\gamma)$, let \[ \mathcal{M}^{PSS}_{k}(\gamma,C)=\{(v,\vec{z})\in Map(\mathbb{C},M)\times \mathbb{C}^k|v\mbox{ satisfies (\ref{psseq})}, [v]=C\}\] and write its standard  compactification as $C\mathcal{M}^{PSS}_{k}(\gamma,C)$, with evaluation maps $ev_i$ ($1\leq i\leq k$) at the $i$th marked point.  Thus an element of $C\mathcal{M}^{PSS}_{k}(\gamma,C)$ amounts to the data of a tree consisting of a solution of (\ref{psseq}) and some number (possibly zero) of pseudoholomorphic spheres and Floer trajectories, with total homotopy class $C$, together with $k$ marked points distributed among the domains of the various components.

Define $e_0\co \mathcal{M}^{PSS}_{k}(\gamma,C)\to M$ by $e_0(v,\vec{z})=v(0)$ (and also use the notation $e_0$ for the extension to the compactification $C\mathcal{M}^{PSS}_{k}(\gamma,C)$), and, for $p\in Crit(f)$, let $j_p\co W^u(p)\to M$ denote the inclusion of the unstable manifold of $p$.  The appropriate spaces of spiked planes can then be written as \[ \mathcal{M}^{PSS}_k(\gamma,C;p,N_I)=C\mathcal{M}^{PSS}_{k}(\gamma,C)_{(e_0,ev_1,\ldots,ev_k)}\times_{(j_p,f_{i_1},\ldots,f_{i_k})}(W^{u}(p)\times N_{i_1}\times\cdots\times N_{i_k}).\] This fiber product has a Kuranishi structure of dimension $\bar{\mu}(C)+ind_f(p)-n-\delta(I)$ where $ind_f(p)$ denotes the Morse index of the critical point $p$.  Where $\frak{s}_{C,p,I}$ is the corresponding multisection, the PSS map $\Phi_{\eta,H}^{PSS}\co CM(f;\Lambda_{\omega})\to CF(H)$ is defined by extending $\Lambda_{\omega}$-linearly from \[ \Phi_{\eta,H}^{PSS}p=\sum_{k=0}^{\infty}\frac{1}{k!}\sum_{\substack{\gamma\in P(H),\\C\in \pi_2(\gamma)}}\sum_{\substack{I\in\{1,\ldots,m\}^k,\\ 
\bar{\mu}(C)+ind_f(p)=n+\delta(I)}}|\frak{s}_{C,p,I}^{-1}(0)|\exp\left(\int_C\theta\right)z_I[\gamma,C],\] where of course $\int_C\theta$ denotes the integral of the closed $\Lambda_{\omega}^{0}$-valued $2$-form $\theta$ over any disc with boundary $\gamma$ representing the homotopy class $C$.  The fact that the right hand side of the above formula validly defines an element of $CF(H)$ follows by an argument very similar to that used earlier to establish the corresponding fact for the continuation map $\Phi_{\eta,\tilde{H},K}$.

In the other direction, one obtains a similar map $\Psi_{\eta,H}^{PSS}\co CF(H)\to CM(f;\Lambda_{\omega})$ which formally enumerates configurations consisting of a finite-energy solution $v\co (\mathbb{C}\cup\{\infty\})\setminus\{0\}\to M$ of the equation 
\begin{equation}\label{psieq} r\frac{\partial v}{\partial r}+J(re^{i\theta},v(re^{i\theta}))\left(\frac{\partial v}{\partial \theta}-\frac{\beta(r^{-1})}{2\pi}X_H(\theta/2\pi,v(re^{ i\theta}))\right)=0 \end{equation} 
and a negative gradient flowline 
$\zeta\co [0,\infty)\to M$ for $f$, with $\zeta(0)=v(\infty)$,  $v(re^{i\theta})\to \gamma(\theta/2\pi)$ as $r\to 0$ for some $\gamma\in P(H)$, and such that $v$ obeys appropriate incidence conditions.  Where, for $p\in Crit(f)$, $\gamma\in P(H)$, $I\subset \{1,\ldots,m\}^k$, and $C\in \pi_2(\gamma)$, we let $\bar{\frak{s}}_{C,p,I}$ be the multisection associated to the Kuranishi structure on the space of such configurations asymptotic to $\gamma$, such that the map 
$\bar{v}(re^{i\theta})=v(r^{-1}e^{i\theta})$ 
represents $C\in \pi_2(\gamma)$, and obeying the incidence conditions corresponding to $I$, the map $\Psi^{PSS}_{\eta,H}\co CF(H)\to CM(f;\Lambda_{\omega})$ takes the form \[ \Psi_{\eta,H}^{PSS}([\gamma,u])=\sum_{p\in Crit(f)}\sum_{k=0}^{\infty}\frac{1}{k!}\sum_{\substack{C,p,I:\\ 
-\bar{\mu}(C)-ind_f(p)=-n+\delta(I)}}|\bar{\frak{s}}_{C,p,I}^{-1}(0)|\exp\left(-\int_C\theta\right)z_IT^{\int_{D^2}u^*\omega-\int_C\omega}p.\] 
(Note that, if $v$ is a solution as above contributing to the term corresponding to $C\in \pi_2(\gamma)$, then  $\int_{D^2}u^*\omega-\int_C\omega$ is the integral of $\omega$ over a sphere obtained by gluing the capping disc $u$ to the solution $v$ along $\gamma$, so that this quantity does belong to $\Gamma_{\omega}$).

The following summarizes some properties of the PSS maps:

\begin{prop}\label{pssprop}For suitable choices of auxiliary data involved in the construction of the deformed PSS maps, the following properties hold:\begin{itemize}
\item[(i)]  The maps $\Phi_{\eta,H}^{PSS}\co (CM(f;\Lambda_{\omega}),\partial_f)\to (CF(H),\partial^{\eta,H})$ and $\Psi_{\eta,H}^{PSS}\co (CF(H),\partial^{\eta,H})\to (CM(f;\Lambda_{\omega}),\partial_f)$ are chain maps, which respect the $\mathbb{Z}_2$-gradings of the respective complexes.
\item[(ii)] If $H_-,H_+$ are two strongly nondegenerate Hamiltonians and if $\tilde{H},K$ are data as in Section \ref{contsect}, resulting in a continuation map $\Phi_{\eta,\tilde{H},K}\co (CF(H_-),\partial^{\eta,H_-})\to (CF(H_+),\partial^{\eta,H_+})$, the maps \[\Phi_{\eta,H_+}^{PSS}\mbox{ and }\Phi_{\eta,\tilde{H},K}\circ\Phi_{\eta,H_-}^{PSS}\co (CM(f;\Lambda_{\omega}),\partial_f)\to (CF(H_+),\partial^{\eta,H_+})\] 
are chain homotopic.
\item[(iii)] The compositions $\Psi_{\eta,H}^{PSS}\circ\Phi_{\eta,H}^{PSS}\co (CM(f;\Lambda_{\omega}),\partial_f) \to (CM(f;\Lambda_{\omega}),\partial_f)$ and $\Phi_{\eta,H}^{PSS}\circ\Psi_{\eta,H}^{PSS}\co (CF(H),\partial^{\eta,H})\to (CF(H),\partial^{\eta,H})$ are each chain homotopic to the identity.
\item[(iv)] For $x=\sum_{i=1}^{\infty}x_ip_iT^{g_i}\in CM(f;\Lambda_{\omega})$ (where $p_i\in Crit(f), x_i\in\mathbb{C}, g_i\in \Gamma_{\omega}$) write \[ \nu(x)=\max\{-g_i|x_i\neq 0\}.\]  Then \[ \ell(\Phi_{\eta,H}^{PSS}x)\leq \nu(x)+\int_{0}^{1}\max_MH(t,\cdot)dt \] 
where $\ell$ is defined in (\ref{elldef}).
\item[(v)] For $c\in CF(H)$, \[ \ell(c)\geq \nu(\Psi_{\eta,H}^{PSS}c)+\int_{0}^{1}\min_M H(t,\cdot)dt.\]
\end{itemize}
\end{prop}

\begin{proof} Properties (i) and (ii) follow from standard gluing and cobordism arguments as in \cite{PSS} and \cite{Lu04}; the only new feature is the presence of incidence conditions, which are handled in the same way as in the proof that $(\partial^{\eta,H})^2=0$.  (The grading property is a straightforward consequence of the definitions (in particular see the second displayed equation of \cite[Remark 5.4]{RS}), taking into account that $\delta(I)$ is always even.)

Consider the composition $\Psi_{\eta,H}^{PSS}\circ\Phi_{\eta,H}^{PSS}$.  A  gluing argument as in \cite{PSS},\cite[Proposition 4.6]{Lu04} shows that this composition is equal to a map $\phi_{\bar{H}}\co CM(f;\Lambda_{\omega})\to CM(f;\Lambda_{\omega})$ given by \begin{equation} \label{psiphi} \phi_{\bar{H}}p=\sum_{q\in Crit(f)}\sum_{k=0}^{\infty}\frac{1}{k!}\sum_{A\in \pi_2(M)}\sum_{\substack{I\in \{1,\ldots,m\}^k,\\ 2\langle c_1(M),A\rangle+ind_f(p)-ind_f(q)=\delta(I)}}n(\bar{H};p,q,A,I)z_I\exp\left(\int_C\theta\right)T^{\int_C\omega}q.\end{equation} 
Here the rational number $n(\bar{H};p,q,A,I)$ formally enumerates triples $(\zeta_-,u,\zeta_+)$ where \begin{itemize} \item $\zeta_-\co (-\infty,0]\to M$, $\zeta_+\co [0,\infty)\to M$ are negative gradient flowlines of $f$ with $\zeta_-(\tau)\to p$ as $\tau\to -\infty$ and $\zeta_+(\tau)\to q$ as $\tau\to\infty$;\item $u\co \mathbb{C}\cup\{\infty\}\to M$ is a solution to $\delbar_{J,\bar{H}}u=0$ which represents $A$ in $\pi_2(M)$ and satisfies incidence conditions corresponding to $I$, for a suitable Hamiltonian perturbation $\bar{H}$ which vanishes near $0$ and $\infty$; and \item $u(0)=\zeta_-(0)$ and $u(\infty)=\zeta_+(0)$.\end{itemize}

Meanwhile, a standard cobordism argument shows that the chain homotopy class of the map $\phi_{\bar{H}}$ is independent of the Hamiltonian perturbation $\bar{H}$.  Moreover, a cobordism argument also shows that the chain homotopy class of $\phi_{\bar{H}}$ is unchanged if we replace the cycles $f_i\co N_i\to M$ used for the incidence conditions $I=\{i_1,\ldots,i_k\}$ with an arbitrary family of choices \[ \{(\alpha_{1}^{I},\ldots,\alpha_{k}^{I})|I\in \{1,\ldots,m\}^k, k\geq 0\}\] so that for each $I,j$ the cycle $\alpha_{j}^{I}$ represents the same homology class $c_{i_j}$ as does $f_{i_j}$.   

In particular, we can choose $\bar{H}=0$, in which case the spherical component $u$ of the triple $(\zeta_-,u,\zeta_+)$ is required to be $J$-holomorphic. There is then an $S^1$-action on the relevant moduli spaces (induced by rotating the sphere) and as in \cite[Proposition 4.7]{Lu04}, \cite[p. 1036]{FO} the fact that this action is locally free on all of the spaces except those corresponding to a topologically trivial $u$ implies that only the class $A=0\in \pi_2(M)$ contributes nontrivially to $\phi_0$, and the only contributions come from constant maps $u\co S^2\to M$ to points $x$ of the intersections $W^u(p)\cap W^s(q)$ (which have dimension $ind_f(p)-ind_f(q)$).  Moreover, the point $x$ must meet the cycles $\alpha_{1}^{I},\alpha_{k}^{I}$; if these latter are mutually transverse to each other  and to the $W^u(p),W^s(q)$ (as we may and do choose them to be) this imposes a condition of codimension $\sum_{j=1}^{k}(2n-2d(i_j))=\delta(I)+2k$.  So since the only terms in (\ref{psiphi}) corresponding to $A=0$ have 
$ind_f(p)-ind_f(q)=\delta(I)$  the only contributions to $\phi_{0}$ (for these choices of  the $\alpha_{j}^{I}$) have $k=0$, and so just as in \cite{PSS},\cite[Proposition 4.7]{Lu04} we find that the map $\phi_{0}$ is the identity.

The proof that $\Phi_{\eta,H}^{PSS}\circ\Psi_{\eta,H}^{PSS}\co (CF(H),\partial^{\eta,H})\to (CF(H),\partial^{\eta,H})$ is chain homotopic to the identity is just the same as that outlined in \cite{PSS}, \cite{Lu04} (note that to complete this outline one needs to incorporate the gluing analysis of \cite{OZ}): using the appropriate gluing and cobordism arguments one finds that $\Phi_{\eta,H}^{PSS}\circ\Psi_{\eta,H}^{PSS}$ is chain homotopic to the map that in the notation of Section \ref{contsect} would be denoted $\Phi_{\eta,\tilde{H},0}$, with $\tilde{H}(s,t,m)=H(t,m)$.  But this latter map is the identity as in \cite[(20.7)]{FO}, since with the exception of the ``trivial cylinders'' $u(s,t)=\gamma(t)$ all of the cylinders which might contribute to it are members of orbits of free $\mathbb{R}$-actions and so do not arise in zero-dimensional moduli spaces.  This completes the proof of part (iii).

As for (iv), note that in order for $[\gamma,v]$ to appear with nonzero coefficient in $\Phi_{\eta,H}^{PSS}p$ there must exist a solution $u\co \mathbb{C}\to M$ to (\ref{psseq}) asymptotic to $\gamma$ having $\int_{\mathbb{C}}u^{*}\omega=\int_{D^2} v^*\omega$.  In this case a computation gives \begin{align*}\mathcal{A}_H([\gamma,v])&=-\int_{\mathbb{C}}u^*\omega+\int_{0}^{1}H(t,\gamma(t))dt
\\&=-E(u)+\int_{0}^{1}\int_{0}^{\infty}\beta'(r)H(t,u(re^{2\pi it}))drdt\\&\leq \int_{0}^{1}\max_M H(t,\cdot)dt,\end{align*} where we've used the facts that $\int_{0}^{\infty}\beta'(r)dr=1$ and that the  energy $E(u)=\int_{\mathbb{C}}\left|\frac{\partial u}{\partial r}\right|^2rdrd\theta$ is nonnegative. Bearing in mind that the expression $\exp(\int_C\theta)z_I$ belongs to $\Lambda_{\omega}^{0}$, so that its action on $CF(H)$ does not increase the value of $\ell$,  this implies that, for any $p\in Crit(f)$, we have \[ \ell(\Phi_{\eta,H}^{PSS}p)\leq \int_{0}^{1}\max_M H(t,\cdot)dt,\] and then (iv) follows from obvious properties of the function $\ell$.

The proof of (v) is similar: For a term $T^gp$ to appear in $\Psi_{\eta,H}^{PSS}[\gamma,u]$ there must be a solution $v\co (\mathbb{C}\cup\{\infty\})\setminus \{0\}\to M$ of (\ref{psieq}), asymptotic to $\gamma$ as $|z|\to 0$, and having \[ \int_{(\mathbb{C}\cup\{\infty\})\setminus \{0\}}v^*\omega=-\int_{D^2}u^*\omega+g\]  For such a solution we find \begin{align*} \int_{(\mathbb{C}\cup\{\infty\})\setminus\{0\}}v^*\omega&=E(v)-\int_{0}^{1}H(t,\gamma(t))dt-\int_{0}^{1}\int_{0}^{\infty}\frac{d}{dr}\left(\beta(r^{-1})\right)H(t,v(re^{2\pi it}))drdt \\&\geq -\int_{0}^{1}H(t,\gamma(t))dt+\int_{0}^{1}\min_M H(t,\cdot)dt.\end{align*}Hence we obtain \[ g\geq \int_{D^2}u^*\omega-\int_{0}^{1}H(t,\gamma(t))dt+\int_{0}^{1}\min_M H(t,\cdot)dt=-\mathcal{A}_H([\gamma,u])+\int_{0}^{1}\min_M H(t,\cdot)dt,\] and then from the definitions of $\Psi_{\eta,H},\ell,\nu$ we see that \[ \nu(\Psi_{\eta,H} c)\leq \ell(c)-\int_{0}^{1}\min_M H(t,\cdot)dt\] for any $c\in CF(H)$.

\end{proof}

\subsubsection{Pair of Pants products}

We now discuss the deformed versions of the pair of pants product on Floer homology.  We continue to regard as fixed the class $\eta\in \oplus_{i=0}^{n-1}H_{2i}(M;\Lambda_{\omega}^{0})$ and the representing chains $f_i\co N_i\to M$ and  $2$-form $\theta\in \Omega^{2}_{cl}(M;\Lambda_{\omega}^{0})$.  

In general, if $H\co S^1\times M\to\mathbb{R}$ is any Hamiltonian and if $\chi\co [0,1]\to [0,1]$ is a smooth monotone function with $\chi(0)=0$, $\chi(1)=1$, and $\chi'(0)=\chi'(1)$, define \[ H^{\chi}(t,m)=\chi'(t)H(\chi(t),m).\]    Note that contractible $1$-periodic orbits of $X_{H^{\chi}}$ are just reparametrizations of contractible $1$-periodic orbits of $X_H$, in view of which $H^{\chi}$ is strongly nondegenerate if and only if $H$ is.  Since $H$ and $H^{\chi}$ represent the same element of $\widetilde{Ham}(M,\omega)$,  if $H$ (hence also $H^{\chi}$) is normalized and strongly nondegenerate then the $\eta$-deformed Floer complexes of $H$ and $H^{\chi}$ are isomorphic as $\mathbb{R}$-filtered chain complexes by
Proposition \ref{filtiso}.

Fix a smooth monotone function $\zeta\co [0,1/2]\to [0,1]$ such that $\zeta(0)=0$, $\zeta(1/2)=1$, and $\zeta'$ vanishes to infinite order at both $t=0$ and $t=1/2$.  Given smooth functions $H,K\co S^1\times M\to \mathbb{R}$, define $H\lozenge K\co S^1\times M\to\mathbb{R}$ by \[ H\lozenge K(t,m)=\left\{\begin{array}{ll} \zeta'(t)H(\zeta(t),m) & 0\leq t\leq 1/2 \\ \zeta'(t-1/2)K(\zeta(t-1/2),m) & 1/2\leq t\leq 1\end{array}\right..\]
The choice of $\zeta$ ensures that $H\lozenge K$ is well-defined and smooth; also, if $H$ and $K$ are normalized then so is $H\lozenge K$.

Be given two Hamiltonians $H,K\co S^1\times M\to\mathbb{R}$.  For some small $\ep>0$, we assume that $H(t,\cdot)=K(t,\cdot)=0$ for $|t|\leq \ep$; this can be achieved by replacing $H,K$ with $H^{\chi},K^{\chi}$ as defined above for some $\chi\co [0,1]\to [0,1]$ as above whose derivative vanishes identically on $[0,\ep]\cup[1-\ep,1]$. We assume that $H,K,$ and $H\lozenge K$ are each strongly nondegenerate; a standard argument shows that this conditions holds for generic pairs $(H,K)$. We now explain the definition of the pair of pants product \[ \ast^{Floer}_{\eta}\co CF(H)\otimes CF(K)\to CF(H\lozenge K),\] carefully arranging the details so that the product will behave well with respect to the filtrations on the complexes. The construction of course closely resembles that in \cite[Section 4.1]{Sc00}, but we phrase it a bit differently, working always in terms of an explicit smooth trivialization of the relevant bundles in order to facilitate the introduction of incidence conditions.

Let $\Sigma$ denote the thrice-punctured sphere, and let $\Sigma_0\subset \Sigma$ be the compact submanifold with boundary obtained by deleting from $\Sigma$ small punctured-disc neighborhoods of the punctures.  

\begin{center}\begin{figure}\label{pantsfig}
\includegraphics[scale=0.9]{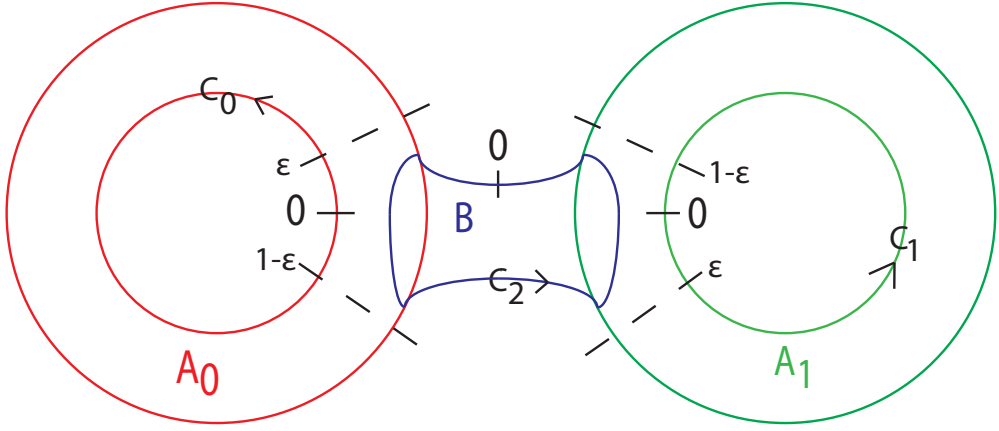} \caption{The surface $\Sigma_0$, decomposed into annuli $A_0,A_1$ and a square $B$.  The thrice-punctured sphere $\Sigma$ is obtained by adding cylindrical ends $(-\infty,0]\times C_0,(-\infty,0]\times C_1,[0,\infty)\times C_2$. }
\end{figure}
\end{center}

As in Figure 1, express $\Sigma_0$ as a union of disjoint annuli $A_0,A_1\cong [0,1]\times S^1$ and a square $B\cong [0,1]\times [0,1]$.  For $i=0,1$ let $t_i$ be the angular ($S^1=\mathbb{R}/\mathbb{Z}$) coordinates on the annuli $A_0,A_1$; values of the $t_i$ at certain points are indicated in the figure.  Consider the manifold $\Sigma_0\times M$ with projection $p_2\co \Sigma_0\times M\to M$, equipped with the symplectic form $\omega_0$ defined by \begin{align*} \omega_0|_{A_0\times M}&=p_{2}^{*}\omega-d(H(t_0,\cdot)dt_0) \\
\omega_0|_{A_1\times M}&=p_{2}^{*}\omega-d(K(t_1,\cdot)dt_1) \\ \omega_0|_{B\times M}&=p_{2}^{*}\omega.\end{align*}  Because of our assumption that $H(t,\cdot)=K(t,\cdot)=0$ for $|t|\leq \ep$, the above definitions of $\omega_0$ are consistent on the overlaps $(A_i\cap B)\times M$. 

For any parametrization of the third boundary component $C_2$ of $\Sigma_0$ by a coordinate $t_2\in\mathbb{R}/\mathbb{Z}$ with orientation and basepoint $t_2=0$ as in the figure, the restriction $\omega_0|_{C_2\times M}$ will take the form \[ \omega_0|_{C_2\times M}=p_{2}^{*}\omega+\beta(t_2)\wedge dt_2\] where $\beta\co \mathbb{R}/\mathbb{Z}\to \Omega^1(M)$.  By considering the horizontal distribution $(TM)^{\perp_{\omega_0}}$, one finds that we will have $\beta(t_2)=-d((H\lozenge K)^{\tilde{\chi}}(t,\cdot))$ for some reparametrizing function $\tilde{\chi}\co [0,1]\to [0,1]$; indeed, we may and do choose the parametrization of $C_2$ so that $\tilde{\chi}$ is the identity and so \[ \omega_0|_{C_2\times M}=\omega-d((H\lozenge K)(t_2,m)dt_2).\]  Accordingly, where we view the thrice-punctured sphere $\Sigma$ as obtained from $\Sigma_0$ by attaching cylindrical ends $\hat{C}_i=(-\infty,0]\times C_i$ ($i=0,1$, where $C_i=\{0\}\times S^1\subset A_i$) and $\hat{C}_2=[0,\infty)\times C_2$ to the three boundary components $C_0,C_1,C_2$ of $\Sigma_0$ we may define a closed $2$-form $\hat{\omega}\in \Omega^2(\Sigma\times M)$ by \begin{align*} \hat{\omega}|_{\Sigma_0\times M}&=\omega_0\\
\hat{\omega}|_{\hat{C}_0}&=p_{2}^{*}\omega-d(H(t,\cdot)dt)\\
\hat{\omega}|_{\hat{C}_1}&=p_{2}^{*}\omega-d(K(t,\cdot)dt)\\
\hat{\omega}|_{\hat{C}_2}&=p_{2}^{*}\omega-d((H\lozenge K)(t,\cdot)dt)
\end{align*}

Let $\mathcal{H}\subset T(\Sigma\times M)$ be the $\hat{\omega}$-orthogonal complement of the vertical bundle $TM\subset T(\Sigma\times M)$.  Thus we have a splitting $T(\Sigma\times M)=\mathcal{H}\oplus TM$; given $v\in T(\Sigma\times M)$ write $v^{vt}$ for the component of $v$ in the $TM$ summand.  Notice that it is a direct consequence of the construction of $\hat{\omega}$ that \begin{equation}\label{horvan}\mbox{If $v^{vt}=0$ then }\iota_v\hat{\omega}=0
\end{equation}

Let $j\co T\Sigma\to T\Sigma$ be the standard complex structure and choose a smooth family $\{J_z\}_{z\in \Sigma}$ of $\omega$-compatible almost complex structures on $M$, such that for $z=(s_i,t_i)$ on the ends $\hat{C}_i\cong (-\infty,0]\times S^1$ or $[0,\infty)\times S^1$ we have $J_{(s_i,t_i)}=J_{t_i}^{i}$ for some one-periodic families $J_{t}^{i}$ ($i=0,1,2$).  Given $u\co \Sigma\to M$, define $U\co \Sigma\to  \Sigma\times M$ by $U(z)=(z,u(z))$, and define the energy $E(u)$ as the integral of the $2$-form $e_u$ on $\Sigma$ whose value $e_u(z)\in \wedge^2T_{z}^{*}\Sigma$ at $z\in \Sigma$ is given by choosing a basis $\{e_1,e_2\}$ for $T_z\Sigma$ with $e_2=je_1$, letting $e^1,e^2\in T_{z}^{*}\Sigma$ be the dual basis, and putting \[ e_u(z)=\frac{1}{2}\left(\omega((U_*e_1)^{vt},J(U_*e_1)^{vt})+\omega((U_*e_2)^{vt},J(U_*e_2)^{vt})\right)e^1\wedge e^2.\]

Where again $U(z)=(z,u(z))$ for $u\co \Sigma\to M$, if we have, for each $z\in \Sigma,v\in T_{z}\Sigma$, \begin{equation}\label{crvt} J(U_*v)^{vt}=(U_*jv)^{vt},
\end{equation} then (using (\ref{horvan}))  if $e_1,e_2\in T_z\Sigma$ with $e_2=je_1$, we have \[ (U^*\hat{\omega})_z=\hat{\omega}(U_*e_1,U_*e_2)e^1\wedge e^2=\hat{\omega}((U_*e_1)^{vt},(U_*e_2)^{vt})e^1\wedge e^2=e_u(z),\] and hence \begin{equation}\label{posint}\int_{\Sigma}U^*\hat{\omega}=E(u)\geq 0 \mbox{ if $E(u)<\infty$ and (\ref{crvt}) holds}. 
\end{equation}

Now for a suitable zeroth order term $Y$, the equation (\ref{crvt}) can be rewritten directly as an equation for a map $u\co \Sigma\to M$ of the form \[ du+J_z(u)\circ du\circ j+Y(z,u)=0.\]  Along the cylindrical ends $\hat{C}_i$, one finds more specifically that (\ref{crvt})  is equivalent to a standard Floer equation:\begin{align*} \frac{\partial u}{\partial s}+J^{0}_{t}(u(s,t))\left(\frac{\partial u}{\partial t}-X_{H}(t,u(s,t))\right)=0& \mbox{ on }\hat{C}_0 \\ \frac{\partial u}{\partial s}+J^{1}_{t}(u(s,t))\left(\frac{\partial u}{\partial t}-X_{K}(t,u(s,t))\right)=0& \mbox{ on }\hat{C}_1 \\ \frac{\partial u}{\partial s}+J^{2}_{t}(u(s,t))\left(\frac{\partial u}{\partial t}-X_{H\lozenge K}(t,u(s,t))\right)=0& \mbox{ on }\hat{C}_2 .\end{align*}  

In particular, a finite-energy solution $u$ to (\ref{crvt}) will necessarily be asymptotic to some $\gamma_0\in P(H)$ as $s\to -\infty$ in $\hat{C}_0$; to some $\gamma_1\in P(K)$ as $s\to -\infty$ in $\hat{C}_1$; and to some $\gamma_2\in P(H\lozenge K)$ as $s\to +\infty$ in $\hat{C_2}$.  Given such $\gamma_0,\gamma_1,\gamma_2$, let $\pi_2(\gamma_0,\gamma_1;\gamma_2)$ denote the set of relative homotopy classes  of maps $u\co \Sigma\to M$ with these asymptotics.  If for $i=0,1$ we choose discs $v_i\co D^2\to M$ with $v_i|_{\partial D^2}=\gamma_i$, for any $P\in \pi_2(\gamma_0,\gamma_1;\gamma_2)$ we obtain a relative homotopy class of discs $v_0\#P\#v_1$ bounding $\gamma_2$ by gluing in the obvious way.  

\begin{lemma}\label{actions} If $u\co \Sigma\to M$ represents $P\in \pi_2(\gamma_0,\gamma_1;\gamma_2)$ then where $U(z)=(z,u(z))$ we have \[ \int_{\Sigma}U^*\hat{\omega}=\mathcal{A}_H([\gamma_0,v_0])+\mathcal{A}_K([\gamma_1,v_1])-\mathcal{A}_{H\lozenge K}([\gamma_2,v_0\#P\#v_1]).\]
\end{lemma}

\begin{proof} Cap off the cylindrical ends $\hat{C}_i$ of $\Sigma$ by discs $D_i$ (identified with $\{z\in \mathbb{C}\cup\{\infty\}||z|\leq 1\}$ for $i=0,1$ and with $\{z\in \mathbb{C}\cup\{\infty\}||z|\geq 1\}$ for $i=2$) to form a copy of $S^2$.  Extend the form $\hat{\omega}$ on $\Sigma\times M$ to a form $\tilde{\omega}$ on $S^2\times M$ by, where $(r,2\pi t)$ are polar coordinates and $\beta\co [0,1]\to [0,1]$ is a smooth monotone function equal to $0$ for $s<1/3$ and to $s>2/3$, putting $\tilde{\omega}|_{D_0\times M}=\omega-d(\beta(r^2)H(t,\cdot)dt)$, $\tilde{\omega}|_{D_1\times M}=\omega-d(\beta(r^2)K(t,\cdot)dt)$, and $\tilde{\omega}|_{D_2\times M}=\omega-d(\beta(r^{-2})H\lozenge K(t,\cdot)dt)$.

Now for $p\in M$ the map $f_p\co S^2\to S^2\times M$ defined by $f_p(z)=(z,p)$ is easily seen to have \[ \int_{S^2}f_{p}^{*}\tilde{\omega}=\int_{0}^{1}H(t,p)dt+\int_{0}^{1}K(t,p)dt-\int_{0}^{1}(H\lozenge K)(t,p)dt=0.\] 
Let $v_2\co D_2\to M$ be any map such that, where $I\co D^2\to D_2$ is the orientation-reversing diffeomorphism $re^{2\pi it}\mapsto r^{-1}e^{2\pi it}$, the composition $v_2\circ I\co D^2\to M$ represents  the relative homotopy class $v_0\#P\#v_1$ in $\pi_2(\gamma_2)$.
The map $\tilde{U}\co S^2\to S^2\times M$ obtained by combining the map $U$ on $\Sigma$ with the maps $V_i\co z\mapsto (z,v_i(z))$ on $D_i$ for $i=0,1,2$ 
has its projection to $M$ homotopic to $0$, so \[ \int_{S^2}\tilde{U}^{*}\tilde{\omega}=\int_{S^2}f_{p}^{*}\tilde{\omega}=0.\]  But \begin{align*}
\int_{S^2}\tilde{U}^{*}\tilde{\omega}&=\int_{\Sigma}U^*\hat{\omega}+\int_{D_0}V_{0}^{*}(\omega-d(\beta(r^2)H(t,\cdot)dt))+\int_{D_1}V_{1}^{*}(\omega-d(\beta(r^2)K(t,\cdot)dt))\\                          &\quad +\int_{D_2}V_{2}^{*}(\omega-d(\beta(r^{-2})H\lozenge K(t,\cdot)dt))
\\&=\int_{\Sigma}U^*\hat{\omega}-\mathcal{A}_H([\gamma_0,v_0])-\mathcal{A}_K([\gamma_1,v_1])+\mathcal{A}_{H\lozenge K}([\gamma_2,v_0\#P\#v_1]),\end{align*} as follows from an application of Stokes' theorem.  \end{proof}

Given $P\in \pi_2(\gamma_0,\gamma_1;\gamma_2)$, choose an arbitrary representative $u$ of $P$ and arbitrary capping discs $v_0,v_1\co D^2\to M$ for $\gamma_0,\gamma_1$ and set \[ \bar{\mu}(P)=\mu_{CZ}(\gamma_2,v_0\#u\#v_1)-\mu_{CZ}(\gamma_0,v_0)-\mu_{CZ}(\gamma_1,v_1)-n.\]  (This is easily seen to be independent of the choices of $v_0,v_1$).  By \cite[Theorem 3.3.11]{Sc95}, at any finite energy solution $u$ of (\ref{crvt}) which represents the class $P\in \pi_2(\gamma_0,\gamma_1;\gamma_2)$, the linearization of (\ref{crvt}) is Fredholm of index $\bar{\mu}(P)$.  Thus in the usual way one obtains a Kuranishi structure with corners of dimension $\bar{\mu}(P)+2k$ on the compactified moduli space $C\mathcal{M}_k(\gamma_0,\gamma_1,\gamma_2,P)$ of such solutions with $k$ marked points, and, for $I=(i_1,\ldots,i_k)$ a Kuranishi structure with corners of dimension $\bar{\mu}(P)-\delta(I)$ on the fiber product \[ \mathcal{M}(\gamma_0,\gamma_1,\gamma_2,P;N_I)=C\mathcal{M}_k(\gamma_0,\gamma_1,\gamma_2,P)_{(ev_1,\ldots,ev_k)}\times_{(f_{i_1},\ldots,f_{i_k})}(N_{i_1}\times\cdots\times N_{i_k}).\]  Then where the associated multisections are denoted $\frak{s}_{P,I}$,  the pair of pants product \[ \ast^{Floer}_{\eta}\co CF(H)\otimes CF(K)\to CF(H\lozenge K) \] is defined by extending linearly from \[ [\gamma_0,v_0]\ast^{Floer}_{\eta}[\gamma_1,v_1]=\sum_{k=0}^{\infty}\frac{1}{k!}\sum_{\substack{\gamma_2\in P(H\lozenge K),\\P\in \pi_2(\gamma_0,\gamma_1;\gamma_2)}}\sum_{\substack{I\in \{1,\ldots,m\}^k,\bar{\mu}(P)=\delta(I)}}|\frak{s}_{P,I}^{-1}(0)|\exp\left(\int_P\theta\right)z_I[\gamma_2,v_0\#P\#v_1].\]  We have: 
\begin{prop}\label{prodfilt}\begin{itemize}\item[(i)] $\ast^{Floer}_{\eta}$ is a chain map (with respect to the differential $\partial^{\eta,H}\otimes 1+(-1)^{|\cdot|}1\otimes \partial^{\eta,K}$ on the domain and $\partial^{\eta,H\lozenge K}$ on the range).
\item[(ii)] For $\lambda,\mu\in\mathbb{R}$, $\ast^{Floer}_{\eta}$ restricts as a map \[ \ast^{Floer}_{\eta}\co CF^{\lambda}(H)\otimes CF^{\mu}(K)\to CF^{\lambda+\mu}(H\lozenge K).\]\end{itemize}
\end{prop}

\begin{proof}  The first item follows by a standard gluing argument as in \cite{Sc95}, combined with the same analysis of incidence conditions as in the proof that $(\partial^{\eta,H})^{2}=0$.  The second item follows from (\ref{posint}) and Lemma \ref{actions}: indeed, 
$\mathcal{M}(\gamma_0,\gamma_1,\gamma_2,P;N_I)$ is empty unless there is a representative of $P$ which obeys (\ref{crvt}), in which case (\ref{posint}) and Lemma \ref{actions} show that \[ \mathcal{A}_{H\lozenge K}([\gamma_2,v_0\#P\#v_1])\leq \mathcal{A}_H([\gamma_0,v_0])+\mathcal{A}_H([\gamma_1,v_1]).\]  The conclusion then follows directly from the definitions of $*_{\eta}$ and of the filtration.
\end{proof}

Now consider, for generic Morse functions $f,g,h$, the composition \[ \xymatrix{CM(f;\Lambda_{\omega})\otimes CM(g;\Lambda_{\omega})\ar[rr]^{\Phi^{PSS}_{\eta,H}\otimes \Phi^{PSS}_{\eta,K}} &  &CF(H)\otimes CF(K)\ar[dll]_{\ast^{Floer}_{\eta}} \\ CF(H\lozenge K)
\ar[rr]_{\Psi^{PSS}_{\eta,H\lozenge K}} &  &  CM(h;\Lambda_{\omega})}  \]

A gluing argument shows that this map $\ast^{Morse}_{\eta}=\Psi^{PSS}_{\eta,H\lozenge K}\circ \ast^{Floer}_{\eta}\circ (\Phi^{PSS}_{\eta,H}\otimes \Phi^{PSS}_{\eta,K})$ is obtained by extending linearly from (for $p\in Crit(f),q\in Crit(g)$) \[ p\ast^{Morse}_{\eta} q=\sum_{k=0}^{\infty}\frac{1}{k!}\sum_{\substack{A\in \pi_2(M),\\r\in Crit(h)}}\sum_{\substack{I\in \{1,\ldots,m\}^k,\\2c_1(A)+ind_f(p)+ind_g(q)-ind_h(r)=2n+\delta(I)}}n_L(p,q;r,A,I)\exp\left(\int_{A}\theta\right)z_IT^{\int_A\omega}r.\]  Here $n_L(p,q;r,A,I)$ enumerates solutions $u\co S^2\to M$ to an equation $\delbar_{J,L}u=0$ for a suitable Hamiltonian perturbation $L$, which represent the class $A\in \pi_2(M)$; pass through the unstable manifolds $W^u(p;f), W^u(q;g)$ and the stable manifold $W^s(r;h)$; and additionally satisfy incidence conditions corresponding to $I=(i_1,\ldots,i_k)$.

Now, a cobordism argument shows that the chain homotopy class of such a map is independent of the Hamiltonian perturbation $L$, so the map on homology is unaffected if we replace $L$ in the above description by $0$.  The homologies of the Morse complexes $CM(\cdot;\Lambda_{\omega})$ are of course canonically isomorphic to $H_*(M;\Lambda_{\omega})$, and so arguing as in \cite[Section 5]{PSS} we find that, on homology, the map $\ast^{Morse}_{\eta}$ induces the map $H_*(M;\Lambda_{\omega})\otimes H_*(M;\Lambda_{\omega})\to H_*(M;\Lambda_{\omega})$ given by, where $\{c_j\}_{j=1}^{b}$ is a basis for $H_*(M;\Lambda_{\omega})$ with dual basis $\{c^j\}$,
\begin{align*} 
a\otimes b &\mapsto & \sum_{k=0}\frac{1}{k!}\sum_{j=1}^{b}\sum_{A\in \pi_2(M)}\sum_{I\in\{1,\ldots,m\}^k}\langle a,b,c_j,z_{i_1}c_{i_1},\ldots,z_{i_k}c_{i_k}\rangle_{0,k+3,A}\exp\left(\int_C\theta\right)T^{\int_{A}\omega}c^j\\& &= \sum_{k=0}\frac{1}{k!}\sum_{j=1}^{b}\sum_{A\in \pi_2(M)}\langle a,b,c_j,\eta,\ldots,\eta\rangle_{0,k+3,A}T^{\int_{A}\omega}c^j,\end{align*} recalling that by definition $\eta=PD[\theta]+\sum_{i=1}^{m}z_ic_i$.  But the last formula above is precisely that for the $\eta$-deformed quantum product $\ast_{\eta}$ on $H_*(M;\Lambda_{\omega})$.  This proves\footnote{To be precise we are implicitly using here the fact that the map $H_*(M;\Lambda_{\omega})\to HF(H)$ induced on homology by the PSS chain map is independent of the Morse function used; this can be established by a standard continuation-type argument that is left to the reader.}:

\begin{prop}\label{ringiso} Where $HF(H)_{\eta}$ denotes the homology of the complex $(CF(H),\partial^{\eta,H})$, the map $\underline{\Phi}^{PSS}_{\eta,H}\co H_*(M;\Lambda_{\omega})\to HF(H)_{\eta}$ on homology induced by  $\Phi_{\eta,H}^{PSS}\co CM(f;\Lambda_{\omega})\to CF(H)$ is an isomorphism of $\Lambda_{\omega}$-modules.  Furthermore, where $\underline{\ast}^{Floer}_{\eta}$ is induced on homology by $\ast_{\eta}^{Floer}$, we have a commutative diagram \[  \xymatrix{ H_*(M;\Lambda_{\omega})\otimes H_*(M;\Lambda_{\omega})\ar[d]_{\ast_{\eta}}\ar[rr]^{\underline{\Phi}^{PSS}_{\eta,H}\otimes \underline{\Phi}^{PSS}_{\eta,K}}& & HF(H)_{\eta}\otimes HF(K)_{\eta} \ar[d]^{\underline{\ast}^{Floer}_{\eta}} \\ H_*(M;\Lambda_{\omega})\ar[rr]^{\underline{\Phi}^{PSS}_{\eta,H\lozenge K}} & & HF(H\lozenge K)_{\eta}
} \] 
\end{prop}

(Of course, the first sentence already follows from Proposition \ref{pssprop}(iii), and shows that, module-theoretically but not ring-theoretically, $HF(H)_{\eta}$ is independent of both $H$ and $\eta$.)

\subsection{Spectral invariants} \label{specsec}

Given $\eta\in \oplus_{i=0}^{n-1}H_{2i}(M;\Lambda_{\omega}^{0})$ and a strongly nondegenerate Hamiltonian $H\co S^1\times M\to\mathbb{R}$, 
we have the chain complex $(CF(H),\partial^{\eta,H})$ and a PSS isomorphism $\underline{\Phi}^{PSS}_{\eta,H}\co H_*(M;\Lambda_{\omega})\to HF(H)_{\eta}$ to the homology $HF(H)_{\eta}$ of $(CF(H),\partial^{\eta,H})$.  Accordingly, for such data $H,\eta$, and for any class $a\in H_*(M;\Lambda_{\omega})\setminus \{0\}$, we may define the $\eta$-deformed spectral invariant:
\[ \rho(a;H)_{\eta}=\inf\{\ell(c)|c\in CF(H),[c]=\underline{\Phi}^{PSS}_{\eta,H}a\in HF(H)_{\eta}\}.\]  

The finiteness of $\rho(a;H)_{\eta}$ follows from Proposition \ref{rhoprop} (iii) below (or, on more general grounds, one could appeal to \cite[Theorem 1.3]{U08}).

Let us also introduce some notation pertaining to $H_*(M;\Lambda_{\omega})$.  First a general element $a\in H_*(M;\Lambda_{\omega})$ can be written as $a=\sum_{g\in \Gamma_{\omega}}a_gT^g$ where each $a_g\in H_*(M;\mathbb{C})$, and we put \[ \nu(a)=\max\{-g|a_g\neq 0\} \] (just as was done at the chain level in Proposition \ref{pssprop} (iv)).  Also, for $a=\sum_{g\in \Gamma_{\omega}}a_gT^g,b=\sum_{g\in \Gamma_\omega} b_g T^g$, put \[ \Pi(a,b)=\sum_{g\in \Gamma_{\omega}}a_g\cap b_{-g} \] where $\cap$ denotes the Poincar\'e intersection pairing.  It follows from standard properties of Gromov--Witten invariants that \[ \Pi(a\ast_{\eta}b,[M])=\Pi(a,b),\] independently of $\eta$.

Recall also the definition of the Hofer norm on the space of continuous functions $H\co S^1\times M\to\mathbb{R}$: \[ \|H\|=\int_{0}^{1}\left(\max H(t,\cdot)-\min H(t,\cdot)\right)dt.\]  

\begin{prop}\label{rhoprop} The spectral invariants $\rho(\cdot;\cdot)_{\eta}$ have the following properties, for any $a,b\in H_*(M;\Lambda_{\omega})\setminus\{0\}$:\begin{itemize}\item[(i)] If $r\co S^1\to \mathbb{R}$ is a smooth function then \[ \rho(a;H+r)_{\eta}=\rho(a;H)_{\eta}+\int_{0}^{1}r(t)dt,\] where $(H+r)(t,m)=H(t,m)+r(t)$. 

\item[(ii)] If $H$ and $K$ are both strongly nondegenerate, then \[ \rho(a;H)_{\eta}-\rho(a;K)_{\eta}\leq \int_{0}^{1}\max_{M}(H-K)(t,\cdot) dt.\]  Consequently if 
$H$ and $K$ are both normalized then \[ |\rho(a;H)_{\eta}-\rho(a;K)_{\eta}|\leq \|H-K\|,\] and so the function $\rho(a;\cdot)_{\eta}$ extends by continuity (with respect to $\|\cdot\|$) to the set of continuous $H\co S^1\times M\to\mathbb{R}$ such that each $H(t,\cdot)$ has mean zero for all $t$, and then to all continuous $H\co S^1\times M\to\mathbb{R}$ by (i) above.

\item[(iii)] $\nu(a)+\int_{0}^{1}\min H(t,\cdot)dt\leq \rho(a;H)_{\eta}\leq \nu(a)+\int_{0}^{1}\max H(t,\cdot)dt$ for all $H\in C(S^1\times M,\mathbb{R})$.

\item[(iv)] If $H$ is strongly nondegenerate, there is $c\in CF(H)$ such that $\rho(a;H)_{\eta}=\ell(c)$.

\item[(v)] If $H$ and $K$ both represent the same element in $\widetilde{Ham}(M,\omega)$ and are both normalized, then \[ \rho(a;H)_{\eta}=\rho(a;K)_{\eta}.\]

\item[(vi)] For any $H,K\in C(S^1\times M,\mathbb{R})$,
\[ \rho(a\ast_{\eta}b;H\lozenge K)_{\eta}\leq \rho(a;H)_{\eta}+\rho(b;K)_{\eta}.\]

\item[(vii)] Where $\bar{H}(t,m)=-H(t,\phi_{H}^{t}(m))$ (so that $\phi_{\bar{H}}^{t}=(\phi_{H}^{t})^{-1}$), we have \[ \rho(a;\bar{H})=-\inf\{\rho(x;H)|\Pi(x,a)\neq 0\}.\]

\item[(viii)] If $\phi\in Symp(M,\omega)$ is any symplectomorphism then $\rho(\phi_*a;H\circ \phi^{-1})_{\phi_*\eta}=\rho(a;H)_{\eta}$.

\end{itemize}
\end{prop}
 
\begin{proof} Since, in light of the results proven so far, most of these properties follow by straightforwardly  adapting arguments that are well-known in the $\eta=0$ case (see, \emph{e.g.}, \cite{Oh06}), we just briefly indicate the ingredients of the proofs.

(i) is an immediate consequence of the definitions, since replacing $H$ by $H+r$ does not affect the Floer differential or the PSS map, and affects the function $\ell$ by addition of $\int_{0}^{1}r(t)dt$.
The first sentence of (ii) follows from Proposition \ref{pssprop}(ii) combined with the estimate (\ref{filtcont}) applied to the continuation map $\Phi_{\tilde{H},0,\eta}$ with $\tilde{H}(s,t,m)=\beta(s)K(t,m)+(1-\beta(s))H(t,m)$ where $\beta\co \mathbb{R}\to [0,1]$ is smooth and monotone with $\beta(s)=0$ for $s<-1$ and $\beta(s)=1$ for $s>1$.  The second sentence of (ii) follows from the first by interchanging the roles of $H$ and $K$ and then using the fact that a mean-zero function cannot have a positive global minimum.

(iii) follows directly from Proposition \ref{pssprop} (iv) and (v), together with an approximation argument via (ii) in case $H$ is not strongly nondegenerate (the first inequality also uses that $\underline{\Psi}_{\eta,H}^{PSS}\circ\underline{\Phi}_{\eta,H}^{PSS}$ is the identity).

(iv) is a special case of the main result of \cite{U08}.

As for (v), by continuity we may reduce to the case that $H$ and $K$ are both strongly nondegenerate.  But in that case the statement follows directly from Proposition \ref{filtiso} together with the naturality statement Proposition \ref{pssprop}(ii).

In (vi), by continuity we may assume that $H,K$, and $H\lozenge K$ are all strongly nondegenerate (since generic pairs $(H,K)$ have this property). Moreover by (v) we can reduce to the case that $H(t,\cdot)=K(t,\cdot)=0$ for $|t|<\ep$ where $\ep>0$.     But in that case the result follows from Propositions \ref{prodfilt} and \ref{ringiso}.

Now consider (vii).  The pairing $\Pi\co H_*(M;\Lambda_{\omega})\otimes H_*(M;\Lambda_{\omega})\to \mathbb{C}$ is, for a suitably generic Morse function $f$, induced on homology by the pairing  \begin{align*} \Pi^{Morse}\co CM(-f;\Lambda_{\omega})\otimes CM(f;\Lambda_{\omega})&\to \mathbb{C} \\ \left(\sum_{\substack{g\in \Gamma_{\omega},\\p\in Crit(f)}}a_{g,p}T^gp\right)\otimes \left(\sum_{\substack{g\in \Gamma_{\omega},\\q\in Crit(f)}}b_{g,q}T^gq\right)&\mapsto\sum_{\substack{g\in\Gamma_{\omega},\\p\in Crit(f)}}a_{-g,p}b_{g,p}.\end{align*}  (Of course we use here the fact that $Crit(f)=Crit(-f)$.) 

Meanwhile, since elements of $\tilde{P}(\bar{H})$ are precisely obtained from elements of $\tilde{P}(H)$ by orientation reversal, and since formal negative gradient flowlines of $\mathcal{A}_{\bar{H}}$ are equivalent to formal negative gradient flowlines of $\mathcal{A}_H$ with both $s$ and $t$ coordinates reversed\footnote{hence the integrals of $2$-forms are the same, not opposite to each other, over the corresponding flowlines} the Floer complex $(CF(\bar{H}),\partial^{\eta,\bar{H}})$ is the opposite complex to the Floer complex $(CF(H),\partial^{\eta,H})$ in the sense defined in \cite{U10b}.  For $[\gamma,v]\in P(H)$ write $[\bar{\gamma},\bar{v}]\in \tilde{P}(\bar{H})$ for the generator obtained by reversing the orientations of both $\gamma$ and $v$.  Then, as in \cite[Section 1.4]{U10b}, we have a pairing \[ \Pi^{Floer}\co CF(\bar{H})\otimes CF(H)\to\mathbb{C} \] defined by \[ \Pi^{Floer}\left(\sum_{[\gamma,v]\in\tilde{P}(H)} a_{[\bar{\gamma},\bar{v}]}[\bar{\gamma},\bar{v}],\sum_{[\delta,w]\in\tilde{P}(H)} b_{[\delta,w]}[\delta,w]\right)=\sum_{[\gamma,v]\in\tilde{P}(H)}a_{[\bar{\gamma},\bar{v}]}b_{[\gamma,v]}.\]

Now the spiked planes counted by the map $\Phi_{\eta,\bar{H}}\co CM(-f;\Lambda_{\omega})\to CF(\bar{H})$ are equivalent to those counted by the map $\Psi_{\eta,H}\co CF(\bar{H})\to CM(f;\Lambda_{\omega})$, in view of which one obtains the adjoint relation \[ \Pi^{Morse}(d,\Psi_{\eta,H}c)=\Pi^{Floer}(\Phi_{\eta,\bar{H}}d,c)\quad (c\in CF(H),d\in CM(-f;\Lambda_{\omega})).\]  Consequently, where $\underline{\Pi}^{Floer}\co HF(\bar{H})_{\eta}\otimes HF(H)_{\eta}\to\mathbb{C}$ is the pairing on homology induced by $\Pi^{Floer}$, we find, for $x,a\in H_*(M;\Lambda_{\omega})$, \[ \Pi(x,a)=\Pi(x,\underline{\Psi}_{\eta,H}(\underline{\Phi}_{\eta,H}a))=\underline{\Pi}^{Floer}(\underline{\Phi}_{\eta,\bar{H}}x,\underline{\Phi}_{\eta,H}a).\]  Consequently (vii) follows from the definition of $\rho$ and \cite[Corollary 1.3]{U10b}.

Finally, (viii) is a consequence of standard naturality properties of the construction of $\rho$.
\end{proof}

For $\tilde{\phi}\in \widetilde{Ham}(M,\omega)$, $a\in H_*(M;\Lambda_{\omega})\setminus\{0\}$, $\eta\in \oplus_{i=0}^{n-1}H_{2i}(M;\Lambda_{\omega}^{0})$, define \[ c(a;\tilde{\phi})_{\eta}=\rho(a;H)_{\eta}\mbox{ where $H\co S^1\times M\to\mathbb{R}$ is normalized and $\tilde{\phi}=\tilde{\phi}_H$}.\]  (By Propositition \ref{rhoprop} (v) the right hand side is independent of the choice of $H$ with $\tilde{\phi}_H=\tilde{\phi}$.)

\begin{prop}\label{disp}  Suppose that $H\co S^1\times M\to\mathbb{R}$ is a Hamiltonian with support contained in a set of form $S^1\times S$ where $S\subset M$ is a displaceable compact subset (i.e., there is $K\co S^1\times M\to\mathbb{R}$ with $\phi_{K}^{1}(S)\cap S=\varnothing$).  Suppose also that $a\in H_*(M;\Lambda_{\omega})$ obeys $a\ast_{\eta}a=a$.  Then \[ \lim_{k\to\infty}\frac{c(a;\tilde{\phi}_{H}^{k})_{\eta}}{k}=\frac{-\int_{0}^{1}\int_M H(t,\cdot)\omega^n}{\int_M\omega^n}.\]
\end{prop}

\begin{proof}  Given Proposition \ref{rhoprop}, the proof is essentially identical to that of \cite[Proposition 3.3]{EP03}.  Namely, if $K$ is as in the statement and is strongly nondegenerate (as we can arrange, since the condition $\phi_{K}^{1}(S)\cap S=\varnothing$ is an open one on $K$) one finds using Proposition \ref{rhoprop} (ii),(iv) that $\rho(a;(sH^{\lozenge k})\lozenge K)_{\eta}$ is independent of $s$, and hence that (by Proposition \ref{rhoprop} (i)) \[ c(a;\tilde{\phi}_K\circ\tilde{\phi}_{H}^{k})_{\eta}=c(a;\tilde{\phi}_K)_{\eta}-k\frac{\int_{0}^{1}\int_M H(t,\cdot)\omega^n}{\int_M\omega^n}.\]  Since the triangle inequality Proposition \ref{rhoprop}(vi) shows \[ |c(a;\tilde{\phi}_K\circ\tilde{\phi}_{H}^{k})_{\eta}-c(a;\tilde{\phi}_{H}^{k})_{\eta}|\leq c(a;\tilde{\phi}_K)_{\eta}+c(a;\tilde{\phi}_{K}^{-1})_{\eta},\] the result follows. 
\end{proof}

Typically, the spectral invariants $\rho(a;H)_{\eta}$ are difficult to compute.  However, we will now discuss an important exception.  Let $H\co M\to\mathbb{R}$ be a smooth function (which we view as a time-independent Hamiltonian).  Following the terminology in \cite[Definition 4.3]{U10a}, we will call $H$ \emph{slow} if all contractible periodic orbits of the Hamiltonian vector field $X_H$ having period at most $1$ are constant, and \emph{flat} if, at all critical points $p\in Crit(H)$, every periodic orbit of period at most $1$ of the linearized flow $(\phi_{H}^{t})_*\co T_pM\to T_pM$ is constant.

Also, given a class $a\in H_*(M;\Lambda_{\omega})$, write \[ a=\left(\sum_{g\in \Gamma_{\omega}}a_gT^g\right)[M]+a'\] where $a'\in \oplus_{i=0}^{2n-1}H_i(M;\Lambda_{\omega})$, and set \[ \nu_{[M]}(a)=\max\{-g|a_g\neq 0\} \] (said differently, $\nu_{[M]}(a)=\nu(a-a')$).

The following generalizes \cite[Theorem IV]{Oh05} and \cite[Proposition 4.1]{U10a}, which apply when $\eta=0$ and $a=[M]$. The result plays an important role in the results of the following Section \ref{capsec}, though it is not used for the other main theorems of this paper.

\begin{theorem}\label{slowthm}  Let $H\co M\to\mathbb{R}$ be a slow autonomous Hamiltonian, and $a\in H_*(M;\Lambda_{\omega})\setminus\{0\}$.  Then \[ \rho(a;H)_{\eta}\geq \max H+\nu_{[M]}(a).\]  In particular,\[ \mbox{if $\nu_{[M]}(a)=\nu(a)$, then }\rho(a;H)_{\eta}=\max H+\nu(a).\]
\end{theorem}

\begin{proof} The second sentence follows from the first together with Proposition \ref{rhoprop}(iii).

By \cite[Theorem 4.5]{U10a}, our slow autonomous Hamiltonian $H$ may be arbitrarily well-approximated in $C^0$ norm by an autonomous Hamiltonian which is slow, flat, and Morse; a further arbitrarily small perturbation yields an autonomous Hamiltonian which is additionally strongly nondegenerate (which, given the ``slow'' property, just means that its critical points miss the images of the fixed maps $f_i\co N_i\to M$).  So by the continuity statement Proposition \ref{rhoprop}(ii), it suffices to prove the theorem under the assumption that $H$ is slow, flat,  Morse, and strongly nondegenerate.

Since $H$ is slow, $P(H)$ consists of the critical points of $H$, and for $q\in P(H)=Crit(H)$ the set $\pi_2(q)$ of relative homotopy classes of discs with boundary mapping to $q$ can be identified with $\pi_2(M)$.  Since $H$ is flat, for $A\in \pi_2(M)$ the corresponding element $A_q\in \pi_2(q)$ has $\bar{\mu}(A_q)=n-ind_H(q)+2c_1(A)$.  

Choose a Morse function $f\co M\to\mathbb{R}$ having a unique local (and global) maximum at a point $p_0$ such that $H(p_0)=\max H$, and consider the PSS map \[ \Phi_{\eta,H}^{PSS}\co CM(f;\Lambda_{\omega})\to CF(H),\] constructed using a family of almost complex structures $\{J(re^{i\theta})\}_{re^{i\theta}\in\mathbb{C}}$ which is independent of the angular coordinate $\theta$.  This map is given by, for $p\in Crit(f)$, \begin{equation}\label{recpss} \Phi_{\eta,H}^{PSS}p= \sum_{k=0}^{\infty}\frac{1}{k!}\sum_{\substack{q\in Crit(H),A\in \pi_2(M)}}\sum_{\substack{{I}\in \{1,\ldots,m\}^k,\\ind_f(p)-ind_H(q)+2c_1(A)=\delta(I)}}|\frak{s}_{A,p,I}^{-1}(0)|\exp\left(\int_C\theta\right)z_I[p,A] \end{equation}

Recall here that $\frak{s}_{A,p,I}$ is the multisection associated to the Kuranishi structure on the fiber product \[  \mathcal{M}^{PSS}_k(q,A;p,N_I)=C\mathcal{M}^{PSS}_{k}(q,A)_{(e_0,ev_1,\ldots,ev_k)}\times_{(j_p,f_1,\ldots,f_k)}(W^{u}(p)\times N_{i_1}\times\cdots\times N_{i_k}) \]
 where $C\mathcal{M}^{PSS}_{k}(q,A)$ is the compactification of the space of perturbed-holomorphic planes with $k$ marked points representing $A$ and asymptotic to $q$; thus a general element of $C\mathcal{M}^{PSS}_{k}(q,A)$ has one plane component and possibly a variety of other components, each of which is either a cylindrical solution to the Floer equation for $H$ or is a $J(r)$-holomorphic sphere for some  $r$.
 
 Now $\mathcal{M}^{PSS}_k(q,A;p,N_I)$ has an $S^1$-action given by angular rotation of each of the plane and cylinder components of the domain.  Just as in \cite[Mainlemma 22.4]{FO}, the fixed locus of this action is isolated; away from the fixed locus the action is locally free, and so one can construct a Kuranishi structure of dimension $-1$ on the quotient of the complement of the fixed locus and lift it to $\mathcal{M}^{PSS}_k(q,A;p,N_I)$ in order to deduce that, for suitable choices of the multisection $\frak{s}_{A,p,I}$, the quantity $|\frak{s}_{A,p,I}^{-1}(0)|$ receives contributions only from the fixed locus of the $S^1$ action on $\mathcal{M}^{PSS}_k(q,A;p,N_I)$ (i.e. the locus for which all planar or cylindrical components are independent of the angular coordinate).  An object in this fixed locus is equivalent to the data of: a solution $x\co [0,\infty)\to M$ of $rx'(r)=-\beta(r)\nabla H(x(r))$ such that $x(0)\in W^u(p;f)$ (corresponding to the $\mathbb{C}$-component); some number of negative gradient flowlines of $H$  (corresponding to the cylindrical components); and some number of spherical bubbles; all subject to appropriate incidence conditions.  Now a dimension count\footnote{This count uses the fact that the appropriate expected dimension is $ind_f(p)-ind_H(q)+2c_1(A)-\delta(I)$, which depends on the fact that $H$ is flat; if $H$ were not flat then some of the conclusions that we are drawing here would be false.}  shows that, after perturbing by a generic multisection, all of the spherical components will be constant (which forces $A=0$ since the non-spherical components are $S^1$-independent), that we must have $k=0$ (so $I=\varnothing$; the point is that now that the spiked planes map to  $1$-dimensional objects in $M$ they will not satisfy nontrivial incidence conditions) and that there will be no cylindrical components.  
 
 Consequently (\ref{recpss}) simplifies to \[ \Phi_{PSS,\eta}p=\sum_{q:ind_f(p)-ind_H(q)=0}m(p,q)[q,0],\] where the number $m(p,q)$ enumerates pairs consisting of: a negative gradient flowline $y\co (-\infty,0]\to M$ for $f$ with $y(t)\to p$ as $t\to -\infty$; and a solution $x\co [0,\infty)\to M$ to $rx'(r)=-\beta(r)\nabla H(x(r))$ with $x(r)\to q$ as $r\to \infty$ and $x(0)=y(0)$.
 
 Now if $p=p_0$ (the common global maximum of $f$ and $H$) the only such pair $(y,x)$ has both $y$ and $x$ equal to the constant map to $p_0$, and as in \cite[p. 14]{Oh05} we obtain $m(p_0,p_0)=1$.  On the other hand if $p\neq p_0$, since the only solution $x\co [0,\infty)\to M$ with $x(r)\to p_0$ as $r\to\infty$ is the constant map to $p_0$, if $m(p,p_0)\neq 0$ we would require a negative gradient flowline $y$ for $f$ asymptotic at $-\infty$ to $p$ with $y(0)=p_0$, which is impossible since $p_0$ is a maximum for $f$.  Thus \[ m(p,p_0)=\left\{\begin{array}{ll}1 & p=p_0 \\ 0 & p\neq p_0\end{array}\right.\]
 
 We now prove the theorem.  If $\nu_{[M]}(a)=-\infty$ the statement is vacuous, so assume $\nu_{[M]}(a)\in\mathbb{R}$.   Since $p_0$ is the unique local (and global) maximum of $f$, the class $a$ is represented by an element of the form \[ T^{-\nu_{[M]}(a)}\lambda_0p_0+\sum_{\substack{p\in Crit(f)},\\p\neq p_0}\mu_p p,\] where $\lambda_0=\sum_{g\in \Gamma_{\omega}}a_gT^g\in \Lambda_{\omega}^{0}$ has $a_0\neq 0$, and where $\mu_p\in \Lambda_{\omega}$.  Consequently, $\underline{\Phi}^{PSS}_{\eta,H}a\in HF(H)_{\eta}$ is represented by an element of the form \[ c=T^{-\nu_{[M]}(a)}\lambda_0[p_0,0]+\left(\begin{array}{cc}\mbox{terms involving $q\in P(H)$}\\ \mbox{with $q\neq p_0$}\end{array}\right).\]  Evidently we have $\ell(c)\geq \nu_{[M]}(a)+H(p_0)$.  
 
 Now any representative of $\underline{\Phi}^{PSS}_{\eta,H}a$ will take the form $c+\partial^{\eta,H}d$ for some $d\in CF(H)$.  Just as in the proof of \cite[Theorem 5.1]{Oh05} (and similarly to the situation with the PSS map above), for a suitable choice of multisection the differential $\partial^{\eta,H}$ will only receive contributions from Floer cylinders which are independent of the $S^1$-variable, i.e. from negative gradient flowlines of $H$.  But (other than a constant flowline, which has the wrong index) there are no negative gradient flowlines for $H$ asymptotic at $+\infty$ to the maximum $p_0$, and so the coefficient on $p_0$ in $\partial^{\eta,H}d$ is zero.  Consequently the coefficient on $[p_0,0]$ in $c+\partial^{\eta,H}d$ is, independently of $d$, equal to $T^{-\nu_{[M]}(a)}\lambda_0$, and so we have, for all $d$, \[ \ell(c+\partial^{\eta,H}d)\geq \nu_{[M]}(a)+H(p_0).\]  This immediately implies the theorem.\end{proof}

\begin{remark}
\label{nokur} Removing the dependence on Kuranishi structures in Theorem \ref{slowthm} is a somewhat more delicate matter than for the other results in this paper.  The proof of Theorem \ref{slowthm} relies on an argument that, if the Hamiltonian and almost complex structure are taken independent of time, moduli spaces of spiked discs underlying the PSS maps admit an action of $S^1$, and that therefore only spiked discs which are fixed by this $S^1$ action will appear in zero-dimensional moduli spaces if these spaces are cut out transversely.  If one does not want to use Kuranishi structures to achieve this transversality, then one can adapt \cite[Theorem 7.4]{FHS} to argue that time-independent $H$ and $J$ can be chosen so that the spaces are cut out transversely \emph{except at multiply-covered spiked discs}.  However, this only helps if  the expected dimensions of a spaces of multiply-covered spiked discs never exceed of a space of simple spiked discs in the same homology class. One can check that this can be arranged provided that the minimal Chern number of $(M,\omega)$ is at least $n$.  Thus under this latter topological condition (which of course is stronger than strong semipositivity) one can dispense with Kuranishi structures in the proof of Theorem \ref{slowthm} using the techniques from Section \ref{bigdefsect}; however in greater generality Kuranishi structures (or something similar) do seem to be needed.
\end{remark}

\section{Capacity estimates}\label{capsec}

Recall that the $\pi_1$-sensitive Hofer--Zehnder capacity of the symplectic manifold $(M,\omega)$ is, by definition, \[ c_{HZ}^{\circ}(M,\omega)=\sup\{\max H-\min H|H\co M\to\mathbb{R}\mbox{ is a slow Hamiltonian}.\}  \]  We begin with the following easy consequence of Theorem \ref{slowthm}.

\begin{cor}\label{hzest}  Fix $C>0$ and $\eta\in \oplus_{i=0}^{n-1}H_{2i}(M;\Lambda_{\omega}^{0})$ and suppose that one of the following two conditions holds: \begin{itemize} \item[(i)] There are $a,b\in H_*(M;\Lambda_{\omega})\setminus\{0\}$ such that, for all autonomous $H\co M\to\mathbb{R}$, \[ (\rho(a,H)_{\eta}-\nu_{[M]}(a))+(\rho(b,\bar{H})_{\eta}-\nu_{[M]}(b))\leq C;\] \emph{or} \item[(ii)] Where $[pt]$ is the standard generator of $H_0(M;\mathbb{C})$, for all $H\co M\to\mathbb{R}$, \[ \rho([pt];H)_{\eta}+\rho([pt];\bar{H})_{\eta}\geq -C.\] \end{itemize} Then the $\pi_1$-sensitive Hofer-Zehnder capacity of $(M,\omega)$ obeys the bound \[ c_{HZ}^{\circ}(M)\leq C.\]
\end{cor}

\begin{proof}
We need to show that, under either of the given conditions, for any slow autonomous Hamiltonian $H\co M\to\mathbb{R}$ we have $\max H-\min H\leq C$.  Of course for an autonomous Hamiltonian one has $\bar{H}=-H$, so this is equivalent to showing that $\max H+\max\bar{H}\leq C$.  

The sufficiency of (i) is then clear from Theorem \ref{slowthm}, since if $H$ is slow we have \[ \max H+\max \bar{H}\leq (\rho(a;H)_{\eta}-\nu_{[M]}(a))+(\rho(b;\bar{H})_{\eta}-\nu_{[M]}(b))\leq C\] by assumption.

Now instead assume (ii). Let $H\co M\to\mathbb{R}$ be slow, and choose any $\ep>0$.  By Proposition \ref{rhoprop}(vii) we have \[ -C\leq -\inf\{\rho(a;\bar{H})_{\eta}|\Pi(a,[pt])\neq 0\}-\inf\{\rho(b;H)_{\eta}|\Pi(b,[pt])\neq 0\},\] so there are $a,b$ with $\Pi(a,[pt])\neq 0$, $\Pi(b,[pt])\neq 0$, and $\rho(a;\bar{H})+\rho(b;H)\leq C+\ep$.  But from the definition of the pairing $\Pi$ one sees easily that in order for $\Pi(a,[pt])\neq 0$ one must have $\nu_{[M]}(a)\geq 0$.  This fact (for both $a$ and $b$) together with Theorem \ref{slowthm} gives \[\max H+\max \bar{H}\leq C+\ep,\] completing the proof  since $\ep$ was arbitrary.

\end{proof}

\begin{lemma}\label{mainhzlemma}  Suppose that $(M,\omega)$ admits a nonzero Gromov--Witten invariant of the form \[ \langle [pt],a_0,[pt],a_1,\ldots,a_k\rangle_{0,k+3,A},\] where $A\in H_2(M;\mathbb{Z})/torsion$ and $a_0,\ldots,a_k$ are rational homology classes of even degree.  Then for an open dense set of possible choices of the deformation parameter $\eta\in \oplus_{i=0}^{n-1}H_{2i}(M;\mathbb{C})$ we have \[ \nu_{[M]}([pt]\ast_{\eta}a_0)\geq -\langle [\omega],A\rangle.\]
\end{lemma} 

\begin{proof}  Since $[pt]\cap [pt]=0$, the class $A$ appearing in the Gromov--Witten invariant cannot be zero, in view of which none of the classes $a_0,\ldots,a_k$ can be a multiple of the fundamental class.  Moreover using the divisor axiom we can reduce to the case that $a_1,\ldots,a_k\in \oplus_{i=0}^{n-2}H_{2i}(M;\mathbb{Q})$.  Choose a homogeneous basis $\Delta_1,\ldots,\Delta_N$ for $\oplus_{i=0}^{n-1}H_{2i}(M;\mathbb{Z})/torsion$, with $\Delta_1,\ldots,\Delta_s$ (for some $s< N$) a basis for $H_{2n-2}(M;\mathbb{Z})/torsion$.  Using the multilinearity and symmetry properties of the Gromov--Witten invariants we can assume that our nonzero Gromov--Witten invariant takes the form \[ 0\neq \langle [pt],a_0,[pt],\underbrace{\Delta_{s+1},\ldots,\Delta_{s+1}}_{\alpha_{s+1}},\ldots,\underbrace{\Delta_N,\ldots,\Delta_N}_{\alpha_N}\rangle_{0,\sum \alpha_j+3,A}\] for some $\alpha=(\alpha_{s+1},\ldots,\alpha_N)\in \mathbb{N}^{N-s}$.

For $\vec{y}=(y_1,\ldots,y_s)\in \mathbb{C}^s$ and $\vec{z}=(z_{s+1},\ldots,z_N)\in \mathbb{C}^{N-s}$, write \[ \eta(\vec{y},\vec{z})=\sum_{i=1}^{s}y_i\Delta_i+\sum_{i=s+1}^{N}z_i\Delta_i.\]  Consider the class $[pt]\ast_{\eta(\vec{y},\vec{z})}a_0\in H_*(M;\Lambda_{\omega})$ as a function of $(\vec{y},\vec{z})$.  On expanding out the formula for $[pt]\ast_{\eta(\vec{y},\vec{z})}a_0$
and using the symmetry properties of the Gromov--Witten invariants, one finds that the coefficient on the fundamental class $[M]$ is an expression of the shape \[ \sum_{g\in \Gamma_{\omega}}\sum_{\beta=(\beta_{s+1},\ldots,\beta_N)\in \mathbb{N}^{N-s}}f_{g,\beta}(\vec{y})z_{s+1}^{\beta_{s+1}}\cdots z_{N}^{\beta_N}T^{g}\] for a certain function $f_{g,\beta}(\vec{y})$ of $\vec{y}$ which is identically zero for all but finitely many $\beta$\footnote{This finiteness statement follows from the dimension axiom for Gromov--Witten invariants; see the proof of Proposition \ref{fin} for the argument};  for the particular values $g=\langle[\omega],A\rangle$ and $\beta=\alpha$ we have 
\begin{align*} \left(\prod_{i=s+1}^{N}(\alpha_i !)\right)&f_{\langle[\omega],A\rangle,\alpha}(\vec{y})=\\&\sum_{\substack{B\in H_2(M;\mathbb{Z})/tors,\\ \langle[\omega],B\rangle=\langle[\omega],A\rangle}} \langle [pt],a_0,[pt],\underbrace{\Delta_{s+1},\ldots,\Delta_{s+1}}_{\alpha_{s+1}},\ldots,\underbrace{\Delta_N,\ldots,\Delta_N}_{\alpha_N}\rangle_{0,\sum \alpha_j+3,B}\left(\prod_{k=1}^{s}(e^{y_k})^{\Delta_k\cap B}\right) \end{align*} 
where $\cap$ denotes the Poincar\'e intersection pairing.  Now since $\{\Delta_k\}_{1\leq k\leq s}$ forms a basis for $H_{2n-2}(M;\mathbb{Z})/torsion$, the map $B\mapsto (\Delta_1\cap B,\ldots,\Delta_s\cap B)$ is injective.  Thus our assumed nonzero Gromov--Witten invariant implies that the coefficient on $\prod_{k=1}^{s}(e^{y_k})^{\Delta_k\cap A}$ in the above sum is nonzero.

This proves that the coefficient on $T^{\langle[\omega],A\rangle}[M]$ in the expansion of $[pt]\ast_{\eta(\vec{y},\vec{z})}a_0$ is a polynomial in $e^{y_i},z_i$ having a nonzero coefficient multiplying \[ \left(\prod_{k=1}^{s}(e^{y_k})^{\Delta_k\cap A}\right)\prod_{k=s+1}^{N}z_{i}^{\alpha_i}.\]  In particular, since this polynomial is not the zero polynomial, there is an open dense set of choices of $(\vec{y},\vec{z})$ at which the coefficient on $T^{\langle[\omega],A\rangle}[M]$ does not vanish.  It then follows directly from the definition of $\nu_{[M]}$ that, for $(\vec{y},\vec{z})$ in this open dense set, we have $\nu_{[M]}([pt]\ast_{\eta(\vec{y},\vec{z})}a_0)\geq -\langle[\omega],A\rangle$.
\end{proof}

\begin{cor}\label{mainhzcor} Suppose that $(M,\omega)$ admits a nonzero Gromov--Witten invariant of the form \[ \langle [pt],a_0,[pt],a_1,\ldots,a_k\rangle_{0,k+3,A},\] where $A\in H_2(M;\mathbb{Z})/torsion$ and $a_0,\ldots,a_k$ are rational homology classes of even degree.  Then \[ c_{HZ}^{\circ}(M,\omega)\leq \langle[\omega],A\rangle.\]
\end{cor}

\begin{proof} By Lemma \ref{mainhzlemma}, choose $\eta\in \oplus_{i=0}^{n-1}H_{2i}(M;\mathbb{C})$ so that $\gamma:=\nu_{[M]}([pt]\ast_{\eta}a_0)\geq -\langle [\omega],A\rangle$.  From the definition of the Poincar\'e pairing $\Pi$, one then sees that \[ \Pi([pt]\ast_{\eta}T^{\gamma}a_0,[pt])=\Pi(T^{\gamma}([pt]\ast_{\eta}a_0),[pt])\neq 0.\]  Consequently, for any normalized Hamiltonian $H$, we have by Proposition \ref{rhoprop}(vi),(vii), \begin{align*} \rho([pt];\bar{H})_{\eta}&=\sup\{-\rho(a;H)_{\eta}|\Pi(a,[pt])\neq 0\}\geq -\rho([pt]\ast_{\eta}T^{\gamma}a_0;H)_{\eta}&\geq -\rho([pt];H)_{\eta}-\rho(T^{\gamma}a_0;0)_{\eta}\end{align*}  (we've also used here that, for normalized $H$ and any $a,\eta$, $\rho(a;H\lozenge 0)_{\eta}=\rho(a;H)_{\eta}$, which follows from Proposition \ref{rhoprop} (v) and the fact that $H\lozenge 0$ is normalized if $H$ is).  

Now recall that $a_0$ is a rational homology class, in light of which we have $\rho(T^{\gamma}a_0;0)_{\eta}=\nu(T^{\gamma}a_0)=-\gamma$ by Proposition \ref{rhoprop} (iii).  Thus we obtain, for any $H$, \[  \rho([pt];\bar{H})_{\eta}+\rho([pt];H)_{\eta}\geq \gamma\geq -\langle[\omega],A\rangle,\] so the desired result follows from Case (ii) of Corollary \ref{hzest}.
\end{proof}

\begin{remark}\label{fixmark} More generally, one could consider the ``mixed'' invariants that in \cite{RT} are denoted by 
$\Phi_{A,\omega,0}(a_1,\ldots,a_k|b_1,\ldots,b_l)$ (with $k\geq 3$); for this invariant one specifies fixed marked points $z_i\in S^2$ ($1\leq i\leq k$) and generic representatives $\alpha_i$ of $a_i$ and $\beta_j$ of $b_j$ and formally enumerates pairs $(u,\{w_j\}_{1\leq j\leq l})$ consisting of a pseudoholomorphic representatives $u\co S^2\to M$ of $A$ with $u(z_i)\in \alpha_i$, $u(w_j)\in\beta_j$ for all $i,j$.  A modification of the proof of Corollary \ref{mainhzcor} 
shows that, if there is a nonzero invariant of the form $\Phi_{A,\omega,0}([pt],[pt],a_3,\ldots,a_k|b_1,\ldots,b_l)$ (with the classes $b_j$ even-dimensional) then we have the same estimate $c_{HZ}^{\circ}(M,\omega)\leq \langle [\omega],A\rangle$.  The key idea is to modify  the pair of pants product $\ast_{\eta}^{Floer}\co CF(H)\otimes CF(K)\to CF(H\lozenge K)$ to a map 
$\ast_{\eta,\vec{a},I}^{Floer}\co CF(H)\otimes CF(K)\to CF(H\lozenge K)$ which counts pairs of pants $u\co\Sigma\to M$ like those in the definition of $\ast_{\eta}^{Floer}$ which additionally satisfy $u(z_i)\in \alpha_i$ for $3\leq i\leq k$ for some preassigned fixed marked points $z_i$.  (On homology, this has the same effect as composing the pair of pants product and the quantum cap actions (as in \cite[Example 3.4]{PSS}) by each of the $a_i$ for $3\leq i\leq k$.)  On homology, one obtains the same type of triangle inequality for the spectral invariants (as in Proposition \ref{rhoprop} (vi)) for this operation as one does for the pair-of-pants product, as a result of which the proof of Corollary \ref{mainhzcor} can be extended to this case.  Details are left to the reader.

\end{remark}

\section{Calabi quasimorphisms}\label{Calabi}

Recall that if $G$ is a group, a \emph{quasimorphism} on $G$ is a map $\mu\co G\to\mathbb{R}$ such that there exists a constant $C$ (called the \emph{defect} of $\mu$) such that, for all $g,h\in G$, we have \[ |\mu(gh)-\mu(g)-\mu(h)|\leq C.\]

\begin{prop}\label{qmprop1}  Let $(M,\omega)$ be a closed symplectic manifold, $C>0$, $\eta\in \oplus_{i=0}^{n-1}H_{2i}(M;\Lambda_{\omega}^{0})$, and suppose that $e\in H_*(M,\Lambda_{\omega})$ has the properties that $e\ast_{\eta}e=e$ and, for all $\tilde{\phi}\in\widetilde{Ham}(M,\omega)$, \begin{equation}\label{qm1} \rho(e;H)_{\eta}+\rho(e;\bar{H})_{\eta}\leq C\end{equation} for all $H\co S^1\times M\to\mathbb{R}$.  Then the function $c(e;\cdot)_{\eta}\co \widetilde{Ham}(M,\omega)\to\mathbb{R}$ defined by $c(e;\tilde{\phi})_{\eta}=\rho(e;H)_{\eta}$ for any normalized $H$ with $\tilde{\phi}_H=\tilde{\phi}$ defines a quasimorphism with defect at most $C$.
\end{prop}

\begin{proof}  Since if $H$ and $K$ are normalized then $H\lozenge K$ is also normalized with $\tilde{\phi}_{H\lozenge K}=\tilde{\phi}_K\circ\tilde{\phi}_H$, the triangle inequality Proposition \ref{rhoprop}(vi) immediately gives \[ c(e;\tilde{\phi}\circ\tilde{\psi})_{\eta}\leq c(e;\tilde{\phi})_{\eta}+c(e;\tilde{\psi})_{\eta}\] for all $\tilde{\phi},\tilde{\psi}$.  On the other hand (\ref{qm1}) shows, for all $\tilde{\phi}$, \[ c(e;\tilde{\phi})_{\eta}+c(e;\tilde{\phi}^{-1})_{\eta} \leq C.\]  So for any $\tilde{\phi},\tilde{\psi}$, \begin{align*} c(e;\tilde{\phi})_{\eta}+c(e;\tilde{\psi})_{\eta}&=c(e;\tilde{\phi})_{\eta}+c(e;\tilde{\phi}^{-1}\circ\tilde{\phi}\circ\tilde{\psi})_{\eta} 
\\&\leq c(e;\tilde{\phi})_{\eta}+c(e;\tilde{\phi}^{-1})_{\eta}+c(e;\tilde{\phi}\circ\tilde{\psi})_{\eta}\leq C+c(e;\tilde{\phi}\circ\tilde{\psi})_{\eta},\end{align*} proving the proposition.
\end{proof}

If $(U,\omega_U)$ is an \emph{open} symplectic manifold, let $Ham(U,\omega_U)$ denote the group of diffeomorphisms of $U$ which arise as time-one maps of the vector fields of compactly supported time-dependent Hamiltonians, and let $\widetilde{Ham}(U,\omega_U)$ be the universal cover.  Recall that there is a homomorphism $Cal_U\co \widetilde{Ham}(M,\omega)\to\mathbb{R}$ defined by \[ Cal_U(\tilde{\phi})=\int_{0}^{1}H(t,\cdot)\omega^n\mbox{ for any compactly supported $H\co S^1\times U\to\mathbb{R}$ such that }\tilde{\phi}_H=\tilde{\phi}.\]  (In particular, the right hand side is independent of $H$.) Following \cite{EP03}, we make the following definition:

\begin{dfn} If $(M,\omega)$ is a closed symplectic manifold and $C>0$, a  \emph{Calabi quasimorphism} on $M$ is a map $\mu\co \widetilde{Ham}(M,\omega)\to\mathbb{R}$ such that \begin{itemize}  \item[(i)]  $\mu$ is a quasimorphism.
\item[(ii)] $\mu$ is homogeneous: for all $\tilde{\phi}\in\widetilde{Ham}(M,\omega)$ and $l\in\mathbb{\mathbb{Z}}$ we have \[ \mu(\tilde{\phi}^l)=l\mu(\tilde{\phi}).\]
\item[(iii)]  If $\tilde{\phi}=\tilde{\phi}_H$ where $H\co S^1\times M\to\mathbb{R}$ has support contained in $S^1\times U$ for some displaceable open set $U$, then \[ \mu(\tilde{\phi})=\int_{0}^{1}\int_M H(t,\cdot)\omega^n.\]
\end{itemize}\end{dfn}

 Recall that \cite[Th\'eor\`eme II.6.1]{Ban} shows that $\widetilde{Ham}(M,\omega)$ is perfect, and so cannot admit any nontrivial homomorphism to $\mathbb{R}$.
 
 The following is an easy generalization of results of \cite{EP03}:
 \begin{prop}  Suppose that $\eta\in \oplus_{i=0}^{n-1}H_*(M;\Lambda_{\omega}^{0})$  and that $e\in H_*(M,\Lambda_{\omega})$ has the properties that $e\ast_{\eta}e=e$ and, for some $C>0$, the estimate \[ \rho(e;H)_{\eta}+\rho(e;\bar{H})_{\eta}\leq C\] holds for all normalized $H\co S^1\times M\to\mathbb{R}$.  Then the formula \[ \mu_{e,\eta}(\tilde{\phi})=\left(-\int_{M}\omega^n\right)\lim_{k\to\infty}\frac{c(e,\tilde{\phi}^k)_{\eta}}{k}
\] defines a Calabi quasimorphism $\mu_{e,\eta}\co \widetilde{Ham}(M,\omega)\to\mathbb{R}$ on $M$ with defect at most $2C\int_M\omega^n$.
Moreover, $\mu_{e,\eta}$ obeys the following stability property (cf. \cite[(3)]{EPZ}): If $H,K\co S^1\times M\to\mathbb{R}$ are normalized then \[ \int_{0}^{1}\min_M(H(t,\cdot)-K(t,\cdot))dt\leq \frac{-1}{\int_M \omega^n}\left(\mu(\tilde{\phi}_H)-\mu(\tilde{\phi}_K)\right)\leq \int_{0}^{1}\max_M(H(t,\cdot)-K(t,\cdot))dt.\]
\end{prop}
\begin{proof} The fact that $\mu_{e,\eta}$ is well-defined (i.e. that the relevant limit always exists) follows from Proposition \ref{qmprop1} and standard facts about subadditive sequences (e.g., \cite[Problem 99]{PoSz}).  Given that $\mu_{e,\eta}$ is well-defined, the fact that it satisfies the homogeneity condition (ii) is trivial.  Quite generally (see \cite[Lemmas 2.21,2.58]{Ca}) the homogenization $\bar{\lambda}$ of a quasimorphism $\lambda$ is a quasimorphism, with defect at most twice the defect of $\lambda$; this establishes condition (i) and the estimate on the defect for $\mu_{e,\eta}$.  The Calabi property (iii) is just a restatement of Proposition \ref{disp}.  Finally, the stability property follows directly by homogenizing Proposition \ref{rhoprop}(ii).

\end{proof}

For $\eta\in \oplus_{i=0}^{n-1}H_{2i}(M;\Lambda_{\omega}^{0})$, we denote by $QH(M,\omega)_{\eta}$ the commutative $\Lambda_{\omega}$-algebra whose underlying $\Lambda_{\omega}$-module is the \emph{even-degree} homology $\oplus_{i=0}^{n}H_{2i}(M;\Lambda_{\omega})$, equipped with multiplication given by the deformed product $\ast_{\eta}$.
 
\begin{prop}\label{splitbound} Suppose that there is a direct sum splitting of algebras $QH(M,\omega)_{\eta}=F\oplus A$, and let $e\in F$ be the multiplicative identity for the subalgebra $F$.  Then, for some $C>0$, we have \[ \rho(e;H)_{\eta}+\rho(e,\bar{H})_{\eta}\leq C\] for all Hamiltonians $H\co S^1\times M\to\mathbb{R}$.
\end{prop} 

The above three propositions immediately imply:
\begin{cor}\label{qmcor}  If  there is a direct sum splitting of algebras $QH(M,\omega)_{\eta}=F\oplus A$, then where $e$ is the multiplicative identity in $F$, the function $\mu_{e,\eta}\co \widetilde{Ham}(M,\omega)\to\mathbb{R}$ defines a Calabi quasimorphism.\end{cor}

\begin{proof}[Proof of Proposition \ref{splitbound}] Given what we have already done the argument is essentially a duplicate of one in \cite[Section 3]{EP03}; we include it for completeness.  First, one notes that by \cite[Lemma 3.2]{EP03} there is a constant $K$ such that, for any $x\in F$, we have \[ \nu(x)+\nu(x^{-1})\leq K;\] recall from Proposition \ref{rhoprop}(iii) that $\rho(x;0)_{\eta}=\nu(x)$.  Also, since $e\ast_{\eta}e=e$, for any $a\in QH(M,\omega)_{\eta}$ we have $e\ast_{\eta}a\in F$ with $\rho(e\ast_{\eta}a;H)_{\eta}\leq \rho(a;H)_{\eta}+\nu(e)$.  We now have, making liberal use of Proposition \ref{rhoprop}, \begin{align*} -\rho(e;\bar{H})_{\eta}&=\inf\{\rho(a;H)_{\eta}|\Pi(e,a)\neq 0\}\\&\geq -\nu(e)+\inf\{\rho(e\ast_{\eta}a;H)_{\eta}|\Pi(e\ast_{\eta}a,[M])\neq 0\}
\\&\geq -\nu(e)+\inf\{\rho(e;H)_{\eta}-\nu((e\ast_{\eta}a)^{-1})|\Pi(e\ast_{\eta}a,[M])\neq 0\}\\&\geq -\nu(e)+\rho(e;H)_{\eta}-K+\inf\{\nu(e\ast_{\eta}a)|\Pi(e\ast_{\eta}a,[M])\neq 0\}.\end{align*}  Since whenever $\Pi(x,[M])\neq 0$ we have $\nu(x)\geq 0$, we hence obtain \[ \rho(e;H)_{\eta}+\rho(e;\bar{H})_{\eta}\leq \nu(e)+K,\] completing the proof with $C=\nu(e)+K$.\end{proof}

\begin{remark} \label{quasistate} It follows also that, under the assumptions of Corollary \ref{qmcor}, the function $\zeta_{\eta,e}\co C(M)\to\mathbb{R}$ defined by $\zeta(F)=\lim_{k\to\infty}\frac{\rho(e;kF)_{\eta}}{k}$ defines a \emph{symplectic quasi-state} in the sense of \cite{EP06}; given Proposition \ref{rhoprop} and Corollary \ref{qmcor} the proof of the quasi-state axioms for $\zeta_{\eta,e}$ is  an exact replication of \cite[Section 6]{EP06}. 
\end{remark}

The remainder of the paper will be concerned with studying circumstances in which the hypothesis of Corollary \ref{qmcor} and some other similar conditions are satisfied.
 
\section{Some algebraic input}\label{algsect}

In this section all rings are assumed commutative with unit, and a ring morphism necessarily maps unit to unit.  We deliberately do not assume our rings to be Noetherian.  If $R$ is a ring and $\mathfrak{p}\in \S R$ we use the customary notations $R_{\frak{p}}$ for the localization at $\frak{p}$ (i.e. $(R\setminus \frak{p})^{-1}R$) and $\kappa(\frak{p})$ for the residue field $\frac{R_{\frak{p}}}{\frak{p}R_{\frak{p}}}$.

The commutative algebra background that we require is mostly summarized by the following two-pronged theorem (we imagine that little if any of this will be surprising to an expert in commutative algebra).  For our later purposes, the most important implications of this theorem are that the subset of $\S R$ on which the condition denoted (A3) holds is open and is equal to the subset on which (A4) holds, and similarly that the subset on which (B3) holds is open and equal to that on which (B4) holds.

\begin{theorem}\label{mainalg}  Let $R$ be a ring containing $\mathbb{Q}$ as a subfield and let $A$ be a commutative $R$-algebra which, considered as an $R$-module, is finitely-generated and free.  Denote by $f\co \S A\to \S R$ the morphism of schemes induced by the unique ring morphism $R\to A$ (sending $r$ to $r\cdot 1$).\begin{enumerate}\item[(A)] The following are equivalent, for a point $\p\in \S R$:\begin{itemize}\item[(A1)] The morphism $f$ is unramified at every point in $f^{-1}(\{\frak{p}\})$.
\item[(A2)] There exists a field extension $\kappa(\p)\to k$ such that the map $\S(A\otimes_R k)\to \S k$ induced by the unique ring morphism $k\to A\otimes_R k$ is unramified.
\item[(A3)] There exists a field extension $\kappa(\p)\to k$ such that $A\otimes_R k$ decomposes as a direct sum\footnote{All direct sums in this theorem are direct sums in the category of algebras---thus both addition and multiplication split component-wise} of field extensions of $k$.
\item[(A4)] For every field extension $\kappa(\p)\to k$ the algebra $A\otimes_R k$ decomposes as a direct sum of field extensions of $k$.
\end{itemize} 
Moreover, the set $U_1$ of points $\p\in \S R$ at which (A1) holds is open in $\S R$.

\item[(B)] The following are equivalent, for a point $\p\in \S R$:  \begin{itemize} \item[(B1)] There is some $\frak{q}\in \S A$ such that $f(\frak{q})=\p$ and the morphism $f\co \S A\to \S R$ is unramified at $\frak{q}$.
\item[(B2)] There exists a field extension $\kappa(\p)\to k$ such that the map $\S (A\otimes_R k)\to \S k$ induced by the unique ring morphism $k\to A\otimes_R k$ is unramified at some point $\frak{q}\in \S(A\otimes_R k)$.
\item[(B3)] There exists a field extension $\kappa(\p)\to k$  and a direct sum splitting of $k$-algebras $A\otimes_R k=K\oplus S$ where $k\to K$ is a  field extension.
\item[(B4)] For every field extension $\kappa(\p)\to k$ there is a direct sum splitting of $k$-algebras $A\otimes_R k=K\oplus S$ where $k\to K$ is a  field extension.\end{itemize}  Moreover, the set $U_2$ of points $\p\in \S R$ at which (B1) holds is open in $\S R$.
\end{enumerate}
\end{theorem}
\begin{proof}[Proof of Theorem \ref{mainalg}]
\begin{lemma}\label{foc}  The morphism $f\co \S A\to \S R$ induced by $R\to A$ is flat, open, and closed.
\end{lemma}
Of course, $f\co \S  A\to \S  R$ is defined by sending a prime $\frak{p}$ in $\S  A$ to its preimage under $R\to A$, i.e. to $\frak{p}\cap R$.  By definition, a morphism of schemes is open (resp. closed) iff it is open (resp. closed) as a map of topological spaces. 
\begin{proof}[Proof of Lemma \ref{foc}] We first show that $f$ is closed.  Let $V(I)=\{\frak{q}\in \S  A|I\subset\frak{q}\}$ be an arbitrary closed set in $\S  A$.  The set $V(I\cap R)=\{\frak{p}\in \S  R|I\cap R\subset \frak{p}\}$ is then closed in $\S  R$ (of course we've identified $R$ with its image in $A$), and clearly $f(V(I))=\{\frak{q}\cap R|\frak{q}\in V(I)\}\subset V(I\cap R)$.  We claim that in fact equality holds.
Indeed, note that by \cite[Corollary 4.5]{Eis}, $A$ is integral over $R$.  Let $\frak{p}\in V(I\cap R)$, so $I\cap R\subset \frak{p}$.  The Going Up theorem (\cite[Proposition 4.15]{Eis}) then shows that there is $\frak{q}\in \S  A$ such that $\frak{q}\cap R=\frak{p}$ (i.e., $f(\frak{q})=\frak{p}$) and $I\subset\frak{q}$.  But this is precisely the statement that $\frak{p}\in f(V(I))$.  Thus $f$ takes an arbitrary closed set $V(I)$ to the closed set $V(I\cap R)$, proving that $f$ is closed.

Since $A$ is finitely-generated and free as an $R$-module, it is clearly flat as an $R$-module, and then the standard fact that flat ring maps induce flat morphisms on $\S $ (\cite[Proposition III.9.2.d]{Ha}) shows that $f$ is flat.   

We now show that $f$ is open.  Since $f$ is of finite presentation (as $A$ is a finitely-presented $R$-algebra), \cite[Corollaire 1.10.4]{GIV1} asserts that $f$ is open if and only if for any $\frak{q}\in \S A$ and any generalization\footnote{If $x$ and $y$ are points in a topological space, $x$ is called a generalization of $y$ if we have $y\in\overline{\{x\}}$.  In the case where the topological space in question is the Spec of a ring, so that $x$ and $y$ are prime ideals, this is equivalent to requiring that $x\subset y$.}  $\frak{p}'$ of the point $\frak{p}=f(\frak{q})$, there is a generalization $\frak{q}'$ of $\frak{q}$ so that $f(\frak{q}')=\frak{p}'$.    But by \cite[5.D]{Ma} the Going Down theorem holds for $R\to A$ because $A$ is a flat $R$-module, and the statement of the Going Down theorem precisely amounts to the existence of such a $\frak{p}'$.
\end{proof}


Let \[ \mathcal{U}=\{x\in \S A|f\mbox{ is unramified at }x\}.\]  By definition (\cite[17.3.7]{GIV4}), $f$ is unramified at $x$ iff there is a neighborhood $U$ of $x$ so that the restricted morphism $f|_U$ is unramified; thus our set $\mathcal{U}$ is obviously open.  Consequently Lemma \ref{foc} shows that $f(\mathcal{U})$ is open, and that $f(\S A\setminus \mathcal{U})$ is closed.  In the statement of Theorem \ref{mainalg}, we evidently have \[ U_2=f(\mathcal{U})\mbox{ and }U_1=(\S R)\setminus f(\S A\setminus \mathcal{U}).\]  This proves that these sets are open.  

Now since $R$ contains $\mathbb{Q}$ as a subfield, the residue fields $\kappa(\p)$ all have characteristic zero, so they are perfect fields (that is, all of their extensions are separable).  The equivalence (A3)$\Leftrightarrow$(A4) is then a quick consequence of the following basic theorem about coefficient extensions of algebras over fields.
\begin{theorem}{\cite[V.6.7, Theorem 4]{Bo}}\label{tensor} If $B$ is a finite-dimensional algebra over a  field $k$, the following are equivalent:
\begin{itemize}\item There is one perfect field extension $k'$ of $k$ such that the $k'$-algebra $B\otimes_kk'$ is reduced (i.e. its only nilpotent element is $0$).
\item For every extension $k\to k'$, the $k'$-algebra $B\otimes_k k'$ is reduced.
\item $B$ decomposes as a direct sum $B=K_1\oplus\cdots \oplus K_n$ where each $K_i$ is an algebraic field extension of $k$.\end{itemize}
\end{theorem}

The equivalence (A3)$\Leftrightarrow$(A4) follows immediately from this: if $A\otimes_R k$ is a direct sum of $k$-extensions for one extension $k$ of $\kappa(\p)$ (which will necessarily be of characteristic zero and hence perfect), Theorem \ref{tensor} shows that $A\otimes_R k$ will be reduced for all extensions $k$ of $\kappa(\p)$, and so any such $A\otimes_R k$ will be a direct sum of $k$-extensions by another application of Theorem \ref{tensor}.  

Meanwhile, the equivalence (B3)$\Leftrightarrow$(B4) follows in a similar way from \cite[Proposition 2.2(A)]{EP08}, which asserts that if $k\to k'$ is an extension of a field of characteristic zero and $B$ is a finite-dimensional algebra over $k$ then $B$ has a field as a direct summand if and only if $B\otimes_k k'$ has a field as a direct summand. In particular, this result shows that if (B3) holds then it holds with $k=\kappa(\p)$, and so applying the result again proves (B4).

\par{\textbf{(A1)$\Leftrightarrow$(A2)$\Leftrightarrow$(A4) and (B1)$\Rightarrow$(B2)}:} Since $f^{-1}(\{\p\})=Spec(A\otimes_R \kappa(\p))$, we appeal to \cite[Th\'eor\`eme 17.4.1,(a)$\Leftrightarrow$(d)]{GIV4}, which asserts that $f$ is unramified at $\frak{q}$ iff $f^{-1}(\{f(\frak{q})\})$ is unramified over $\kappa(f(\frak{q}))$ at $\frak{q}$.  It immediately follows that  (A1)$\Rightarrow$(A2) and (B1)$\Rightarrow$(B2) (just take $k=\kappa(\frak{p})$) in view of the fact that a morphism of schemes is unramified iff it is unramified at every point of the domain, as noted   
just after \cite[D\'efinition 17.3.7]{GIV4}.  It also follows that the special case of (A2) in which $k=\kappa(\frak{p})$ implies (A1), in view of which the proof of the implications stated at the start of this paragraph will be completed by the following lemma:

\begin{lemma}  Condition (A2) is equivalent to the following condition:\begin{itemize}\item[(A2')] For every field extension $\kappa(\frak{p})\to k$, 
the map $\S(A\otimes_R k)\to \S k$ induced by the unique ring morphism $k\to A\otimes_R k$ is unramified.\end{itemize}  Moreover, we have the equivalence (A2')$\Leftrightarrow$(A4).
\end{lemma}
\begin{proof}
Let $\p\in \S R$ and let $k$ be a field extension of $\kappa(\p)$ obeying the conclusion of (A2).  Write $f_{k,\p}\co \S(A\otimes_R k)\to \S k$ for the morphism induced by $k\to A\otimes_R k$.  Since the latter ring map is flat (\cite[3.C]{Ma}), $f_{k,\p}$ is flat. Our assumption (A2) also states that $f_{k,\p}$ is unramified.  So by the implication (c)$\Rightarrow$(c') of \cite[Corollaire 17.6.2]{GIV4}, the unique fiber of $f_{k,\p}$ is the $Spec$ of a direct sum of finite extensions of $k$.  Thus $A\otimes_Rk$ is reduced.  Of course $k$ is perfect since it has characteristic zero, so Theorem \ref{tensor} shows that $A\otimes_R \kappa(\p)$ is reduced, and moreover decomposes as a direct sum of $\kappa(\p)$-extensions.  Applying Theorem \ref{tensor} again shows that if $k'$ is now an arbitrary extension of $k$ then $A\otimes_R k'$ is a direct sum of $k'$-extensions.  We have now shown that (A2)$\Rightarrow$(A4).  Given (A4), applying the implication (c')$\Rightarrow$(c) of \cite[Corollaire 17.6.2]{GIV4}  shows that for any extension $k$ of $\kappa(\p)$ the morphism $\S(A\otimes_R k)\to \S k$ is unramified, thus establishing (A4)$\Rightarrow$(A2').  Since (A2')$\Rightarrow$(A2) is trivial the proof of the lemma is complete.
\end{proof}

We have now established all of part (A) of Theorem \ref{mainalg};   to complete the proof of (B) we will prove that (B2)$\Rightarrow$(B3) and (B4)$\Rightarrow$(B1).

Assume that (B2) holds for $\p\in \S R$ and the extension $\kappa(\p)\to k$, write $C=A\otimes_R k$ and let $f_{k,\p}\co \S C\to \S k$ be the map associated to $k\to A\otimes_R k=C$.  (B2) asserts that $f_{k,\p}$ has an unramified point, and we claim that we may reduce to the case that this unramified point is a closed point, i.e. corresponds to a maximal ideal in $C$.  Indeed, the set of unramified points of $f_{k,\frak{p}}$ is open in $\S C$, and hence is equal to a set of the form $\{\frak{q}\in \S C|I\not\subset \frak{q}\}$ for some ideal $I$.  The set in question is nonempty, and so $I$ must not be contained in the intersection of all prime ideals of $C$.  But $C$, being a finitely-generated  algebra over a field, is a Jacobson ring by the Nullstellensatz \cite[Theorem 4.19]{Eis}; thus the intersection of all prime ideals of $C$ is equal to the intersection of all maximal ideals.  So there is a maximal ideal, say $\frak{q}$, such that $I\not\subset\frak{q}$, and so our open set of unramified points contains this closed point $\frak{q}$.

     Since  $f_{k,\p}$ is unramified at $\frak{q}$, the implication (a)$\Rightarrow$(d') of \cite[Th\'eor\`eme 17.4.1]{GIV4} shows that the localization $C_{\frak{q}}$ is a field extension of $\kappa(\p)$ and that $\frak{q}$ is isolated in $\S C$. If $\frak{q}$ were the only point of $\S C$ then since $\frak{q}$ is maximal $C$ would be a local ring, and we would have $C_{\frak{q}}=C$, so $C$ would be a field and so (B3) would certainly hold.  So we may assume $\S C\setminus \{\frak{q}\}$ to be nonempty.   Thus since $\frak{q}$ is isolated we can write $\S C$ as a disjoint union of nonempty closed sets \[ \S C=\{\frak{q}\}\coprod \{\frak{r}|I\subset \frak{r}\}\] for some ideal $I\leq C$.  Since $\frak{q}$ is maximal and $I\not\subset \frak{q}$, $\frak{q}+I=C$.  Arguing as in \cite[Exercise 2.25]{Eis}, we find idempotents $e_1\in \frak{q},e_2\in I$ with $e_1+e_2=1$ and $e_1e_2=0$.  This gives a direct sum splitting $C=e_1C\oplus e_2C$.  Now since $e_2\in I$, the distinguished open set $D(e_2)=\{\frak{r}\in \S C|e_2\notin \frak{r}\}$ is equal to $\{\frak{q}\}$, so the ring $e_{2}^{-1}C$ is isomorphic to the localization $C_{\frak{q}}$ (for instance this follows directly from \cite[Proposition II.2.2]{Ha}) and therefore is a field.  But the natural map $C\to e_{2}^{-1}C$ is easily seen to restrict to $e_2 C$ as an isomorphism.  Thus $C$ decomposes as a direct sum isomorphic to $e_1C\oplus C_{\frak{q}}$ where $C_{\frak{q}}$ is a field extension of $k$.  This proves the implication (B2)$\Rightarrow$(B3).

Finally, assume that (B4) holds; in particular we may choose $k=\kappa(\p)$, so that $A\otimes_R \kappa(\p)\cong K\oplus S$ where $K$ is a field extension of $\kappa(\p)$.  Then $S$ is easily seen to be a maximal ideal, which may alternatively be characterized as the annihilator $\{x\in A\otimes_R\kappa(\p)|xK=0\}$.  Denote the multiplicative unit in $K$ by $e\in K\subset A\otimes_R\kappa(\p)$. If $\frak{r}$ is any prime ideal in $A\otimes_R \kappa(\p)$, the factorization $0=e(1-e)$ shows that either $e\in\frak{r}$ or $1-e\in \frak{r}$; in the latter case $\frak{r}$ contains and hence is equal to $S$.  Thus the open set $D(e)=\{\frak{r}|e\notin \frak{r}\}$ is equal to $\{S\}$.  So $S$ is an isolated point of  $\S(A\otimes_R \kappa(\p))$ and
 the local ring $(A\otimes_R \kappa(\p))_S$ at $S$ is isomorphic to $e^{-1}(A\otimes_R \kappa(\p))$, which in turn is isomorphic to the field $K$.  
So $(A\otimes_R \kappa(\p))_S\cong K$ is a field extension of $\kappa(\p)$, which is separable since we are working in characteristic zero.  So the implication (d')$\Rightarrow$(a) of \cite[Th\'eor\`eme 17.4.1]{GIV4} proves that $\S A\to \S R$ is unramified at the point $\frak{q}=\iota(S)$ where $\iota\co \S(A\otimes_R\kappa(\p))\to \S A$ is the map induced by the natural map $A\to A\otimes_R\kappa(\p)$.  This completes the proof of the implication (B4)$\Rightarrow$(B1) and thus of all of Theorem \ref{mainalg}.
\end{proof}

\begin{dfn}\label{ssdfn} Let $A$ be an $R$-algebra as in Theorem \ref{mainalg}.\begin{itemize}
\item We say that $A$ is \emph{generically semisimple} if the subset $U_1\subset \S R$ of Theorem \ref{mainalg}(A) is nonempty.
\item We say that $A$ is \emph{generically field-split} if the subset $U_2\subset \S R$ of Theorem \ref{mainalg}(B) is nonempty.
\end{itemize}
\end{dfn}

\begin{prop}\label{basechange} Let $A$ be an $R$-algebra as in Theorem \ref{mainalg}, let $\phi\co R\to S$ be a ring map, and consider the resulting $S$-algebra $A\otimes_R S$.\begin{itemize}\item[(i)] If $A\otimes_R S$ is generically semisimple (resp. generically field-split) then $A$ is generically semisimple (resp. generically field-split).
\item[(ii)] Suppose that $R$ and $S$ are integral domains and that the ring map $\phi\co R\to S$ is injective.  Then $A\otimes_R S$ is generically semisimple (resp. generically field-split) if and only if $A$ is generically semisimple (resp. generically field-split).
\end{itemize}
\end{prop}

\begin{proof} For (i), that $A\otimes_R S$ is generically semisimple (resp. generically field split) implies, by the equivalence (A1)$\Leftrightarrow$(A3) (resp. (B1)$\Leftrightarrow$(B3)), that there is a field $k$ and a ring map $\psi\co S\to k$ such that $(A\otimes_R S)\otimes_S k$ decomposes as a direct sum of extensions of $k$ (resp. has an extension of $k$ as a direct summand) (indeed we could take $k=\kappa(\p)$ where $\p$ is an arbitrary element of $U_1$ (resp. $U_2$)).  Now where $k$ is made into an $R$-algebra via the map $\psi\circ\phi$, we have \[ A\otimes_R k=(A\otimes_R S)\otimes_S k.\]  Where $\frak{q}=\ker(\psi\circ\phi)$, $\frak{q}$ is a prime ideal of $R$, and since $\kappa(\frak{q})$ is the  field of fractions of $R/\frak{q}$ the map $\psi\circ\phi$ factors as a composition $R\to \kappa(\frak{q})\to k$ where $\kappa(\frak{q})\to k$ is a field extension.  So since $A\otimes_R k$  decomposes as a direct sum of extensions of $k$ (resp. has an extension of $k$ as a direct summand) we have $\frak{q}\in U_1$ (resp. $\frak{q}\in U_2$), proving (i).

As for (ii), since $R$ and $S$ are integral domains their spectra contain generic points $\eta_R\in \S R$, $\eta_S\in \S S$, corresponding to the zero ideals in the respective rings.  The residue fields $\kappa(\eta_R)$ and $\kappa(\eta_S)$  at these generic points are just the fields of fractions of the respective domains, and so the monomorphism $R\to S$ induces a field extension $\kappa(\eta_R)\to \kappa(\eta_S)$.  Moreover we have \[ (A\otimes_R \kappa(\eta_R))\otimes_{\kappa(\eta_R)}\kappa(\eta_S)=(A\otimes_R S)\otimes_S \kappa(\eta_S).\]  Consequently, by Theorem \ref{tensor}  and \cite[Proposition 2.2(A)]{EP08}, $A\otimes_R \kappa(\eta_R)$ decomposes as a direct sum of extensions of $\kappa(\eta_R)$ (resp. has an extension of $k$ as a direct summand)
if and only if the corresponding property holds for $(A\otimes_R S)\otimes_S \kappa(\eta_S)$.  Now any nonempty open set in $\S R$ contains $\eta_R$, and likewise any nonempty open set in $\S S$ contains $\eta_S$, so  by the equivalence (A1)$\Leftrightarrow$(A3) (resp. (B1)$\Leftrightarrow$(B3))
it follows that the set $U_{1}^{R}\subset\S R$ (resp. $U_{2}^{R}\subset \S R$) associated to $A$ via Theorem \ref{mainalg} is nonempty if and only if the corresponding subset of $\S S$ associated to $A\otimes_R S$ via Theorem \ref{mainalg} is nonempty.\end{proof}

\section{Quantum homology II} \label{quant2} 

What we will call the ``universal big quantum homology'' $\AR$  in this paper may be regarded as an invariant associated to a pair $(M,\mathcal{C})$ where  $\mathcal{C}$ is a nonempty connected component of the space of symplectic forms on the $2n$-dimensional closed manifold $M$; the pair $(M,\mathcal{C})$ shall be fixed throughout this section.  This invariant is a commutative algebra $\mathbb{A}_{R_M}$ over a ring $R_M$; this latter ring will be referred to as the ``universal quantum coefficient ring'' of $M$ (or, more properly, of $(M,\mathcal{C})$).  Writing \[ H_k(M):=\frac{H_k(M;\mathbb{Z})}{torsion},\quad H_{ev}(M)=\oplus_{i=0}^{n}H_{2i}(M),\] 
let  $\{\Delta_0,\Delta_1,\ldots,\Delta_N\}$ be an integral basis of $H_{ev}(M)$ for which each $\Delta_i$ has some even homogeneous grading $|\Delta_i|$, such that \[ \Delta_0=[M],\mbox{ and for some }s\in \{1,\ldots,N\},|\Delta_j|=2n-2\Leftrightarrow 1\leq i\leq s.\]
Thus the subgroup of $H_{ev}(M)$ with codimension at least $4$ is spanned by $\Delta_{s+1},\ldots,\Delta_N$.  As a module, we will have $\AR=H_{ev}(M)\otimes_{\mathbb{Z}}R_M$, with $\{\Delta_0,\ldots,\Delta_N\}$ serving as a standard $R_M$-basis for $\AR$.  We must now describe the ring $R_M$, and the multiplication rule for $\AR$.

\subsection{The universal quantum coefficient ring $R_M$}

As before Gromov--Witten invariants are denoted with the notation $\langle a_1,\ldots,a_k\rangle_{0,k,\beta}$; from now on we will always take $a_i\in H_*(M)$ to be homogeneous, and $\beta\in H_2(M)$.  Note that these quantities are rational numbers (integers if $M$ is semipositive) which are independent of $J$ and of the particular symplectic form $\omega$ representing the deformation class $\mathcal{C}$.  The quantity is nonzero only when \begin{equation}\label{gwdim} \sum_{i=1}^{k}(2n-|a_i|)=2\left(n+\langle c_1(TM),\beta\rangle+(k-3)\right).\end{equation}  

Let \[ H_{2}^{eff}(M)=\{\beta\in H_{2}(M)|(\exists a_1,\ldots,a_k\in H_*(M))(\langle a_1,\ldots,a_k\rangle_{0,k,\beta}\neq 0)\}\] and define the ``GW-effective cone'' to be \[ C^{eff}=C^{eff}(M)=\left\{\left.\sum_{i=1}^{l}n_i\beta_i\right|n_1,\ldots,n_l\in\mathbb{N},\beta_1,\ldots,\beta_l\in H_{2}^{eff}(M).\right\}.\]

Since the Gromov--Witten invariants are independent of the choice of $\omega\in\mathcal{C}$, so too is the GW-effective cone $C^{eff}$.  

\begin{lemma}\label{cone} \begin{itemize}\item[(i)] If $\omega$ is any symplectic form in the deformation class $\mathcal{C}$ and if $D\in\mathbb{R}$, there are only finitely many elements $\beta\in C^{eff}$ such that $\int_{\beta}\omega\leq D$.
\item[(ii)]
If $\beta\in C^{eff}$ then there are just finitely many pairs $(\beta_1,\beta_2)$ such that $\beta_1,\beta_2\in C^{eff}$ and $\beta_1+\beta_2=\beta$.
\end{itemize}
\end{lemma}

\begin{proof} If $\beta\in C^{eff}\setminus\{0\}$ with $\int_{\beta}\omega\leq D$, say $\beta=\sum_{i=1}^{l} n_i\beta_i$ with each $n_i\beta_i\neq 0$ (so $n_i\geq 1$) and $\beta_i\in H_{2}^{eff}(M)$, then for $J$ an arbitrary $\omega$-compatible almost complex structure the existence of a nonzero $\langle a_1,\ldots,a_k\rangle_{0,k,\beta_i}$ produces for each $i$ a genus-zero $J$-holomorphic bubble tree representing $\beta_i$.
Then $\int_{\beta_i}\omega\leq \int_{\beta}\omega\leq D$ for each $i$.  
By Gromov compactness, there are only finitely many nonzero classes $\beta_i$ of energy at most $D$ represented by a $J$-holomorphic bubble tree, and $\omega$ evaluates on each of these as at least some positive number $\hbar$.  Consequently there are only finitely many positive integer combinations of these $\beta_i$ having energy at most $D$, and therefore there are only finitely many possibilities for the class $\beta$.  This proves (a).

Part (b) then follows immediately: choose an arbitrary symplectic form $\omega$ from $\mathcal{C}$.  If $\beta_1+\beta_2=\beta$ and $\beta_1,\beta_2\in C^{eff}$, then $\beta_1$ and $\beta_2$ are necessarily each drawn from among the finitely many classes $\gamma\in C^{eff}$ with $\int_{\gamma}\omega\leq D:=\int_{\beta}\omega$.
\end{proof}

We can now define the ring $R_M$: set theoretically, let \[ R_M=\left\{\left.\sum_{\beta\in C^{eff}}f_{\beta}q^{\beta}\right|(\forall \beta\in C^{eff})(f_{\beta}\in\mathbb{Q}[x_{s+1},\ldots,x_N])\right\}.\]

We use the obvious componentwise addition $\sum f_{\beta}q^{\beta}+\sum g_{\beta}q^{\beta}=\sum (f_{\beta}+g_{\beta})q^{\beta}$,  while multiplication is, as one would expect, defined by \[ \left(\sum_{\beta\in C^{eff}}f_{\beta}q^{\beta}\right)\left(\sum_{\eta\in C^{eff}}g_{\eta}q^{\eta}\right)=\sum_{\zeta\in C^{eff}}\left(\sum_{\beta+\eta=\zeta}f_{\beta}g_{\eta}\right)q^{\zeta}.\]  That the right hand side is well-defined follows directly from Lemma \ref{cone}(ii), which ensures that the inner sum on the right is finite for any given $\zeta$.  So since $0\in C^{eff}$ and $C^{eff}$ is closed under addition, $R_M$ is a well-defined ring (with unit $1:=q^0$).  It is not difficult to check that $R_M$ is an integral domain. On the other hand I do not know what assumptions, if any, on $M$ are needed to ensure that $R_M$ is Noetherian; fortunately, Theorem \ref{mainalg} applies regardless of whether or not the ring $R$ in its hypothesis is Noetherian.

Given other conventions in the literature, it perhaps bears emphasizing that while an element of $R_M$ may have a nonzero coefficient $f_{\beta}$ on $q^{\beta}$ for infinitely many different $\beta$, the coefficients $f_{\beta}$ themselves are taken to be \emph{polynomials}, not power series, in the variables $x_{s+1},\ldots,x_N$.  These latter variables may be regarded as being dual to the basis $\{\Delta_{s+1},\ldots,\Delta_N\}$ for $\oplus_{i=0}^{n-2}H_{2i}(M)$ from earlier.  Formal variables dual to the basis $\{\Delta_1,\ldots,\Delta_s\}$ for $H_{2n-2}(M)$ (or, more accurately, exponentiated versions of these formal variables) can be regarded as being incorporated into the formal symbol $q$.  It will typically not be true that the various polynomials $f_{\beta}$ appearing in a given element of $R_M$ have uniformly bounded degree.

\begin{remark}\label{whyrm}This choice of coefficient ring $R_M$ is motivated by the fact that it enjoys the following two properties:
\begin{itemize}\item[(a)]  The quantum homology $\AR$ of $M$  may be naturally defined as an algebra over $R_M$ for any $M$.
\item[(b)] For many other rings $\Lambda$ obeying property (a), there is a diagram of ring maps incorporating the rings $R_M$ and $\Lambda$ which allows Proposition \ref{basechange} to be used to relate the properties of the quantum homology with coefficients in $\Lambda$ to the properties of $\AR$.
\end{itemize}

We will see many examples of (b) below.  In the simplest cases, the diagram alluded to in (b) simply consists of a map $R_M\to \Lambda$, and the quantum homology with coefficients in $\Lambda$ is just $\AR\otimes_{R_M}\Lambda$.  In other cases the diagram will be more complicated: the one involved in our discussion of the quantum homology of blowups $\tilde{M}$ in Section \ref{blsect} takes the form \begin{equation}\label{bldiag}\xymatrix{ R_{\tilde{M}}\ar[d] & B  \ar@<-.4ex>@{^{(}->}[d]\ar[r] & B/ZB \\ R_{\tilde{M}}^{0} \ar@<-.4ex>@{^{(}->}[r]& R_{\tilde{M}}^{0}[q^{\pm\frac{E'}{n-1}}] & }\end{equation}
\end{remark}

\subsection{Quantum multiplication}\label{mult}

Having introduced $R_M$, we now define the quantum product $\ast$ on the big quantum homology $\AR$ in a standard way.  Recall that $\AR$ is freely generated as a $R_M$-module by the homogeneous basis $\Delta_0,\dots,\Delta_N$ for $H_{ev}(M)$, where $\Delta_0=[M]$ and $\Delta_1,\ldots,\Delta_s$ span $H_{2n-2}(M)$.  For $i,j,k\in\{0,\ldots,N\}$, consider the formal sum \[ c_{ijk}=\sum_{\beta\in H_{2}^{eff}}\left(\sum_{\alpha=(\alpha_{s+1},\ldots,\alpha_N)\in\mathbb{N}^{N-s}}\frac{1}{\alpha!}\langle\Delta_i,\Delta_j,\Delta_k,\underbrace{\Delta_{s+1},\ldots,\Delta_{s+1}}_{\alpha_{s+1}},\ldots,\underbrace{\Delta_{N},\ldots,\Delta_{N}}_{\alpha_{N}}\rangle_{0,|\alpha|+3,\beta}x^{\alpha}\right)q^{\beta}.\]

Here we use standard multi-index notation for a tuple of nonnegative integers $\alpha=(\alpha_{s+1},\ldots,\alpha_N)$, namely $|\alpha|=\sum \alpha_i$, $\alpha!=\prod_{i=s+1}^{N}(\alpha_{i}!)$, and $x^{\alpha}=x_{s+1}^{\alpha_{s+1}}\cdots x_{N}^{\alpha_N}$.

\begin{prop}\label{fin} For each $i,j,k$ we have $c_{ijk}\in R_M$.
\end{prop}
\begin{proof} This proposition amounts to the statement that, for any given $\beta\in H_{2}^{eff}(M)$, the expression \[ \sum_{\alpha=(\alpha_{s+1},\ldots,\alpha_N)\in\mathbb{N}^{N-s}}\frac{1}{\alpha!}\langle\Delta_i,\Delta_j,\Delta_k,\underbrace{\Delta_{s+1},\ldots,\Delta_{s+1}}_{\alpha_{s+1}},\ldots,\underbrace{\Delta_{N},\ldots,\Delta_{N}}_{\alpha_{N}}\rangle_{0,|\alpha|+3,\beta}x^{\alpha}\] is a polynomial, which in turn is to say that, again for any given $\beta\in H_{2}^{eff}(M)$, there are just finitely many multi-indices $\alpha$ such that \[  \langle\Delta_i,\Delta_j,\Delta_k,\underbrace{\Delta_{s+1},\ldots,\Delta_{s+1}}_{\alpha_{s+1}},\ldots,\underbrace{\Delta_{N},\ldots,\Delta_{N}}_{\alpha_{N}}\rangle_{0,|\alpha|+3,\beta}\neq 0.\]  
Bearing in mind that, for $l=s+1,\ldots,N$, we have $2n-|\Delta_l|\geq 4$, by (\ref{gwdim}) the above invariant can be nonzero only if \[ 4|\alpha|+(2n-|\Delta_i|)+(2n-|\Delta_j|)+(2n-|\Delta_k|)\leq 2(n+\langle c_1(TM),\beta\rangle+|\alpha|),\] which in turn forces \[ |\alpha|\leq \frac{|\Delta_i|}{2}+\frac{|\Delta_j|}{2}+\frac{|\Delta_k|}{2}+\langle c_1(TM),\beta\rangle-2n.\]  Since, for fixed $\beta$, there are only finitely many multi-indices $\alpha$ obeying this bound on $|\alpha|$ the proposition follows.
\end{proof}

For $k=0,\ldots,N$ define a dual element $\Delta^k\in H_{2n-|\Delta_k|}$ by the property that \[ \Delta_j\cap \Delta^k=\delta_{j}^{k}\mbox{ for all }j\] where $\cap$ is the Poincar\'e intersection pairing and  $\delta_{j}^{k}$ is the Kronecker symbol (equivalently, $\Delta^k=\sum_j g^{kj}\Delta_j$ if $\{g^{kj}\}$ is the inverse of the matrix representing the Poincar\'e pairing in the basis $\Delta_0,\ldots,\Delta_N$).  The multiplication law for the algebra $\AR$ is then defined by extending bilinearly from \[ \Delta_i\ast \Delta_j=\sum_{k=0}^{N}c_{ijk}\Delta^k.\]  Since the $c_{ijk}$ belong to $R_M$ this multiplication law is well-defined.  $\AR$ is then a commutative (since we are restricting to even dimensional homology) algebra with unit $\Delta_0=[M]$; from \cite[Section 4]{KM} it follows that the associativity of the algebra is a formal consequence of a certain set of axioms for Gromov--Witten invariants, and in \cite[Section 23]{FO} it is shown that the Gromov--Witten invariants for general symplectic manifolds constructed in \cite{FO} indeed satisfy all of the axioms needed for associativity (\cite{LiT},\cite{R} also contain such results).



We have now associated to the deformation class $(M,\mathcal{C})$ of symplectic manifolds a ring $R_M$ and an $R_M$-algebra $\AR$ which, module-theoretically, is free and finitely generated. Theorem \ref{mainalg} and Definition \ref{ssdfn} thus apply to the algebra $\AR$, so we may consider the questions of whether $\AR$ is generically semisimple or generically field-split (in which case we say that the symplectic deformation class $(M,\mathcal{C})$  ``has generically semisimple big quantum homology'' or ``has generically field-split big quantum homology,'' respectively).



\subsection{Other coefficient systems}

Given a ring map $\phi\co R_M\to S$, we may form a quantum homology ring with coefficients in $S$: \[ QH^{\phi}(M;S):=\AR\otimes_{R_M}S\] where we use $\phi$ to view $S$ as an $R_M$-algebra (if the map $\phi$ is obvious from the context we will just write $QH(M;S)$).  Thus 
$QH^{\phi}(M;S)$ is the $S$-algebra freely generated as a module by $\Delta_0,\ldots,\Delta_N$ with the multiplication law \[ \Delta_i\ast\Delta_j=\sum_{k=0}^{N}\phi(c_{ijk})\Delta^k\] where $c_{ijk}\in R_M$ are the constants defined at the start of Section \ref{mult}.
As mentioned in Remark \ref{whyrm}, our choice of universal quantum coefficient ring $R_M$ has been motivated in part by the existence of many ring maps from $R_M$ to various rings in common use as coefficient rings for quantum homology.

\subsubsection{Small quantum homology}\label{sqh}
For example, let \[ R_{M}^{0}=\left\{\left.\sum_{\beta\in C^{eff}}c_{\beta}q^{\beta}\right|c_{\beta}\in\mathbb{Q}\right\}\] where as before $C^{eff}$ is the GW-effective cone and we use the obvious ``power series'' multiplication (which is well-defined by Lemma \ref{cone}(ii)).  There is an obvious map $\sigma\co R_M\to R_{M}^{0}$ defined by $\sigma\left(\sum f_{\beta}q^{\beta}\right)=\sum f_{\beta}(0)q^{\beta}$. 

\begin{dfn}\label{sqdfn} Let a deformation class $(M,\mathcal{C})$ of symplectic manifolds be given. \begin{itemize} \item The \emph{small quantum homology} of $(M,\mathcal{C})$ is the $R_{M}^{0}$-algebra \[ QH(M;R_{M}^{0})=QH^{\sigma}(M;R_{M}^{0})=\AR\otimes_{R_M}R_{M}^{0}\] where $\sigma\co R_M\to R_{M}^{0}$ is the above map.\item We say that $(M,\mathcal{C})$ has generically semisimple small quantum homology (resp. has generically field-split small quantum homology) if the $R_{M}^{0}$-algebra $QH(M;R_{M}^{0})$ is generically semisimple (resp. generically field-split) in the sense of Definition \ref{ssdfn}.\end{itemize} \end{dfn}


\begin{prop} If $(M,\mathcal{C})$ has generically semi-simple (resp. generically field-split) small quantum homology, then $(M,\mathcal{C})$ has 
generically semi-simple (resp. generically field-split) big quantum homology.
\end{prop}
\begin{proof}
Indeed, this follows immediately from Proposition \ref{basechange}(i).
\end{proof}
Consulting the definitions of the map $\sigma$ and of the multiplication law in big quantum homology, we see that $QH(M;R_{M}^{0})$ is the free $R_{M}^{0}$-module generated by $\Delta_0,\ldots,\Delta_N$ subject to the multiplication law\[ \Delta_i\ast\Delta_j=\sum_{k=0}^{N}\left(\sum_{\beta\in H_{2}^{eff}(M)}\langle \Delta_i,\Delta_j,\Delta_k\rangle_{0,3,\beta}q^{\beta}\right)\Delta^k, \] consistently with a formulation that some readers may find more familiar (again, $\{\Delta^l\}$ is a Poincar\'e dual basis to $\{\Delta_l\}$).

\subsubsection{Novikov rings} \label{novikov}

Choose a symplectic form $\omega$ belonging to the given deformation class $\mathcal{C}$ of forms on $M$,  with de Rham cohomology class $[\omega]$.
Let the subgroup $\Gamma_{\omega}\leq \mathbb{R}$ and the Novikov ring $\Lambda_{\omega}$ be as before (see Section \ref{quant1}).



Consider a general element \[ \eta=\eta_D+\sum_{i=s+1}^{N}\eta_i\Delta_i\in \bigoplus_{k=0}^{n-1}H_{2k}(M;\Lambda_{\omega}^{0})\] where $\eta_i\in\Lambda_{\omega}^{0}$ and $\eta_D\in H_{2n-2}(M;\Lambda_{\omega}^{0})$; here as before  $\Delta_{s+1},\ldots,\Delta_N$ is a fixed basis of $\oplus_{k=0}^{n-2}H_{2k}(M)$.  Define a map \begin{align}\label{phieta} \phi_{\eta}\co R_M &\to \Lambda_{\omega} \nonumber\\\sum_{\beta\in C^{eff}}\left(\sum_{\alpha=(\alpha_{s+1},\ldots,\alpha_N)}c_{\alpha}x^{\alpha}\right)q^{\beta}&\mapsto \sum_{\beta}\left(\sum_{\alpha}c_{\alpha}\prod_{i=s+1}^{N}\eta_{i}^{\alpha_i}\right)\exp(\eta_D\cap \beta)T^{\langle [\omega],\beta\rangle}.\end{align}  Here $\eta_D\cap\beta$ denotes the Poincar\'e intersection pairing between the ``divisor'' class $\eta_D$ and $\beta\in H_{2}(M)$.   That this map is well-defined (i.e. that $\phi_{\eta}$ sends every element of $R_M$ to a formal sum which obeys the finiteness condition in the definition of $\Lambda_{\omega}$) follows directly from Lemma \ref{cone}(i).  

\begin{dfn}  Let $(M,\omega)$ be a symplectic manifold, determining a symplectic deformation class $(M,\mathcal{C})$ where $\omega\in\mathcal{C}$. Let $\eta=\eta_D+\sum_{i=s+1}^{N}\eta_i\Delta_i\in \oplus_{k=0}^{n-1}H_{2k}(M;\Lambda_{\omega}^{0})$. \begin{itemize}\item[(i)] The \emph{$\eta$-deformed quantum homology} of $(M,\omega)$, denoted $QH(M,\omega)_{\eta}$, is the $\Lambda_{\omega}$-algebra \[ 
QH(M,\omega)_{\eta}=QH^{\phi_{\eta}}(M;\Lambda_{\omega})=\AR\otimes_{R_M}\Lambda_{\omega},\] where the $R_M$-algebra structure on $\Lambda_{\omega}$ is that induced by the ring map $\phi_{\eta}$ of (\ref{phieta}).
\item[(ii)] In the special case that $\eta=\eta_D\in H_{2n-2}(M)$, $QH(M,\omega)_{\eta}$ will also be called the ``$\eta$-twisted small quantum homology'' of $(M,\omega)$.
\end{itemize}
\end{dfn}

This is clearly consistent with the terminology from before: as in Section \ref{Calabi} $QH(M,\omega)_{\eta}$ is the even part of the algebra $(H_*(M;\Lambda_{\omega});\ast_{\eta})$.

To prepare for our next result, we introduce some notation:
\begin{dfn} Given a basis $\mathcal{B}=\{\Delta_0,\ldots,\Delta_N\}$ for $H_{ev}(M)$ with $\Delta_0=[M]$ and $\Delta_1,\ldots,\Delta_s$ a basis for $H_{2n-2}(M)$, we define \[ \EB\co \oplus_{k=0}^{n-1}H_{2k}(M;\mathbb{C})\to \mathbb{C}^N\] by \[ \EB\left(\sum_{i=1}^{N}\eta_i\Delta_i\right)=\left(e^{\eta_1},\ldots,e^{\eta_s},\eta_{s+1},\ldots,\eta_N\right),\] and \[ \EB^{0}\co H_{2n-2}(M;\mathbb{C})\to \mathbb{C}^s \] by \[ \EB^{0}\left(\sum_{i=1}^{s}\eta_i\Delta_i\right)=\left(e^{\eta_1},\ldots,e^{\eta_s}\right).\]
\end{dfn}

\begin{theorem} \label{bigsymp}For a closed symplectic manifold $(M,\omega)$, the following are equivalent:\begin{itemize} \item[(i)] For some $\eta\in \oplus_{k=0}^{n-1}H_{2k}(M;\Lambda_{\omega}^{0})$, the $\eta$-deformed quantum homology $QH(M,\omega)_{\eta}$ is a semisimple $\Lambda_{\omega}$-algebra (resp. has a field as a direct summand).
\item[(ii)] Where $\mathcal{C}$ is the deformation class of $\omega$, $(M,\mathcal{C})$ has generically semisimple big quantum homology (resp. has generically field-split big quantum homology).
\item[(iii)]  There is a nonzero Laurent polynomial \[ f\in \mathbb{Q}[z_1,z_{1}^{-1},\ldots,z_s,z_{s}^{-1},z_{s+1},z_{s+2},\ldots,z_N]\] such that, for all $\eta\in \oplus_{k=0}^{n-1}H_{2k}(M;\mathbb{C})$ such that $f(\EB(\eta))\neq 0$, the $\eta$-deformed quantum homology $QH(M,\omega)_{\eta}$ is a semisimple $\Lambda_{\omega}$-algebra (resp. has a field as a direct summand).\end{itemize}
\end{theorem}

\begin{remark}\label{algind} Note that since, in (iii), the polynomial $f$ has its coefficients in $\mathbb{Q}$, it follows that when any of the above equivalent conditions holds, if we choose a particular $\eta$ with the property that the coordinates of $\EB(\eta)=\left(e^{\eta_1},\ldots,e^{\eta_s},\eta_{s+1},\ldots,\eta_N\right)$ are algebraically independent over $\mathbb{Q}$, then $QH(M,\omega)_{\eta}$ will automatically be semisimple (resp. have a field as a direct summand) for this specific choice of $\eta$.
\end{remark}

\begin{proof}  The fact that (i)$\Rightarrow$(ii) follows directly from Proposition \ref{basechange}(i).  The implication (iii)$\Rightarrow$(i) is trivial.  It remains to prove that (ii)$\Rightarrow$(iii).

Accordingly, assume that $(M,\mathcal{C})$ has generically semisimple big quantum homology (resp. has generically field-split big quantum homology).
Thus the open subset $U_1$ (resp. $U_2$) of $\S R_M$ produced by applying Theorem \ref{mainalg} to the $R_M$-algebra $\AR$ is nonempty.  Recall that  a basis for the topology of $\S R_M$ is formed by distinguished open sets of the form $D(g)=\{\p\in \S R_M|g\notin \p\}$ where $g\in R_M$.  So the open set produced by Theorem \ref{mainalg} contains one of these sets $D(g)$ with $g\neq 0$ (as of course $D(0)=\varnothing$); we fix this $g$.  

Since the codomain of the map $\phi_{\eta}\co R_M\to \Lambda_{\omega}$  is a field, $\ker\phi_{\eta}$ is a prime ideal; let us denote this prime ideal by $\p_{\eta}$.  If $\p_{\eta}\in D(g)$, then $\AR\otimes_{R_M}\kappa(\p_{\eta})$ is semisimple (resp. has a field as a direct summand).  Now $\phi_{\eta}\co R_M\to \Lambda_{\omega}$ factors through the canonical map $R_M\to \kappa(\p_{\eta})$ to give a field extension $\kappa(\p_{\eta})\to \Lambda_{\omega}$, so the equivalences in Theorem \ref{mainalg} show that $QH(M,\omega)_{\eta}$ is semisimple (resp. has a field as a direct summand) whenever the same property holds for $\AR\otimes_{R_M}\kappa(\p_{\eta})$.

As such, the proof will be complete if we show that, whenever $0\neq g\in R_M$, there is $f\in \mathbb{Q}[z_1,z_{1}^{-1},\ldots,z_s,z_{s}^{-1},z_{s+1},z_{s+2},\ldots,z_N]$ such that $\ker\phi_{\eta}\in D(g)$ whenever $f(\EB(\eta))\neq 0$.  Of course, $\ker\phi_{\eta}\in D(g)$ if and only if $\phi_{\eta}(g)\neq 0$.  Let us write \[ g=\sum_{\beta\in C^{eff}}g_{\beta}q^{\beta} \] where each $g_{\beta}\in \mathbb{Q}[z_{s+1},\ldots,z_N]$.  Since $g\neq 0$, let $\lambda_0$ be the minimal value of $\langle [\omega],\beta\rangle$ over all those $\beta$ with $g_{\beta}\neq 0$.  By Lemma \ref{cone}(i), there are just finitely many $\beta\in H_2(M)$, say $\beta_1,\ldots,\beta_k$, such that $g_{\beta}\neq 0$ and $\langle [\omega],\beta\rangle=\lambda_0$.

For $i=1,\ldots,s$ and $j=1,\ldots,k$ write \[ \Delta_i\cap \beta_j=n_{ij}.\]  Then, if \[ \eta=\sum_{i=1}^{N}\eta_i\Delta_i,\] the coefficient on $T^{\lambda_0}$ in $\phi_{\eta}(g)$ is \[ \sum_{j=1}^{k}g_{\beta_j}(\eta_{s+1},\ldots,\eta_{N})\prod_{i=1}^{s}(e^{\eta_i})^{n_{ij}}.\]  So let \begin{equation}\label{dfnf} f(z_1,\ldots,z_N)=\sum_{j=1}^{k}g_{\beta_j}(z_{s+1},\ldots,z_{N})\prod_{i=1}^{s}z_{i}^{n_{ij}}.\end{equation}  The above discussion and the definitions show that we will have $\ker\phi_{\eta}\in D(g)$ whenever $f(\EB(\eta))\neq 0$.  So the proof will be complete if we establish that $f$ is not the zero polynomial.  But, recalling that by definition $H_2(M)=\frac{H_2(M;\mathbb{Z})}{torsion}$, since $\{\Delta_1,\ldots,\Delta_s\}$ is a basis for $H_{2n-2}(M)$, the map \begin{align*}  H_2(M)&\to \mathbb{Z}^s \\ \beta &\mapsto (\Delta_1\cap\beta,\ldots,\Delta_s\cap \beta)\end{align*} is injective.  Consequently the only terms in (\ref{dfnf}) with powers of $z_{1},\ldots,z_{s}$ respectively equal to $n_{i1},\ldots,n_{is}$  are those arising from $j=1$.  So since $g_{\beta_1}$ is not the zero polynomial it follows that $f$ is not the zero polynomial and we are done. 
\end{proof}

Similarly, we have 
\begin{theorem} \label{smallsymp}For a closed symplectic manifold $(M,\omega)$, the following are equivalent:\begin{itemize} \item[(i)] For some $\eta\in H_{2n-2}(M;\Lambda_{\omega}^{0})$, the $\eta$-twisted small quantum homology $QH(M,\omega)_{\eta}$ is a semisimple $\Lambda_{\omega}$-algebra (resp. has a field as a direct summand).
\item[(ii)] Where $\mathcal{C}$ is the deformation class of $\omega$, $(M,\mathcal{C})$ has generically semisimple small quantum homology (resp. has generically field-split small quantum homology).
\item[(iii)]  There is a nonzero Laurent polynomial \[ f\in \mathbb{Q}[z_1,z_{1}^{-1},\ldots,z_s,z_{s}^{-1}]\] such that, for all $\eta\in H_{2n-2}(M;\mathbb{C})$ such that $f(\EB^{0}(\eta))\neq 0$, the $\eta$-twisted small quantum homology $QH(M,\omega)_{\eta}$ is a semisimple $\Lambda_{\omega}$-algebra (resp. has a field as a direct summand).\end{itemize}
\end{theorem}
\begin{proof} The proof differs only notationally from that of Theorem \ref{bigsymp} and so is left to the reader.
\end{proof}

A reader who still prefers to work with undeformed (i.e. $\eta=0$) quantum homology may take solace in the following, which is somewhat reminiscent of \cite[Theorem 4.1]{OT} and \cite[Proposition 8.8]{FOOO10}:

\begin{prop}\label{sympclass} Given a deformation class $(M,\mathcal{C})$ of closed symplectic manifolds, the following are equivalent:
\begin{itemize}\item[(i)] There exists a symplectic form $\omega\in\mathcal{C}$ such that the undeformed quantum homology $QH(M,\omega)_0$ is semisimple (resp. has a field as a direct summand).
\item[(ii)] $(M,\mathcal{C})$ has generically semisimple (resp. generically field-split) small quantum homology.
\item[(iii)] Where $[\mathcal{C}]=\{[\omega]\in H^2(M;\mathbb{R})|\omega\in\mathcal{C}\}$, there is a countable intersection $\mathcal{B}$ of open dense subsets of $[\mathcal{C}]$  such that $QH(M,\omega)_0$ is semisimple (resp. has a field as a direct summand) whenever $\omega\in\mathcal{C}$ and $[\omega]\in \mathcal{B}$.\end{itemize}
\end{prop}

\begin{proof} Again (i)$\Rightarrow$(ii) follows from Proposition \ref{basechange}(i) and (iii)$\Rightarrow$(i) is trivial so we just need to prove (ii)$\Rightarrow$(iii).   So assume that $(M,\mathcal{C})$ has generically semisimple (resp. generically field-split) small quantum homology.

We may then choose a basis for $H^2(M;\mathbb{Q})$, let $\mathcal{B}_0\subset H^2(M;\mathbb{R})$ be the set of classes having rationally independent coefficients when written in terms of this basis, and let $\mathcal{B}=\mathcal{B}_0\cap [\mathcal{C}]$.  Since $\mathcal{B}_0$ is a countable intersection of open dense subsets of $H^2(M;\mathbb{R})$ and $[\mathcal{C}]$ is open (because nondegeneracy is an open condition on a $2$-form), $\mathcal{B}$ is a countable intersection of open dense subsets of $[\mathcal{C}]$.  Moreover if $[\omega]\in \mathcal{B}$ then $\beta\mapsto \langle[\omega],\beta\rangle$ is an injective map $C^{eff}\to\mathbb{R}$, and so the map  \begin{align*} \psi_{\omega}\co R_{M}^{0}&\to \Lambda_{\omega} \\ \sum_{\beta\in C^{eff}}c_{\beta}q^{\beta}&\mapsto \sum_{\beta\in C^{eff}}c_{\beta}T^{\langle[\omega],\beta\rangle}\end{align*}  is also injective.  So since $QH(M,\omega)_0=QH(M;R_{M}^{0})\otimes_{R_{M}^{0}}\Lambda_{\omega}$ where $\Lambda_{\omega}$ is made into a $R_{M}^{0}$-module by $\psi_{\omega}$, it follows from Proposition \ref{basechange}(ii) that, for $[\omega]\in\mathcal{B}$,   
$QH(M,\omega)_0$ is semisimple (which, since its coefficient ring is a field, is equivalent to being generically semisimple) if and only if $(M,\mathcal{C})$ has generically semisimple small quantum homology.
\end{proof}
\begin{remark}
There is an obvious isomorphism between our Novikov field $\Lambda_{\omega}$ and the field denoted by $\mathcal{K}_{\Gamma}$ in \cite{EP08}, and the algebra $\Lambda_{\omega}$-algebra $QH(M,\omega)_0$ is straightforwardly seen to be isomorphic to the $\mathcal{K}_{\Gamma}$-algebra $QH_{2n}(M,\omega)$ from \cite{EP08} (in \cite{EP08} a degree-shifting parameter $q$, which does not belong to $\mathcal{K}_{\Gamma}$, is used to move all of the even-degree homology into degree $2n$; note that the ring denoted by $\Lambda_{\Gamma}$  in \cite{EP08} plays a different role than our $\Lambda_{\omega}$).  In particular the notion of semisimple quantum homology from \cite{EP08} is equivalent to, in our notation, the property that $QH(M,\omega)_0$ is semisimple.  In turn, in the case that $(M,\omega)$ is monotone, this notion can be identified with that in \cite{EP03} by the argument in \cite[Section 5]{EP08}.

  In \cite{OT}, the authors use a slightly different convention for the Novikov ring, in that they consider the ring $\mathbb{K}^{\downarrow}$ in which the exponents are allowed to be arbitrary real numbers rather than being restricted to the period group $\Gamma=\Gamma_{\omega}$.  However, as follows from Proposition \ref{basechange} (or, indeed, \cite[Proposition 2.2]{EP08}), this distinction does not affect whether the quantum homology is semisimple or field-split provided that one works in characteristic zero. 

In particular, it follows from this that the symplectic manifolds that were found to have semisimple or field-split quantum homology in \cite{EP03}, \cite{EP08}, or \cite{OT} all fall under the purview of Theorem \ref{smallsymp} and Proposition \ref{sympclass}.

\end{remark}

\subsubsection{The case of convergent structure constants}\label{cvgsect}

Having proven results relating generic semisimplicity in the sense of Definition \ref{ssdfn} to quantum homology over the Novikov rings used in symplectic topology, we now connect Definition \ref{ssdfn} to semisimplicity as it is studied by algebraic geometers. 
 As before, we will consider a basis $\Delta_0,\ldots,\Delta_N$ for $H_{ev}(M)$ with $\Delta_0=[M]$ and $\Delta_1,\ldots,\Delta_s$ spanning $H_{2n-2}(M)$.
Consistently with algebraic geometry conventions, we will choose $\Delta_1,\ldots,\Delta_s$ to be ``nef'' in the sense that $\Delta_i\cap \beta\geq 0$ whenever $\beta\in H_{2}^{eff}$ (it's straightforward to find such a basis, regardless of whether the symplectic deformation class $(M,\mathcal{C})$  
arises from algebraic geometry: choose $\Delta_1$ equal to the Poincar\'e dual to a symplectic form in $\mathcal{C}$ representing a primitive integral homology class; complete this to an integral basis $\Delta_1,\Delta'_2,\ldots,\Delta'_s$ for $H_{2n-2}(M)$, and then for some large integer $K$ set $\Delta_j=\Delta'_j+K\Delta_1$ for $2\leq j\leq s$).

As far as I can tell, there is not a universal consensus in the algebraic geometry community regarding the most appropriate coefficient ring for quantum homology.  Some authors use the Novikov rings of Section \ref{novikov}.  In other cases, a formal power series ring of the form $\mathbb{Q}[[t_1,\ldots,t_N]]$ is used (see for instance \cite[Section 2]{Ir}, \cite{Bay}); in this case the quantum homology can be described in our language as being obtained from the $R_M$-algebra $\mathcal{A}_M$ induced by the coefficient extension \begin{align*} \Phi\co R_M &\to \mathbb{Q}[[t_1,\ldots,t_N]] \\ \sum_{\beta\in C^{eff}}g_{\beta}q^{\beta}&\mapsto \sum_{\beta\in C^{eff}}g_{\beta}(t_{s+1},\ldots,t_N)\prod_{i=1}^{s}t_{i}^{\Delta_i\cap\beta}.\end{align*}  Since $\Phi$ is injective (because $\beta\mapsto (\Delta_1\cap\beta,\ldots,\Delta_s\cap \beta)$ is injective, as was noted in the proof of Theorem \ref{bigsymp}) it immediately follows from Lemma \ref{basechange}(ii) that $(M,\mathcal{C})$ has generically semisimple (resp. generically field-split) quantum homology in our sense if and only if 
this coefficient extension over $\mathbb{Q}[[t_1,\ldots,t_N]]$ is generically semisimple (resp. generically field-split).

The context in which semisimple quantum homology has been of greatest interest in algebraic geometry is when the power series that appear in the algebra converge, so that the coefficient ring may be taken to be $\mathbb{C}$; in this case the quantum homology gives a Frobenius manifold (rather than a formal Frobenius manifold in the terminology of \cite{Man}), and the Frobenius manifolds obtained in the semisimple case have remarkable relations to disparate areas of mathematics (see for instance \cite{Dub}).  

Where $\ep>0$ and $B_{\ep}(\vec{0})$ denotes the ball of radius $\ep$ around the origin in $\mathbb{C}^N$, let \[ S_{\ep}=\{f\in \mathbb{Q}[[t_1,\ldots,t_N]]|f \mbox{ is absolutely convergent on }B_{\ep}(\vec{0})\} \] Also let \[ R_{M,\ep}=\Phi^{-1}(S_{\ep}).\]  Clearly $R_{M,\ep}$ and $S_{\ep}$ are rings, and $\Phi$ restricts to an injective map $\Phi\co R_{M,\ep}\to S_{\ep}$.  Similarly define $R_{M,\ep}^{0}=R_{M}^{0}\cap R_{M,\ep}$.

\begin{dfn} 
Given $\ep>0$, we say that $(M,\mathcal{C})$ has $\ep$-convergent big (resp. small) quantum homology if the structure constants $c_{ijk}$ of the start of Section \ref{mult} (resp. the elements $\sigma(c_{ijk})\in R_{M}^{0}$ where $\sigma$ is defined at the start of Section \ref{sqh}) belong to $R_{M,\ep}$ (resp. $R_{M,\ep}^{0}$).
\end{dfn}

If $(M,\mathcal{C})$ has $\ep$-convergent big quantum homology, then we may define $\AR^{\ep}$ to be the $R_{M,\ep}$-algebra freely generated as a module by the $\Delta_i$ with $\Delta_i\ast \Delta_j=\sum_k c_{ijk}\Delta^k$, so that obviously $\AR=\AR^{\ep}\otimes_{R_{M}^{\ep}}R_M$.  Similarly we may define a version $\AR^{\ep,0}$ of the small quantum homology with coefficients in $R_{M,\ep}^{0}$ so that the small quantum homology as we originally defined it, namely $QH(M;R_{M}^{0})$, is given by   $QH(M;R_{M}^{0})=\AR^{\ep,0}\otimes_{R_{M,\ep}^{0}}R_{M}^{0}$.  

Incidentally, note that if $(M,\mathcal{C})$ is symplectically Fano (i.e. if there is a symplectic form in $\mathcal{C}$ representing $c_1(TM)$; this subsumes all cases where $M$ is Fano in the complex algebraic sense), then the structure constants $\sigma(c_{ijk})$ for the \emph{small} quantum homology are all just finite sums.  Indeed, the $\sigma(c_{ijk})$ involve only Gromov--Witten invariants with three insertions, and (\ref{gwdim}) imposes an upper bound on $\langle c_1(TM),\beta\rangle$ for the homology class $\beta\in H_2(M)$ of the curves counted by such invariants; in the symplectically Fano case Gromov compactness then implies that there can be only finitely many $\beta$ represented by pseudoholomorphic curves which obey this bound. So in this case $(M,\mathcal{C})$ has $\ep$-convergent small quantum homology for every $\ep>0$.  

The definition of $R_{M,\ep}$ ensures that, for any $\vec{t}\in B_{\ep}(\vec{0})\in \mathbb{C}^N$, we have a well-defined ring map \[ ev_{\vec{t}}\co R_{M,\ep}\to \mathbb{C}\] defined by \[ ev_{\vec{t}}\left(\sum_{\beta}g_{\beta}q^{\beta}\right)= \sum_{\beta\in C^{eff}}g_{\beta}(t_{s+1},\ldots,t_N)\prod_{i=1}^{s}t_{i}^{\Delta_i\cap\beta}.\]  By the same token, if $\vec{z}\in\mathbb{C}^{s}$ with $\|\vec{z}\|<\ep$, we have a ring map \[ ev_{\vec{z}}\co R_{M,\ep}^{0}\to \mathbb{C}\] defined by \[ ev_{\vec{z}}\left(\sum_{\beta}g_{\beta}q^{\beta}\right)=\sum_{\beta\in C^{eff}}g_{\beta}\prod_{i=1}^{s}z_{i}^{\Delta_i\cap\beta}.\]

\begin{dfn} If $\vec{t}\in \mathbb{C}^N$ (resp. $\vec{z}\in \mathbb{C}^s$), the big (resp. small) \emph{quantum homology of $(M,\mathcal{C})$ at} $\vec{t}$ (resp. at $\vec{z}$) is the algebra defined by $QH(M)|_{\vec{t}}=\AR^{\ep}\otimes_{R_{M,\ep}}\mathbb{C}$ (resp.
$QH(M)|_{\vec{z}}=\AR^{\ep,0}\otimes_{R_{M,\ep}^{0}}\mathbb{C}$), where $\mathbb{C}$ has been made into an algebra over $R_{M,\ep}$ by the map $ev_{\vec{t}}$ (resp. $\mathbb{C}$  has been made into an algebra over $R_{M,\ep}^{0}$ by the map $ev_{\vec{z}}$).
\end{dfn}

\begin{theorem}\label{cvgthm} If $(M,\mathcal{C})$ has $\ep$-convergent big quantum homology, the following are equivalent:\begin{itemize}
\item[(i)] There exists $\vec{t}\in B_{\ep}(\vec{0})$ such that $QH(M)|_{\vec{t}}$ is a semisimple $\mathbb{C}$-algebra (resp. contains $\mathbb{C}$ as a direct summand).
\item[(ii)] $(M,\mathcal{C})$ has generically semisimple (resp. generically field-split) big quantum homology.
\item[(iii)] There is a nonzero analytic function $f\co B_{\ep}(\vec{0})\to \mathbb{C}$ such that  $QH(M)|_{\vec{t}}$ is a semisimple $\mathbb{C}$-algebra (resp. contains $\mathbb{C}$ as a direct summand) whenever $f(\vec{t})\neq 0$.
\end{itemize}
\end{theorem}

There is an essentially identical theorem for small quantum homology, whose statement is left to the reader.

\begin{proof} As has become customary in this paper, (i)$\Rightarrow$(ii) by Theorem \ref{basechange}, while (iii)$\Rightarrow$(i) is trivial, so we just need to prove (ii)$\Rightarrow$(iii).

Assume that $(M,\mathcal{C})$ has generically semisimple (resp. generically field-split) big quantum homology.  Since $\AR=\AR^{\ep}\otimes_{R_{M,\ep}}R_M$, it follows from Theorem \ref{basechange}(i) that the $R_{M,\ep}$-algebra $\AR^{\ep}$ is generically semisimple (resp. generically field-split).  So Theorem \ref{mainalg} produces an open set $U_1$ (resp. $U_2$) in $\S R_{M,\ep}$, which necessarily contains an open set of the form $D(f)=\{\p\in \S R_{M,\ep}|f\notin \p\}$ where $f\neq 0$.  By an argument that we have used before, the point $\vec{t}\in B_{\ep}(\vec{0})$ will have the property that $QH(M)|_{\vec{t}}$ is a semisimple $\mathbb{C}$-algebra (resp. contains $\mathbb{C}$ as a direct summand\footnote{Typically the condition is ``contains a field as a direct summand,'' but in any event this field is a finite extension of the base field, and here the base field in $\mathbb{C}$, whose only finite extension is itself}) provided that $\ker ev_{\vec{t}}\in D(f)$, i.e. provided that $f\notin \ker ev_{\vec{t}}$.  But by definition elements $f$ of $R_{M,\ep}$ are power series which define analytic functions on $B_{\ep}(\vec{0})$, and $ev_{\vec{t}}$ is just given by evaluating such a function at $\vec{t}$.  Thus our condition on $\vec{t}$ is simply that, viewing $f$ now as an analytic function, $f(\vec{t})\neq 0$.   
\end{proof}

\subsubsection{Examples from the literature}\label{lit}

From our above results we can immediately read off from the literature some broad families examples of deformation classes $(M,\mathcal{C})$ which have generically semisimple (small or big) quantum homology:
\begin{itemize} \item Any symplectic toric Fano manifold has generically semisimple small quantum homology.  This follows from the Batyrev-Givental formula for the quantum homology of such a manifold as re-expressed in, \emph{e.g.}, \cite{FOOO10}, \cite{OT}; in particular, in light of Proposition \ref{sympclass} above, we can simply read off this conclusion from \cite[Theorem 4.1]{OT}.
\item Any closed symplectic toric manifold, Fano or not, has generically semisimple big quantum homology.  Indeed, Delzant's theorem \cite{Del} shows that any closed symplectic toric manifold is deformation equivalent to a projective toric manifold (isotope the Delzant polytope to have integral vertices), and Iritani showed \cite[Theorem 1.3]{Ir} that the big quantum homology of a projective toric manifold has convergent structure constants and is generically semisimple in the sense considered in Theorem \ref{cvgthm}.
\item Of the 59 Fano $3$-folds which have no odd rational cohomology, 36 of them were shown to have generically semisimple small quantum homology (over
$\mathbb{C}$) in \cite{Cio}; by the small-quantum-homology version of Theorem \ref{cvgthm} this is equivalent to generically semisimple small quantum homology in our sense (i.e. over $R_{M}^{0}$).
\item It was shown in \cite{Bay} that a blowup at a point of a manifold with generically semisimple quantum homology still has generically semisimple quantum homology (the theorem is stated for big quantum homology, but the proof works equally well for small quantum homology).  Bayer works over a formal power series ring into which $R_M$ admits an embedding as in Section \ref{cvgsect}, so once again this is equivalent to semisimple quantum homology in our sense. 
\end{itemize}

On the negative side, it was observed in \cite{HMT} that a projective algebraic manifold $M$ cannot have generically semisimple quantum homology if there are any nonzero Hodge numbers $h^{p,q}(M)$ with $p\neq q$. In the symplectic category it remains true that a symplectic manifold cannot have generically semisimple quantum homology if it has any nonzero odd Betti numbers (the point is that the product of a Poincar\'e dual pair of odd homology classes would be a nonzero even homology class, which would however be nilpotent by the supercommutativity properties of the Gromov--Witten invariants).  On the other hand nonzero odd Betti numbers do not give any particular obstruction to the quantum homology of $M$ being generically field-split, as Theorem \ref{blowup} below demonstrates.

Also, all evidence points to the notion that one cannot delete the  word ``generically'' from the above discussion.  For example, there is given in  \cite[Section 5]{OT} an example of a monotone Fano toric $4$-fold (thus of $8$ real dimensions) whose \emph{untwisted} small quantum homology is not semisimple. 

\subsection{Symplectic blow-ups}\label{blsect}

The purpose of this final subsection is to prove the following, thus establishing Theorem \ref{examples}(ii):

\begin{theorem}\label{blowup} Let $(\tilde{M},\mathcal{C})$ be a deformation class of symplectic manifolds obtained by blowing up a symplectic manifold $(M,\omega)$ at a point.  Then $(\tilde{M},\mathcal{C})$ has generically field-split small quantum homology.\end{theorem}

As mentioned earlier, Bayer showed in \cite{Bay} that if $M$ has generically semisimple quantum homology then so does $\tilde{M}$.  At the other extreme, if $M$ is not uniruled, then the \emph{undeformed} quantum homology of the blowup has a field direct summand; this fact is proven based on results of \cite{M} in \cite[Section 3]{EP08}, where its discovery is attributed to McDuff.

Our proof of Theorem \ref{blowup} will be based on largely the same approach used by Bayer in his proof of the semisimple case.  
Let $E\in H_{2n-2}(\tilde{M})$ denote the class of the exceptional divisor, and for $j\geq 1$ abbreviate $E^{\cap j}=E_{j}$.  We have a splitting \[ H_{ev}(\tilde{M})=H_{ev}(M)\oplus span\{E_1,\ldots,E_{n-1}\}\] which is orthogonal with respect to the classical cap product. Also let $E'$ be the class of a line in the exceptional divisor $E$; thus \[ E'=(-1)^{n}E_{n-1}.\]  Let $\Delta_0,\ldots,\Delta_N$ be a basis of the usual form for $H_{ev}(M)$ with $\Delta_1,\ldots,\Delta_s$ a basis for $H_2(M)$, so that $\Delta_1,\ldots,\Delta_s,E'$ is a basis for $H_2(\tilde{M})$.

The standard universal coefficient ring $R_{\tilde{M}}^{0}$ for the small quantum homology of $\tilde{M}$, according to the conventions that we have used so far, consists of formal sums \[ \sum_{\beta\in C^{eff}} c_{\beta}q^{\beta}\] where $c_{\beta}\in\mathbb{Q}$, and where here and below $C^{eff}$ refers to the GW-effective cone of $\tilde{M}$ (not of $M$).  Following \cite{Bay}, we formally adjoin to this ring an invertible element \[ Z=q^{-\frac{1}{n-1}E'} \] to obtain a ring $R_{\tilde{M}}^{0}[q^{\pm\frac{E'}{n-1}}]$.  If we decompose a general element $\beta\in H_2(\tilde{M})$ as \[ \beta=\beta'+d_{\beta}E' \quad \beta'\in H_2(M),d_{\beta}\in\mathbb{Z},\] then a general element of   $R_{\tilde{M}}^{0}[q^{\pm\frac{E'}{n-1}}]$ may be written \[ \sum_{k=-K}^{K}\sum_{\beta\in C^{eff}}c_{\beta,k}q^{\beta'}Z^{k-(n-1)d_{\beta}}\] where the natural number $K$ depends on the particular element.  Now let \[ B=\left\{\left.\sum_{k=-K}^{K}\sum_{\beta\in C^{eff}}c_{\beta,k}q^{\beta'}Z^{k-(n-1)d_{\beta}}\in R_{\tilde{M}}^{0}[q^{\pm\frac{E'}{n-1}}]\right|k-(n-1)d_{\beta}\geq 0\mbox{ whenever }c_{\beta,k}\neq 0\right\}.\]  In other words, $B$ consists of those elements $c=\sum c_{\beta,k}q^{\beta}Z^k$ of 
$R_{\tilde{M}}^{0}[q^{\pm\frac{E'}{n-1}}]$ such that, when the $q^{\beta}$ appearing in the sum are broken up as $q^{\beta'}Z^{l}$ where $\beta'\in H_2(M)$, all powers of $Z$ appearing in the expansion of $c$ are nonnegative.

\begin{lemma}\label{blowlemma}(cf. \cite[Lemma 3.4.2]{Bay}) Let \[ N=span_{B}\{H_{ev}(M),ZE,Z^2E^2,\ldots,Z^{n-1}E^{n-1}\}.\]  (Thus $N$ is an additive subgroup of the small quantum homology $QH(\tilde{M};R_{\tilde{M}}^{0}[q^{\pm\frac{E'}{n-1}}])$.)  Then $N$ is closed under quantum multiplication, and so quantum multiplication makes $N$ into a $B$-algebra.  Moreover, this $B$-algebra $N$ is generically field-split.  
\end{lemma}

Theorem \ref{blowup} immediately follows from Lemma \ref{blowlemma} and Proposition \ref{basechange}.  Indeed, assuming the lemma, we have $QH(\tilde{M};R_{\tilde{M}}^{0}[q^{\pm\frac{E'}{n-1}}])=N\otimes_BR_{\tilde{M}}^{0}[q^{\pm\frac{E'}{n-1}}]$ as algebras, so  Proposition \ref{basechange}(ii) applied to the inclusion $B\to R_{\tilde{M}}^{0}[q^{\pm\frac{E'}{n-1}}]$ implies that $QH(\tilde{M};R_{\tilde{M}}^{0}[q^{\pm\frac{E'}{n-1}}])$ is generically field-split.  Then applying Proposition \ref{basechange} to the inclusion map $R_{\tilde{M}}^{0}\to  R_{\tilde{M}}^{0}[q^{\pm\frac{E'}{n-1}}]$ shows that $\tilde{M}$ has generically field-split small quantum homology (as defined in Definition \ref{sqdfn}).  (The relations between the various rings involved here were summarized in (\ref{bldiag}).)

\begin{proof}[Proof of Lemma \ref{blowlemma}] As in the corresponding result in \cite{Bay}, we will use properties of the Gromov--Witten invariants of blowups that were discovered by Gathmann \cite{Ga} in the context of convex algebraic varieties, and which were extended to the symplectic case by Hu and McDuff \cite{Hu},\cite{M}.  (To be specific, we will require extensions to the context of genus zero symplectic Gromov--Witten invariants of Lemmas 2.2 and 2.4 and Proposition 3.1 of \cite{Ga}.  Of these, Lemma 2.2 is generalized to the symplectic context by \cite[Theorem 1.2]{Hu}, while \cite[Lemma 2.4(i)]{Ga} is generalized by \cite[Lemma 1.1]{Hu} and \cite[Lemma 2.4(ii)]{Ga} is generalized by \cite[Lemma 2.3]{M}.  Meanwhile Gathmann's proof of his Proposition 3.1 depends only on these other results together with the splitting axiom \cite[pp. 224-225]{MS} for Gromov--Witten invariants, which of course also extends to the symplectic case.)

Having started with a basis $\Delta_0,\ldots,\Delta_N$ for $H_{ev}(M)$, we have a standard basis \[ \Delta_0,\Delta_1,\ldots,\Delta_s,E_{n-1},\Delta_{s+1},\ldots,\Delta_N,E_1,\ldots,E_{n-2},\] where  $\Delta_1,\ldots,E_{n-1}$ form a basis for $H_2(\tilde{M})$.  If $\Delta^0,\ldots,\Delta^N$ is a Poincar\'e dual basis for the above basis $H_{ev}(M)$, then our basis for $H_{ev}(\tilde{M})$ will have Poincar\'e dual basis \[ \Delta^0,\ldots\Delta^s,(-1)^{n-1}E,\Delta^{s+1},\ldots,\Delta^N,(-1)^{n-1}E_{n-1},\ldots,(-1)^{n-1}E_2.\]

Consider a (small) quantum product of elements $\Delta_i,\Delta_j\in H_{ev}(\tilde{M})$ which come from classes in $H_{ev}(M)$. We have \begin{equation}\label{didj} \Delta_i\ast\Delta_j=\sum_{k=0}^{N}\sum_{\beta\in C^{eff}}\langle\Delta_i,\Delta_j,\Delta_k\rangle_{0,3,\beta}q^{\beta}\Delta^k+(-1)^{n-1}\sum_{k=1}^{n-1}\sum_{\beta\in C^{eff}}\langle\Delta_i,\Delta_j,E_k\rangle_{0,3,\beta}q^{\beta}E_{n-k}.\end{equation}

With respect to invariants of the form $\langle\Delta_i,\Delta_j,\Delta_k\rangle_{0,3,\beta}$, \cite[Theorem 1.2]{Hu} shows that, whenever $\beta$ belongs to the subgroup $H_2(M)\leq H_{2}(\tilde{M})$, such a Gromov--Witten invariant is equal to the corresponding Gromov--Witten invariant in $M$.  As for classes $\beta$ not belonging to $H_2(\tilde{M})$, if such a class has the form $\beta=\beta'+d_{\beta}E'$ where $\beta'\in H_2(M)$ with $\beta'\neq 0$, \cite[Proposition 3.1(ii),(iii)]{Ga} shows that the invariant vanishes unless $d_{\beta}<0$.  Now we have $ q^{\beta}=q^{\beta'}Z^{-(n-1)d_{\beta}}$, so the term $\langle \Delta_i,\Delta_j,\Delta_k\rangle_{0,3,\beta}q^{\beta}\Delta^k$ belongs to $N$, and indeed belongs to the $B$-submodule $Z^{n-1}N$ of $N$.  Meanwhile if $\beta=d_{\beta}E'$ then \cite[Lemma 1.1]{Hu} shows that the invariant $\langle \Delta_i,\Delta_j,\Delta_k\rangle_{0,3,\beta}$ is zero.  Thus all of the terms in (\ref{didj}) arising from invariants 
$\langle \Delta_i,\Delta_j,\Delta_k\rangle_{0,3,\beta}$ with $\beta\notin H_2(M)$ contribute terms belonging to the submodule $Z^{n-1}N\leq N$, while all of the terms arising from $\langle \Delta_i,\Delta_j,\Delta_k\rangle_{0,3,\beta}$ with $\beta\in H_2(M)$ contribute terms in $H_2(M)\leq N$.

Now we consider the invariants $\langle\Delta_i,\Delta_j,E_k\rangle_{0,3,\beta}$ appearing in (\ref{didj}).  Again writing $\beta=\beta'+d_{\beta}E'$, if $\beta'=0$ then by \cite[Lemma 1.1]{Hu} the invariant vanishes.  So we may assume $\beta'\neq 0$.    Then Gathmann's vanishing theorem \cite[Proposition 3.1]{Ga} shows that, in order  that  $\langle\Delta_i,\Delta_j,E_k\rangle_{0,3,\beta}\neq 0$, we must have \[ k-1\geq (d_{\beta}+1)(n-1),\mbox{ and therefore }(n-1)d_{\beta}\leq k-n.\] Bearing in mind that $q^{\beta}=q^{\beta'}Z^{-(n-1)d_{\beta}}$, this shows that the term in $\Delta_i\ast \Delta_j$ corresponding to any such invariant has $E_{n-k}$  multiplied by some $Z^l$ where $l\geq n-k$.  In particular such terms always give rise to elements of $N$.  This completes the analysis of the various terms of $\Delta_i\ast\Delta_j$ and proves that, for all $i,j$, \begin{equation}\label{dd} \Delta_i\ast\Delta_j\in N.\end{equation}

Now consider a quantum product of elements $\Delta_i,Z^jE_j\in N$, where $i\geq 1$, $j\geq 2$. We have   \begin{equation}\label{diej} \Delta_i\ast Z^jE_j=\sum_{k=0}^{N}\sum_{\beta\in C^{eff}}\langle\Delta_i,E_j,\Delta_k\rangle_{0,3,\beta}q^{\beta}Z^j\Delta^k+(-1)^{n-1}\sum_{k=1}^{n-1}\sum_{\beta\in C^{eff}}\langle\Delta_i,E_j,E_k\rangle_{0,3,\beta}q^{\beta}Z^jE_{n-k}.\end{equation}

For any of the Gromov--Witten invariants appearing in (\ref{diej}), we write $\beta=\beta'+d_{\beta}E'$ where $\beta'\in H_2(M)$ and $d_{\beta}\in\mathbb{Z}$.  According to \cite[Lemma 1.1]{Hu}, an invariant $\langle\Delta_i,E_j,\Delta_k\rangle_{0,3,\beta}$ or 
$\langle\Delta_i,E_j,E_k\rangle_{0,3,\beta}$ necessarily vanishes if $\beta'=0$, so we assume that $\beta'\neq 0$.  In this case, since we assume $j\geq 2$ we may apply Gathmann's vanishing theorem \cite[Proposition 3.1]{Ga} to infer the following: the invariants   $\langle\Delta_i,E_j,\Delta_k\rangle_{0,3,\beta}$ vanish unless $(n-1)d_{\beta}\leq j-n$, while the invariants $\langle\Delta_i,E_j,E_k\rangle_{0,3,\beta}$  vanish unless $(n-1)d_{\beta}\leq j+k-n-1$.   The invariants of the former type lead in (\ref{diej}) to a term in which $\Delta^k$ is multiplied by a power of $Z$ at least equal to $n$, while the invariants of the latter type lead to a term in which $E_{n-k}$ is multiplied by a power of $Z$ at least equal to $j+(n+1-j-k)=n-k+1$.  Hence, recalling that $N$ is by definition spanned over $B$ by generators $\Delta_k,Z^{l}E_l$, it follows that \begin{equation}\label{dze} \Delta_k\ast Z^jE_j\in ZN\quad \mbox{if }k\geq 1,j\geq 2.\end{equation}

As for $\Delta_i\ast ZE_1$, consulting again (the $j=1$ version of) (\ref{diej}), note first that any invariant $\langle\Delta_i,E_1,\Delta_k\rangle_{0,3,\beta}$ or $\langle\Delta_i,E_1,E_k\rangle_{0,3,\beta}$ vanishes unless $\beta'\neq 0$ by \cite[Lemma 1.1]{Hu}.  So we assume that $\beta'\neq 0$, in which case for some $l\in\{1,\ldots,s\}$ we will have $\Delta_l\cap \beta\neq 0$.  In this case we can use the divisor axiom twice to obtain \[ \langle\Delta_i,E_1,\Delta_k\rangle_{0,3,\beta}=-\frac{d_{\beta}}{\Delta_l\cap \beta}
\langle\Delta_i,\Delta_l,\Delta_k\rangle_{0,3,\beta},\quad 	\langle\Delta_i,E_1,E_k\rangle_{0,3,\beta}=-\frac{d_{\beta}}{\Delta_l\cap \beta}
\langle\Delta_i,\Delta_l,E_k\rangle_{0,3,\beta},\] so in particular the invariants are trivial unless $d_{\beta}\neq 0$.  We can then use Gathmann's vanishing theorem again to see that the only nonzero invariants of the first type have $d_{\beta}\leq 0$ and that the only nonzero invariants of the second type have $k-1\geq (d_{\beta}+1)(n-1)$, i.e. $-(n-1)d_{\beta}\geq n-k$.  From this it follows directly that (\ref{dze}) extends to the case $j=1$:\begin{equation} \label{dze1} \Delta_k\ast ZE_1\in ZN\quad \mbox{if }k\geq 1.\end{equation}

Finally, we consider products \begin{equation}\label{zeze} Z^iE_i\ast Z^jE_j=Z^{i+j}\left(\sum_{k=0}^{N}\sum_{\beta\in C^{eff}}\langle E_i,E_j,\Delta_k\rangle_{0,3,\beta}q^{\beta}\Delta^k+(-1)^{n-1}\sum_{k=1}^{n-1}\sum_{\beta\in C^{eff}}\langle\Delta_i,\Delta_j,E_k\rangle_{0,3,\beta}E_{n-k}\right).\end{equation}  \cite[Lemma 1.1]{Hu} and Gathmann's vanishing theorem show that the invariants $\langle E_i,E_j,\Delta_k\rangle_{0,3,\beta}$ with $k\neq 0$ can be nonzero only when $\beta'\neq 0$ and $(n-1)d_{\beta}\leq i+j-n-1$, and so they contribute terms in (\ref{zeze}) in which $\Delta^k$ is multiplied by $Z$ to a power at least $n+1$.  As for the case $k=0$ (so $\Delta_0=[\tilde{M}]$), we have $\langle E_i,E_j,[\tilde{M}]\rangle_{0,3,\beta}=0$ unless $\beta=0$ and $i+j=n$, in which case it equals $E_i\cap E_j=(-1)^{n-1}$.  Thus the contribution of the terms in (\ref{zeze}) arising from $\Delta_0$ is equal to $(-1)^{n-1}Z^{i+j}E_{i+j}$

As for the invariants $\langle E_i,E_j,E_k\rangle_{0,3,\beta}$, if $\beta'\neq 0$ Gathmann's vanishing theorem shows that they are zero unless $(n-1)d_{\beta}\leq i+j+k-n-2$, so that they contribute a term to (\ref{zeze}) in which $E^{n-k}$ is multiplied by a power at least $n-k+2$.  Meanwhile the invariants $\langle E_i,E_j,E_k\rangle_{0,3,\beta}$ with $\beta'=0$, i.e. the invariants $\langle E_i,E_j,E_k\rangle_{0,3,rE'}$ are, according to \cite[Lemma 2.3]{M}, equal to $0$ unless $r=0$ (in which case they come from the classical cap product $E_i\cap E_j=E_{i+j}$) or $r=1$, in which case they are $-1$ if $i+j+k=2n-1$ and zero otherwise.  In view of this, we have \begin{equation}\label{ee} Z^iE_i\ast Z^jE_j\in\left\{\begin{array}{ll} Z^{i+j}E_{i+j}+ZN & i+j<n \\(-1)^nZ^{i+j+1-n}E_{i+j+1-n}+ZN & i+j\geq n\end{array}\right.\end{equation}  

Combining (\ref{dd}),(\ref{dze}),(\ref{dze1}), and (\ref{ee}), it is immediate that $N$ closed under quantum multiplication, so that quantum multiplication endows $N$ with the structure of a $B$-algebra.  It remains to show that $N$ is generically field-split.  For this it suffices to find one prime ideal in the set $U_2\subset \S B$ associated to $N$ by Theorem \ref{mainalg}.  Consider the ideal $ZB\leq B$, which is evidently prime.  Define \[ \underline{B}=\frac{B}{ZB}\quad \underline{N}=N\otimes_{B}\underline{B}.\] If we let \[ C_1=span_{\underline{B}}\{\Delta_1,\ldots,\Delta_N\}, \quad C_2=span_{\underline{B}}\{E_{1},\ldots,E_{n-1}\},\] we have a direct sum decomposition of modules (not of algebras, of course) \[ \underline{N}=\underline{B}\Delta_0+C_1+C_2.\]  Of course $\Delta_0$ acts as the mutliplicative identity, and  (\ref{dze}),(\ref{dze1}) show that (as we have reduced mod $Z$) $C_1C_2=0$.  Meanwhile (\ref{ee}) shows that $C_2$ is closed under quantum multiplication, and that we have, as an algebra,  \[ C_2\cong\frac{\underline{B}[s]}{\langle s^n-(-1)^ns\rangle}\] (where the variable $s$ corresponds to $ZE$).  Moreover, as in \cite[p. 9]{Bay}, the element $Y=(-1)^nZ^{n-1}E^{n-1}=-(-s)^{n-1}\in C_2$, $Y$ acts as a multiplicative identity on $C_2$, which is thus a subalgebra.  So since $YC_1=0$, we have a direct sum splitting \emph{of algebras} \[ \underline{N}=C_2\oplus(\langle\Delta_0-Y\rangle+C_1) \] (here $\oplus$ denotes direct sum of algebras and $+$ denotes module sum).  But then when we extend coefficients to the fraction field $k$ of $\underline{B}$, we will have \[ \underline{N}\otimes_{\underline{B}}k\cong \frac{k[s]}{\langle s^n-(-1)^ns\rangle}\oplus D\]  for some $k$-algebra  
$D$.  It's obvious from the Chinese Remainder Theorem that $\frac{k[s]}{\langle s^n-(-1)^ns\rangle}$ in turn decomposes as a direct sum of a field and an algebra.  Thus $\underline{N}\otimes_{\underline{B}}k$ has a field as a direct summand.  In the notation of Theorem \ref{mainalg} we have $k=k(ZB)$, so the prime $\frak{p}=ZB$ belongs to the open set $U_2$ of Theorem \ref{mainalg}. 
\end{proof}

\appendix\section{Proof of Proposition \ref{pertmain}}\label{app}

This appendix outlines the proof of the basic properties of the perturbed moduli spaces which are used in Section \ref{bigdefsect} to give a construction of the deformed Floer boundary operator in the semipositive case without appealing to \cite{FOOO09}.  As described there, the basic strategy is to achieve transversality by means of ``domain-dependent incidence conditions'': we modify the evaluation map at the $i$th marked point on the cylinder by the time-$\tau_{\beta,i}$ flow of a vector field $V$ where $\tau_{\beta,i}$ depends on the locations of the various marked points as in (\ref{taubetadef}).  As we will see, in analyzing moduli spaces of expected dimension $0$ or $1$, one in principle encounters many strata corresponding to various configurations of Floer cylinders, holomorphic spheres, and flowlines of $V$; all of these strata except the simplest, expected, ones can be shown to be empty for generic choices of the auxiliary data.  A complete combinatorial analysis of all of these strata would be something of a notational nightmare to which we will not subject the reader,   but we will provide enough  of an outline of the required arguments  that a diligent reader who is comfortable with standard techniques such as those in \cite[Chapter 6]{MS} should be able to fill in the details.

We fix a strongly nondegenerate Hamiltonian $H_0$ and consider tuples $(H,J,V,\beta)$ where $H$ varies in a small  $C^{l+1}$-neighborhood $\mathcal{H}^{l}$ of $H_0$ in the space of those Hamiltonians whose $2$-jet with coincide with $H_0$ near each of its $1$-periodic orbits; $J$ belongs to the space $\mathcal{J}^l$ of $S^1$-families of $C^l$ $\omega$-compatible almost complex structures; the vector field $V$ varies in the space $\mathcal{V}^l$ of $C^l$ gradient-like vector fields for a fixed Morse function $g$ whose critical points are disjoint from the fixed submanifolds $f_i(N_i)$; and $\beta$ varies in the space $\mathcal{B}$ of functions introduced shortly after Definition \ref{sndeg}; recall that this space is a Banach manifold (it is diffeomorphic to an open subset of a Banach space) and that all of its members are smooth positive functions with Gaussian decay.

For any given $\gamma^-,\gamma^+\in P(H)$, $C\in \pi_2(\gamma^-,\gamma^+)$ and $I=(i_1,\ldots,i_k)\in\{1,\ldots,m\}^k$ write \[ \mathcal{U}^l(C)=\left\{(u,J,H)\left|\begin{array}{c}u\co \R\times S^1\to M,\,J\in\mathcal{J}^l,H\in\mathcal{H}^l, \\ \, \delbar_{J,H}u=0,\, \int_{\R\times S^1}\left|\frac{\partial u}{\partial s}\right|^2dsdt<\infty,\\ \, u(s,\cdot)\to\gamma^{\pm}\mbox{ as }s\to\pm\infty,\, [u]=C\in \pi_2(\gamma^-,\gamma^+)\end{array}\right.\right\}\]

and \[ \mathcal{U}^l(C,I)=\left\{(u,\vec{z},n_1,\ldots,n_k,J,H,V,\beta)\left| \begin{array}{c} (u,J,H)\in \mathcal{U}^l(C),\,\vec{z}\in (\R\times S^1)^k,\, n_j\in N_{i_j}, \\ V\in\mathcal{V}^l,\, \beta\in \mathcal{B},\\ \psi^{\tau_{\beta,j}(\vec{z})}_{V}(u(z_j))=f_{i_j}(n_j)
\end{array}\right.\right\} \]

(The reader can think of $\mathcal{U}$ as standing for ``universal moduli space.'')
We have:

\begin{prop}\label{univevsub} Assume that it is not the case that both $\gamma^-=\gamma^+$ and $C\in \pi_2(\gamma^-,\gamma^+)$ is the trivial class.  Then $\mathcal{U}^l(C)$ is a $C^{l-1}$-Banach manifold.  Moreover, for any fixed $J_0\in \mathcal{J}^l$, the subspace $\mathcal{U}^l(C;J_0)=\{(u,H):(u,J_0,H)\in \mathcal{U}^l(C)\}$ is also a $C^{l-1}$-Banach manifold,  and for any \emph{distinct} points $w_1,\ldots,w_p\in\R\times S^1$ the evaluation map \begin{align*}  ev_{w_1,\ldots,w_p}\co \mathcal{U}^l(C;J_0)&\to M^p \\  (u,H)&\mapsto (u(w_1),\ldots,u(w_p)) \end{align*} is a submersion.
\end{prop}

\begin{proof}
 The proof of \cite[Theorem 5.1 (ii)]{FHS} shows that, for fixed $J_0$, the map $(u,H)\mapsto \delbar_{J_0,H}u$ (which is a class $C^{l-1}$ map between appropriate Banach manifolds) is transverse to the zero section; by the implicit function theorem this suffices to show that both $\mathcal{U}^l(C)$ and $\mathcal{U}^{l}(C;J_0)$ are $C^{l-1}$ Banach manifolds.   The statement about the evaluation map can be proven by combining the argument used in the proof of \cite[Lemma 3.4.3]{MS} with properties of the linearization of $(u,H)\mapsto \delbar_{J,H}u$ from \cite{FHS}; see also the proof of \cite[Proposition A.1.4]{LO} for a similar argument.

\end{proof}

\begin{dfn} If $p\in\mathbb{N}$ and if $S$ is a subset of $\mathbb{N}$, a surjective map $\pi\co S\to\{1,\ldots,p\}$ is called \emph{order-respecting} if whenever $1\leq i<j\leq p$ the minimal element of $\pi^{-1}\{i\}$ is less than the minimal element of $\pi^{-1}\{j\}$.
\end{dfn}

Note that the correspondence which assigns to each order-respecting surjective $\pi\co S\to\{1,\ldots,p\}$ the collection of sets $\{\pi^{-1}\{i\}|1\leq i\leq p\}$ is a one-to-one correspondence onto the set of partitions of $S$ into $p$ disjoint subsets.

The space $\mathcal{U}^l(C,I)$  has various strata corresponding to the extent to which the marked points $z_j$ ($j=1,\ldots,k$) overlap.  We label any one of these strata by means of a surjective order-respecting map $\pi\co\{1,\ldots,k\}\to \{1,\ldots,p\}$ for some natural number $p$: the stratum $\mathcal{U}^{l}_{\pi}(C,I)$ will consist of those $(u,\vec{z},\vec{n},J,H,V,\beta)$ for which $z_{j_1}=z_{j_2}$ iff $\pi(j_1)=\pi(j_2)$.  

\begin{prop}\label{ghostsub} Fix a surjective order-respecting map $\pi\co \{1,\ldots,k\}\to \{1,\ldots,p\}$ and let $\tau_1,\ldots,\tau_k\in [0,\infty)$ have the property that $\tau_j>\tau_{j'}$ whenever $j>j'$ and $\pi(j)=\pi(j')$. Let \[ \Delta_{\pi}=\{(m_1,\ldots,m_k)\in M^k|(\exists j,j')(\pi(j)\neq \pi(j')\mbox{ and }m_j=m_{j'})\}.\] Then the map \[ \phi_{\pi,\tau_1,\ldots,\tau_k}\co M^p\times \mathcal{V}^l\to M^k \] defined by \[ \phi_{\pi,\tau_1,\ldots,\tau_k}(m_1,\ldots,m_p,V)=\left(\psi^{\tau_{1}}_{V}(m_{\pi(1)}),\ldots,\psi^{\tau_{k}}_{V}(m_{\pi(k)})\right)\] restricts to $\phi_{\pi,\tau_1,\ldots,\tau_k}^{-1}((M\setminus Crit(g))^k\setminus\Delta_{\pi})$ as a submersion.  
\end{prop}

\begin{proof} Let $(m_1,\ldots,m_p,V)\in \phi_{\pi,\tau_1,\ldots,\tau_k}^{-1}((M\setminus Crit(g))^k\setminus\Delta_{\pi})$.  Write $x_j=\psi^{\tau_{j}}_{V}(m_{\pi(j)})$, so \[ \phi_{\pi,\tau_1,\ldots,\tau_k}(m_1,\ldots,m_p,V)=(x_1,\ldots,x_k).\]  Of course, since the vector field $V$ has zero locus equal to $Crit(g)$, for each $j$ neither $m_{\pi(j)}$ nor $x_j$ lies in $V^{-1}(0)$.  Note also that the $x_j$ are all distinct points: the fact that $(x_1,\ldots,x_k)\notin \Delta_{\pi}$ immediately implies that $x_j\neq x_{j'}$ when $\pi(j)\neq \pi(j')$, while if $\pi(j)=\pi(j')$ and $j>j'$ we have by assumption $\tau_j>\tau_{j'}$ and $x_j=\psi^{\tau_j}_{V}(m_{\pi(j)})$ and $x_{j'}= \psi^{\tau_{j'}}_{V}(m_{\pi(j)})$; so since $V$ is a gradient-like vector field and $m_{\pi(j)}\notin V^{-1}(0)$ we indeed have $x_j\neq x_{j'}$.

We are to  show that if $v\in T_{x_j}M$ then the element of $T_{(x_1,\ldots,x_k)}M^k$ whose $j$th component is equal to $v$ and whose other components equal zero lies in the image of the linearization of $\phi_{\pi,\tau_1,\ldots,\tau_k}$ at $(m_1,\ldots,m_p,V)$.  Now since the $x_j$ are all distinct and lie in $V^{-1}(0)$, we can find disjoint flow boxes for $V$ around each of the $x_j$, say with the property that the integral curve of $V$ starting at $m_{\pi(j)}$ enters the flow box around $x_j$ at time $\tau_j-\ep$ and exits at time $\tau_j+\ep$; moreover we can arrange that the $j$th flow box intersects the integral curve of $V$ through one of the $m_r$ iff $m_r$ and $x_j$ lie on the same flowline of $V$.  In the case that $\tau_j\neq 0$,  it is straightforward to construct a one-parameter family of perturbations $\{V_{s,j}\}_{s\in (-\delta,\delta)}$ of $V$, each equal to $V$ outside the flow box, such that $\psi_{V_{s,j}}^{\tau_j\pm\ep}(m_{\pi(j)})=\psi_{V}^{\tau_j\pm\ep}(m_{\pi(j)})$ while $\frac{d}{ds}\psi_{V_{s,j}}^{\tau_j}(m_{\pi(j)})=v$.  Then where $\xi=\frac{dV_{s,j}}{ds}$ the element $(0,\ldots,0,\xi)$ is sent by the linearization to our desired element $(0,\ldots,v,\ldots,0)$.  Meanwhile if $\tau_j=0$ (so that $x_j=m_{\pi(j)}$) we can obtain the element $(0,\ldots,v,\ldots,0)$ as the image under the linearization of an element of form $(0,\ldots,v,\ldots,0,\xi)$ where $v\in T_{m_{\pi(j)}}M=T_{x_j}M$ and the perturbation $\xi$ of $V$ is supported in the flow boxes around the various $x_j$ with $j\neq j'$ but $\pi(j)=\pi(j')$.

\end{proof}

Now, as suggested earlier, if $\gamma^-,\gamma^+\in P(H_0)$, $C\in \pi_2(\gamma^-,\gamma^+)$, $I\in \{1,\ldots,m\}^k$, and $\pi\co \{1,\ldots,k\}\to \{1,\ldots,p\}$ is a surjective order-respecting map, let \[ \mathcal{U}^{l}_{\pi}(C,I)=\{(u,\vec{z},\vec{n},J,H,V,\beta)\in \mathcal{U}^l(C,I)|z_j=z_{j'}\Leftrightarrow \pi(j)=\pi(j')\}.\]  Also let 
\[ \mathcal{U}^{l,*}_{\pi}(C,I)=\{(u,\vec{z},\vec{n},J,H,V,\beta)\in \mathcal{U}^{l}_{\pi}(C,I)|\mbox{ if }\pi(j)\neq \pi(j')\mbox{ then }f_{i_j}(n_j)\neq f_{i_{j'}}(n_{j'})\}.\]

\begin{prop}\label{maindim}Assume that it is not the case that both $\gamma^-=\gamma^+$ and $C\in \pi_2(\gamma^-,\gamma^+)$ is the trivial class. \begin{itemize} \item[(i)]
For any surjective order-respecting map $\pi\co \{1,\ldots,k\}\to \{1,\ldots,p\}$, $\mathcal{U}^{l,*}_{\pi}(C,I)$ is a $C^{l-1}$-Banach manifold.
\item[(ii)] Let $\frak{a}=(J,H,V,\beta)\in \mathcal{J}^l\times \mathcal{H}^l\times\mathcal{V}^l\times \mathcal{B}$ be a regular value of the projection $\mathcal{U}^{l,*}_{\pi}(C,I)\to \mathcal{J}^l\times \mathcal{H}^l\times\mathcal{V}^l\times \mathcal{B}$.  Then \[ \tilde{\mathcal{M}}^{\frak{a},*}_{\pi}(\gamma^-,\gamma^+,C;N_I)=\left\{(u,\vec{z},\vec{n})|(u,\vec{z},\vec{n},J,H,V,\beta)\in \mathcal{U}^{l,*}_{\pi}(C,I)\right\}\] is a $C^{l-1}$ manifold of dimension $\barmu(C)-\delta(I)-2(k-p)$.
\item[(iii)] A residual subset of the space $\mathcal{A}=\cap_{l=2}^{\infty} \mathcal{J}^l\times \mathcal{H}^l\times\mathcal{V}^l\times \mathcal{B}$ has the property that all of its members are regular values of the projections in (ii) above for all sufficiently large $l$.\end{itemize}
\end{prop}

\begin{proof} For $r=1,\ldots,p$ write $j_r$ for the minimal element of $\pi^{-1}(r)$.

Modulo reordering of the factors, $\mathcal{U}^{l,*}_{\pi}(C,I)$ may be identified with the space of tuples \[ ((u,J,H),\vec{z},(m_1,\ldots,m_p,V),n_1,\ldots,n_k,\beta)\in \mathcal{U}^l(C)\times (\R\times S^1)^k\times (M^p\times\mathcal{V}^l)\times\left(\prod_{j=1}^{k}N_{i_j}\right)\times\mathcal{B}\] such that \begin{itemize}\item $z_{j}=z_{j'}\mbox{ if }\pi(j)=\pi(j')$, \item $ev_{z_{j_1},\ldots,z_{j_p}}(u,J,H)=(m_1,\ldots,m_p)$, \item  $\phi_{\pi,\tau_{\beta,1}(\vec{z}),\ldots,\tau_{\beta,k}(\vec{z})}(m_1,\ldots,m_p,V)=(f_{i_1}(n_1),\ldots,f_{i_k}(n_k))\in   (M\setminus Crit(g))^k\setminus\Delta_{\pi}$ \end{itemize}

The first condition above is obviously cut out transversely (and imposes a condition of codimension two for each of the $k-p$ indices $j$ which are not equal to $j_r$ for some $r$), while the second and third are cut out transversely by, respectively, Proposition \ref{univevsub} and Proposition \ref{ghostsub}.  Said differently, $\mathcal{U}^{l,*}_{\pi}(C,I)$ is identified with the preimage of the diagonal under a certain map \[
 \mathcal{U}^l(C)\times (\R\times S^1)^k\times (M^p\times\mathcal{V}^l)\times\left(\prod_{j=1}^{k}N_{i_j}\right)\times \mathcal{B}\to \left((\R\times S^1)^{k-p}\times M^p\times ((M\setminus Crit(g))^k\setminus\Delta_{\pi})\right)^2,\] and the above shows that this map is transverse to the diagonal.  Since $\mathcal{U}^l(C)\times (\R\times S^1)^k\times (M^p\times\mathcal{V}^l)\times\left(\prod_{j=1}^{k}N_{i_j}\right)\times \mathcal{B}$ is a $C^{l-1}$ Banach manifold it therefore follows from the implicit function theorem that $\mathcal{U}^{l,*}_{\pi}(C,I)$ is as well, proving (i).

As for (ii), the implicit function theorem implies that $\tilde{\mathcal{M}}^{\frak{a},*}_{\pi}(\gamma^-,\gamma^+,C;N_I)$ is a $C^{l-1}$ manifold of dimension equal to the index of the projection; we need only determine this index.  Now as in \cite[Section 2]{Sal},\cite{RS}, the index of the projection $\mathcal{U}^l(C)\to \mathcal{J}^l\times\mathcal{H}^l$ is $\barmu(C)$, while of course the identity map on $\mathcal{V}^l\times\mathcal{B}$ has index zero. So by using the characterization of the previous paragraph of  $\mathcal{U}^{l,*}_{\pi}(C,I)$ as the preimage of the diagonal under a certain map, and recalling that the manifold $N_{i_j}$ has dimension $2d(i_j)$, we calculate the dimension to be 
\begin{align*} &\left(\barmu(C)+2k+2np+\sum_{j=1}^{k}2d(i_j)\right)-\left(2(k-p)+2np+2nk\right)\\&=\barmu(C)+2k+\sum_{j=1}^{k}2d(i_j)-2nk-2(k-p)=\barmu(C)-\delta(I)-2(k-p),\end{align*} as claimed in (ii).

Finally, assertion (iii) follows from the Sard-Smale theorem (applied with $l$ sufficiently large) together with a straightforward adaptation of the argument of Taubes described on \cite[pp. 52--53]{MS} which allows one to pass from $C^l$ auxiliary data $(J,H,V)$ to $C^{\infty}$ such data.

\end{proof}

The complement $\mathcal{U}_{\pi}^{l}(C,I)\setminus \mathcal{U}_{\pi}^{l,*}(C,I)$ involves configurations in which one has, among other conditions, a Floer cylinder $u\co \R\times S^1\to M$ and \emph{distinct} marked points $z_{j},z_{j'}\in \R\times S^1$ (with $\pi(j)\neq \pi(j')$) such that $u(z_{j})$ and $u(z_{j'})$ are connected to the \emph{same} point $f_{i_j}(n_j)=f_{i_{j'}}(n_{j'})$ by prescribed-length flowlines of $V$.  This gives rise to a variety of different substrata of $\mathcal{U}_{\pi}^{l}(C,I)$ determined by precisely which indices $j$ correspond to ``duplicated'' contact points with the $f_{i_j}(N_j)$.  All of these substrata can easily be seen to have large codimension in $\mathcal{U}_{\pi}^{l}(C,I)$.  Namely, although the condition that $f_{i_j}(n_j)=f_{i_{j'}}(n_{j'})$ implies that one cannot directly appeal to Proposition \ref{ghostsub}, one can (assuming without loss of generality that $d(i_{j'})\geq d(i_j)$) forget about the incidence constraint corresponding to index $j'$, but impose the constraint that the distinct points $z_{j}$ and $z_{j'}$ are mapped by $u$ to points lying on the same flowline of the vector field $V$, and that moreover this flowline passes at the time $\tau_{\beta,j}(\vec{z})$ through the submanifold $f_{i_j}(N_j)$.  This amounts to replacing a constraint of codimension $2n-2-2d(i_{j'})$ by a constraint of codimension $2n-3+2n-2d(i_j)$; thus at least formally the codimension increases by at least $2n-1$. (Of course, we have $2n\geq 4$, since if $2n=2$ there are no ``big deformations'' to consider.)  
 Moreover the newly imposed 
constraints (on $u(z_j)$ and $u(z_{j'})$) are easily seen to be cut out transversely using Propositions \ref{univevsub} and \ref{ghostsub}.  Any additional duplicated incidence conditions may be handled by repeating this same procedure of forgetting the duplicated condition but imposing the condition that different marked points on $\R\times S^1$ are both mapped to the same  flowline of $V$ which itself satisfies various incidence conditions; it is easy to see that at each stage the expected dimension only decreases. Consequently just as in the proof of Proposition \ref{maindim}, the Sard--Smale theorem allows one to show that:

\begin{prop}\label{substrat} For a residual set of  $\frak{a}=(J,H,V,\beta)$, the set  \[ \tilde{\mathcal{M}}^{\frak{a}}_{\pi}(\gamma^-,\gamma^+,C;N_I)=\left\{(u,\vec{z},\vec{n})|(u,\vec{z},\vec{n},J,H,V,\beta)\in \mathcal{U}^{l}_{\pi}(C,I)\right\}\]  has the property that $\tilde{\mathcal{M}}^{\frak{a}}_{\pi}(\gamma^-,\gamma^+,C;N_I)\setminus  \tilde{\mathcal{M}}^{\frak{a},*}_{\pi}(\gamma^-,\gamma^+,C;N_I)$ is contained in a union of manifolds of dimension at most $\tilde{\mathcal{M}}^{\frak{a},*}_{\pi}(\gamma^-,\gamma^+,C;N_I)-(2n-1)$.  
\end{prop}

From this we quickly obtain:
\begin{cor}\label{nobubmain} Assume that it is not the case that both $\gamma^+=\gamma^-$ and $C\in \pi_2(\gamma^-,\gamma^+)$ is the trivial class.  For a residual set of $\frak{a}=(J,H,V,\beta)$, if $\bar{\mu}(C)-\delta(I)\leq 2$ then \[ \mathcal{M}^{\frak{a}}(\gamma^-,\gamma^+,C;N_I)= \mathcal{M}^{\frak{a},\ast}_{id}(\gamma^-,\gamma^+,C;N_I)\] where $id\co \{1,\ldots,k\}\to \{1,\ldots,k\}$ is the identity, and $\mathcal{M}^{\frak{a}}(\gamma^-,\gamma^+,C;N_I)$ is a smooth manifold of dimension $\barmu(C)-\delta(I)$.
\end{cor} 

\begin{proof} Since all of the various strata and substrata of $\tilde{\mathcal{M}}^{\frak{a}}(\gamma^-,\gamma^+,C;N_I)$ admit free $\R$-actions, these strata and substrata are empty unless they have positive dimension.  But if $\barmu(C)-\delta(I)\leq 2$
then Propositions \ref{maindim} and \ref{substrat} show that all strata have nonpositive dimension for generic $\frak{a}$ except when $p=k$. Since the only surjective order-respecting map $\{1,\ldots,k\}\to\{1,\ldots,k\}$ is the identity, the result follows.
\end{proof}

Of course, $\tilde{\mathcal{M}}^{\frak{a},\ast}_{id}(\gamma^-,\gamma^+,C;N_I)$ can be oriented using coherent orientations in a standard way. This therefore completes the proof of Proposition \ref{pertmain}(i).  For the remainder of Proposition \ref{pertmain} we must of course address the failure of compactness of $\mathcal{M}^{\frak{a}}(\gamma^-,\gamma^+,C;N_I)=\tilde{\mathcal{M}}^{\frak{a}}(\gamma^-,\gamma^+,C;N_I)/\R$.  The idea is familiar from \cite{HS}: the standard Gromov-Floer compactification of $\mathcal{M}^{\frak{a}}(\gamma^-,\gamma^+,C;N_I)$ involves configurations of broken trajectories and sphere bubbles; those configurations involving a two-stage broken trajectory and no sphere bubbles form a codimension-one stratum of the boundary, while all other strata have codimension at least two.  The analysis is somewhat trickier than in \cite{HS}, however, in part because in our case the possible sphere bubbles that arise in studying the boundaries of moduli spaces of dimension two can have arbitrarily large Chern number.  Indeed, the reader may have noticed that to prove Corollary \ref{nobubmain} it was not necessary let the function $\beta$ which determines the ``contact times'' $\tau_{\beta,i}$ vary in the universal moduli space; a version of Corollary \ref{nobubmain} would have held if we had simply set $\beta$ equal to (for instance) the Gaussian $s\mapsto e^{-s^2}$.  However in analyzing certain highly degenerate substrata of the compactification of $\mathcal{M}^{\frak{a}}(\gamma^-,\gamma^+,C;N_I)$ we will see that it becomes useful to vary $\beta$.

At an initial level, any stratum of the compactification of $\mathcal{M}^{\frak{a}}(\gamma^-,\gamma^+,C;N_I)$, where $I=(i_1,\ldots,i_k)\in\{1,\ldots,m\}^k$, may be described by the following data: \begin{itemize}
\item A sequence $\gamma_0=\gamma^-,\gamma_1,\ldots,\gamma_{\rho}=\gamma^+\in P(H)$
\item Classes $C_{a}\in \pi_2(\gamma_{a-1},\gamma_a)$ ($1\leq a\leq \rho$) and classes $A_1,\ldots,A_{\sigma}\in\pi_2(M)$ such that, where $\#$ denotes the obvious gluing operation, we have \[ (C_1\#\cdots\#C_{\rho})\#(A_1\#\cdots\#A_{\sigma})=C.\]
\item A function $\zeta\co \{1,\ldots,\sigma\}\to\{1,\ldots,\rho\}$ (the significance of $\zeta$ is that its domain parametrizes the (stable, possibly multi-component) sphere bubbles, while its codomain parametrizes the cylindrical components; the $s$th  bubble will be attached 
to the $\zeta(s)$th cylindrical component.)
\item A partition of the index set $\{1,\ldots,k\}$ as \[ \{1,\ldots,k\}=(S_{1}^{C}\cup\cdots\cup S^{C}_{\rho})\cup(S_{1}^{S}\cup\cdots\cup S_{\sigma}^{S})\] (this partition specifies the components onto which the various marked points fall).  We will write $I_{a}^{C}$ (resp. $I_{b}^{S}$) for the tuple consisting of those $i_j$ for $j\in S_{a}^{C}$ (resp. $S_{b}^{S}$), taken in increasing order of $j$.
\item Surjective order-respecting maps $\pi_{a}^{C}\co S_{a}^{C}\to \{1,\ldots,p_{a}^{C}\}$ and $\pi_{b}^{S}\co S_{b}^{S} \to \{1,\ldots,p_{b}^{S}\}$ for appropriate integers $p_{a}^{C}$, $p_{b}^{S}$.  (These maps play the same role as our earlier maps $\pi\co \{1,\ldots,k\}\to \{1,\ldots,p\}$). 
\end{itemize}

Any element of such a stratum of the compactification of $\tilde{\mathcal{M}}^{\frak{a}}(\gamma^-,\gamma^+,C;N_I)$ for a fixed $\frak{a}=(J,H,V,\beta)$ is determined by the following data:   
\begin{itemize}
\item[(i)] Solutions $u_a\co \R\times S^1\to M$ ($1\leq a\leq \rho$) to the equation $\delbar_{J,H}u_a=0$ which represent the classes 
$C_a\in \pi_2(\gamma_{a-1},\gamma_a)$.
\item[(ii)] Stable genus-zero $J$-holomorphic maps $\mathbf{v}_b$ ($1\leq b\leq \sigma$) with domain $D(\mathbf{v}_b)$ representing the classes $A_b\in \pi_2(M)$.  The $\mathbf{v}_b$ will be assumed to have no trivial components. (For background on stable genus-zero maps see \cite[Chapter 6]{MS}.)
\item[(iii)] For each $b=1,\ldots,\sigma$, a point $z_{b0}\in D(\mathbf{v}_b)$ and a  point $w_b\in \R\times S^1$
with the property that $u_{\zeta(b)}(w_b)=\mathbf{v}_b(z_{b0})$.
\item[(iv)] For each $c=1,\ldots,p_{a}^{C}$ (resp. $c=1,\ldots,p_{b}^{S}$), distinct points $z_{ac}^{C}\in \R\times S^1$ (resp., distinct points $z_{bc}^{S}\in D(\mathbf{v}_b)$).
\item[(v)] For each $j\in \{1,\ldots,m\}^k$, points $n_j\in N_{i_j}$\end{itemize}
These data are required to satisfy the incidence conditions which we now describe:
For $a=1,\ldots,\rho$, let $\tilde{I}_a$ denote the multi-index obtained by combining together $I_{a}^{C}$ and all of the $I_{b}^{S}$ such that $\zeta(b)=a$, and arranging the indices in the original order in which they appeared in $I$.  Define $\vec{\eta}_a\in (\R\times S^1)^{\#\tilde{I}_{a}}$ by setting the entry corresponding to an index $j\in S_{a}^{C}$ equal to $z_{a\pi_{a}^{C}(j)}^{C}$ from (iv) above, and the entry corresponding to an index $j\in S_{b}^{S}$ where $\zeta(b)=a$ equal to the point $w_b$ from (iii).  We then require that, if $j$ is the $r_{j}$th index appearing in the multi-index  $\tilde{I}_a$, we have \begin{equation}\label{geninc} f_{i_j}(n_j)=\left\{\begin{array}{ll}\psi_{V}^{\tau_{\beta,r_j}(\vec{\zeta}_a)}(u_a(z_{a\pi_{a}^{C}(j)}^{C})) & \mbox{ if }j\in S_{a}^{C} \\ \psi_{V}^{\tau_{\beta,r_j}(\vec{\eta}_a)}(\textbf{v}_b(z_{a\pi_{b}^{S}(j)}^{S})) & \mbox{ if }j\in S_{b}^{S}\mbox{ where }\zeta(s)=a
 \end{array}\right.\end{equation}

Informally, these strata thus involve various combinatorial arrangements of Floer cylinders representing the $C_a$; stable genus-zero $J$-holomorphic curves representing the $A_b$; and flowlines of the vector field $V$ which begin at marked points on the cylinders or spheres and pass through the submanifolds $f_{i_j}(N_{i_j})$ at times that are prescribed by the locations of the various marked points.  The reader will likely be relieved to learn that we do not intend to analyze these strata in full generality in the above complicated combinatorial notation; rather we will indicate the arguments that are generally used, and leave it to the reader to convince themselves that these arguments can be applied to deal with all of the strata as described above.  

Let us call an element of the compactification \emph{simple} provided that:  none of the spherical components of any of the stable curves $\mathbf{v}_b$ are multiply-covered; none of the cylindrical components are ``trivial cylinders'' $(s,t)\mapsto \gamma(t)$; all cylindrical and spherical components have distinct images; and all of the contact points $f_{i_j}(n_j)$ are distinct.  Within any of the strata described above, the space of simple elements of the compactification can be shown  to be a manifold for generic data $\frak{a}$ in much the same way as we handled $\tilde{\mathcal{M}}^{\frak{a},\ast}_{id}(\gamma^-,\gamma^+,C;N_I)$: for this purpose we appeal again to Propositions \ref{univevsub} and \ref{ghostsub} and (for the sphere bubbles) \cite[Lemma 3.4.2]{MS}.  Note that the evaluation maps for the universal moduli spaces of cylinders are made submersive by varying $H$ (in Proposition \ref{univevsub}); those for the flowlines of $V$ are made submersive by varying $V$ (in Proposition \ref{ghostsub}) and those for spheres are made submersive in \cite{MS} by varying $J$; hence by varying the tuple $(J,H,V)$ we can simultaneously achieve transversality for all of the evaluation maps at the marked points for the universal moduli spaces of simple configurations in any one of our  strata.  In view of this, if we restrict to simple configurations, arguments much like those given in \cite[Chapter 6]{MS} show that these universal moduli spaces are Banach manifolds and that, using the Sard-Smale theorem, for generic $\frak{a}$ the associated stratum of the moduli space has dimension at most, with notation as above and  after dividing by symmetry groups (given by translation of the cylindrical components and automorphisms of $S^2$ for the spherical component), \[ \barmu(C)-\delta(I)-\rho-2\sigma-2\left(k-\left(\sum_{a=1}^{\rho}p_{a}^{C}+\sum_{b=1}^{\sigma}p_{b}^{S}\right)\right).\]  Thus if, as in Proposition \ref{pertmain} (ii) and (iii), we have $\barmu(C)-\delta(I)\leq 2$, then all of these strata are (for generic $\frak{a}$) empty unless $\sigma=0$ (\emph{i.e.}, there are no sphere bubbles), $\sum_{a=1}^{\rho}p_{a}^{C}=k$, and either $\barmu(C)-\delta(I)=1$ and $\rho=1$ or $\barmu(C)-\delta(I)=2$ and $\rho\in\{1,2\}$.  In case $\barmu(C)-\delta(I)=1$, the only  stratum containing any simple configurations for generic $\frak{a}$ is thus precisely $\tilde{\mathcal{M}}^{\frak{a},\ast}_{id}(\gamma^-,\gamma^+,C;N_I)/\R$, while if $\barmu(C)-\delta(I)=2$ the only such strata are $\tilde{\mathcal{M}}^{\frak{a},\ast}_{id}(\gamma^-,\gamma^+,C;N_I)/\R$ (which has dimension $1$) together with all those strata involving two cylindrical components, no sphere bubbles, and $k$ distinct points distributed among the two cylindrical components connected by flowlines of $V$ to the appropriate $f_{i_j}(N_{i_{j}})$.  These latter strata precisely give (\ref{bdryfib}) in Proposition \ref{pertmain}.  As described in Remark \ref{mainrmk}, standard gluing arguments show that corresponding to each element of (\ref{bdryfib}) one can obtain a unique end of the space $\tilde{\mathcal{M}}^{\frak{a},\ast}_{id}(\gamma^-,\gamma^+,C;N_I)/\R$.  Consequently the proof of Proposition \ref{pertmain} will be complete if we show that the compactification of $\tilde{\mathcal{M}}^{\frak{a},\ast}_{id}(\gamma^-,\gamma^+,C;N_I)/\R$ generically does not include any nonsimple configurations when $\barmu(C)-\delta(I)\leq 2$.

Since \[ \sum_{a=1}^{\rho}\left(\left(\barmu(C_a)+\sum_{b\in\zeta^{-1}(a)}2c_1(A_b)\right)-\left(\delta(I_{a}^{C})+\sum_{b\in\zeta^{-1}(a)}\delta(I_{b}^{S})   \right)\right)=\barmu(C)-\delta(I),\] we may reduce to the case that there is just one cylindrical component, and so it suffices to prove:

\begin{prop}\label{nns} For generic $\frak{a}$ the following holds.  Consider a stratum as described above with $\rho=1$ and associated data $(C_1,I_{1}^{C},\pi_{1}^{C},\{A_b,I_{b}^{S},\pi_{b}^{S}\}_{b=1}^{\sigma})$.  Assuming that $\barmu(C)-\delta(I)\leq 2$, this stratum contains no nonsimple configurations.
\end{prop}

We complete this appendix by outlining the proof of Proposition \ref{nns}, leaving some details to the reader.  Within each of the strata of configurations as described above there are various substrata describing ways in which the configuration may fail to be simple. In effect, we show that each of these substrata is, for generic $\frak{a}$, contained in a manifold of negative dimension; this suffices since there are only countably many substrata and a countable intersection of residual sets is residual.  For the most part, the proof follows the standard strategy of associating to a nonsimple configuration by an ``underlying simple configuration'' and appealing to transversality for the underlying simple configuration; of course this 
only works if this replacement does not increase the expected dimension.

In particular, the semipositivity condition implies that for generic $J$ there will be no $J$-holomorphic spheres of negative Chern number. Thus as a first step we may replace any multiply-covered sphere bubble components by their underlying simple spheres; since the Chern numbers of these spheres are nonnegative doing so cannot increase the expected dimension of the configuration. 

In most cases, nonsimple configurations in which two or more of the $f_{i_j}(n_j)$ are equal can be handled by essentially the same method as in our earlier analysis of  $\mathcal{U}_{\pi}^{l}(C,I)\setminus \mathcal{U}_{\pi}^{l,*}(C,I)$: namely, we use the fact that if $f_{i_j}(n_j)=f_{i_{j'}}(n_{j'})$ and  if $j\in S_{a}^{C}$ and $j'\in S_{a'}^{C}$  then $u_a(z_{a\pi_{a}^{C}(j)}^{C})$ and $u_{a'}(z_{a'\pi_{a'}^{C}(j')}^{C})$ must both lie on the same flowline of $V$, and this flowline satisfies additional incidence conditions (if instead $j\in S_{b}^{S}$ and/or $j'\in S_{b'}^{S}$ for some $b,b'$ then of course a similar condition holds for $\textbf{v}_b(z_{a\pi_{b}^{S}(j)}^{S})$ and/or $\textbf{v}_{b'}(z_{a\pi_{b'}^{S}(j)}^{S})$).  Just as discussed earlier, replacing the duplicated incidence condition at $f_{i_j}(n_j)$ by this new condition lowers the expected dimension.

However there is a new complication in this analysis that did not appear earlier, namely that our configurations  may have more than one spherical component, and it might be the case that two different spherical components have the same image, in which case the new condition produced by the previous paragraph may not be cut out transversely in the appropriate universal moduli space.
(Such a configuration could in principle arise in the compactification as a limit in which the same sphere bubbles off from two distinct points on the cylinder.)  Now most configurations in which there is such a ``duplicated sphere'' can also be ruled out by a similar technique as in the previous paragraph: the sphere would have to meet the other components of the configuration at two distinct points, and by forgetting one copy of the sphere but imposing the condition that the other components meet the sphere twice we replace a condition which is not cut out transversely in the universal moduli space by one which usually is cut out transversely and does not have a larger expected dimension.

We noted that this new condition is ``usually'' cut out transversely: the proof of this requires Proposition \ref{univevsub} (or, in the case where the components meeting the duplicated sphere are also spheres, \cite[Lemma 3.4.2]{MS}), but that proposition of course requires the assumption that it is not true that $\gamma^-=\gamma^+$ and $C_1\in\pi_2(\gamma^-,\gamma^-)$ is the trivial class.  Thus we arrive at the one remaining set of cases where a new argument is required, namely that where the unique cylindrical component of our configuration represents the trivial class; of course, by energy considerations it is easy to show that this is equivalent to the unique cylindrical component $u\co \R\times S^1\to M$ being a ``trivial cylinder'' $u(s,t)=\gamma(t)$.  In all other cases, the arguments sketched above allow one to replace a hypothetical nonsimple configuration by a simple configuration contained in a moduli space whose expected dimension before taking the quotient by translations of the cylinder is at most $\barmu(C)-\delta(I)-2\leq 0$; hence for generic $\frak{a}$ once we take the $\R$-symmetry into account the appropriate moduli space will be empty.

Accordingly we consider configurations in which the unique cylindrical component is a trivial cylinder $u(s,t)=\gamma(t)$.  It is in analyzing these types of configurations that we find it useful to vary the function $\beta\co \R\to\R$ that we have included in our auxiliary data. 

The first observation to make in this context is that for generic choices of the vector field $V$, no flowline of $V$ will pass through both a periodic orbit $\gamma\in P(H)$ and one of the submanifolds $f_i(N_i)$ (since the latter have codimension at least four).  Consequently for generic $\frak{a}$ the only possible nonempty strata corresponding to a single, trivial, cylindrical component are ones in which, in our earlier notation, $S_{1}^{C}=\varnothing$, \emph{i.e.}, in which all of the marked points used for the incidence conditions are on the spheres, not on the cylinder.  Moreover just as in \cite{HS} one can see that for generic choices of the pair $(J,V)$ any stratum involving just a trivial cylinder together with a single sphere bubble will be empty: the single sphere bubble would represent a class $A$ with $2c_1(A)=\barmu(C)\leq 2+\delta(I)$, and by imposing the incidence conditions corresponding to $I$ together with the condition that the sphere would need to pass through the periodic orbit $\gamma$ we would find the dimension of the relevant space equal to \[ 2n+2c_1(A)-6-\delta(I)+2-(2n-1)=2c_1(A)-\delta(I)-3<0.\]  (Note that this conclusion uses the fact that since all of the marked points are on the sphere which has bubbled off at a single point $(s_0,t_0)$ on the cylinder, the element $\vec{\eta}_1\in(\R\times S^1)^k$ from (\ref{geninc}) will have all its entries equal to $(s_0,t_0)$, and so the various $\tau_{\beta,i}(\vec{\eta}_1)$ will all be zero).  

A similar analysis together with the tricks that have been discussed earlier deal with all of the other strata, except those where the following situation holds: we have two or more copies of the same sphere (or potentially  multiple covers thereof) which have bubbled off at different points $(s_1,t),\ldots,(s_r,t)$  on the cylinder.  Indeed, in this case we have $u(s_1,t)=\ldots=u(s_r,t)=\gamma(t)$ since the cylinder is a trivial cylinder and so the condition that the sphere meets the cylinder at all $r$ of these points is obviously not cut out transversely.  
 
The main new difficulty that this situation creates is that, because the incidence conditions that a sphere must satisfy depend in part on the positions on the cylinder at which the other spheres are considered to have bubbled off, there may be particular choices of the bubbling points $(s_i,t)$ that force us to consider possible sphere bubbles whose homology classes and incidence conditions would have been ruled out by a dimension count if the other sphere components had not been present.  The way that we resolve this issue is by noting that the occurrence of such unexpected spheres imposes conditions on the parameters $s_i$, and that we can ensure that these conditions on the $s_i$ are cut out transversely in the universal moduli space by varying the function $\beta\co\R\to\R$.  

To keep the notational difficulties under control, we will illustrate the method on a particular type of substratum in which the essential point is present, leaving the general case to the reader.  Consider a case in which the cylindrical component is a trivial cylinder $u(s,t)=\gamma(t)$, while the spherical components have homology classes $A_1=\ldots=A_r=A$ and $A_{r+1}=B$, with all of the first $r$ spherical components represented by the same map $v_1\co S^2\to M$ and the remaining component represented by $v_{r+1}\co S^2\to M$.  Assume moreover that the multi-indices $I_{1}^{S},\ldots,I_{r}^{S}$ representing the incidence conditions 
obeyed by the various copies of $v_1$ are all the same, say equal to $G=(g_1,\ldots,g_p)\in\{1,\ldots,m\}^{p}$.  We then have (since $\barmu(C)-\delta(I)\leq 2$), \begin{equation}\label{mdim} 2rc_1(A)+2c_1(B)-r\delta(G)-\delta(I_{r+1}^{S})\leq 2.\end{equation}  Also write $I_{r+1}^{S}=(g'_1,\ldots,g'_q)$ for the multi-index representing the incidence condition corresponding to the other sphere.  We will consider the most highly degenerate case in which, on the $r$ copies of the representative of $A$, each of the $p$ incidence conditions are satisfied at the same point for each of the $r$ copies; less degenerate situations can be handled by combining the methods described below with earlier techniques.

The problematic configurations then entail the data of tuples \[ (J,H,V,\beta,v_1,v_{r+1},(s_1,t_1),(s_2,t_1),\ldots,(s_r,t_1),(s_{r+1},t_{r+1}),z_{10},\ldots,z_{1p},z_{r+1,0},\ldots,z_{r+1,q},n_1,\ldots,n_p,n'_1,\ldots,n'_q)\]
such that \begin{itemize}
\item[(i)] $\delbar_J v_1=\delbar_J v_{r+1}=0$,
\item[(ii)] $v_1(z_{10})=\gamma(t_1)$
\item[(iii)] $v_{r+1}(z_{r+1,0})=\gamma(t_{r+1})$
\item[(iv)] Where $\vec{\eta}\in (\R\times S^1)^{pr+q}$ has its $j$th entry given by $(s_b,t_1)$  if $j\in S_{b}^{S}$
with $1\leq b\leq r$ and by $(s_{r+1},t_{r+1})$ if $j\in S_{r+1}^{S}$, we have \[ \psi_{V}^{\tau_{\beta,j}(\vec{\eta})}(v_1(z_{1\ell}))=f_{g_{\ell}}(n_{\ell}) \mbox{ if $j$ is the $\ell$th largest element of } S_{b}^{S} \,(1\leq b\leq r)\] and \[ \psi_{V}^{\tau_{\beta,j}(\vec{\eta})}(v_{r+1}(z_{r+1,\ell}))=f_{g'_{\ell}}(n_{\ell})\mbox{ if $j$ is the $\ell$th largest element of }S_{r+1}^{S}.\]
\end{itemize}

Here all of the $z_{ij}$ vary in $S^2$; $n_{\ell}\in N_{g_{\ell}}$; and $n'_{\ell}\in N_{g'_{\ell}}$.  For simplicity we will assume that these points are all distinct, as the case where some of them coincide can be handled by incorporating our previous methods.

This is equivalent to the data of \[ (J,H,V,\beta,v_1,v_{r+1},t_1,t_{r+1},s_1,s_2,\ldots,s_{r+1},\{z_{ij}\},\{n_{\ell}\},\{n'_{\ell}\})\] 
 such that (i), (ii), (iii) above hold and we replace (iv) by \begin{itemize} \item[(v)]  \[ \psi_{V}^{\tau_{\beta,j}(\vec{\eta})}(v_1(z_{1\ell}))=f_{g_{\ell}}(n_{\ell}) \mbox{ if $j$ is the $\ell$th largest element of } S_{1}^{S} \] and \[ \psi_{V}^{\tau_{\beta,j}(\vec{\eta})}(v_{r+1}(z_{r+1,\ell}))=f_{g'_{\ell}}(n_{\ell})\mbox{ if $j$ is the $\ell$th largest element of }S_{r+1}^{S},\] and \item[(vi)] For $1\leq b\leq r$ and $1\leq \ell\leq p$ write $j(b\ell)$ for the $\ell$th largest element of $S_{b}^{S}$.  Then for $2\leq b\leq r$, $\tau_{\beta,j(b\ell)}(\vec{\eta})=\tau_{\beta,j(1\ell)}(\vec{\eta})$ .\end{itemize}

For $j=1,\ldots,r$, $j(b1)$ is the minimal element of $S_{b}^{S}$. The map \begin{align*} \mathcal{B}\times \R^{r+1}&\to \R^{r} \\
(\beta,s_1,\ldots,s_{r+1})&\mapsto (\tau_{\beta,j(11)}(\vec{\eta}),\ldots,\tau_{\beta,j(1r)}(\vec{\eta})) \end{align*} is easily seen by (\ref{taubetadef}) to have rank at least $r-1$ (the only reason that it might not have rank $r$ is that one of the indices $j(b1)$ might be equal to one and we always have $\tau_{\beta,1}=0$).  Consequently at least $r-2$ of the equations in (vi) above are cut out transversely.  Meanwhile the conditions in (i), (ii), (iii), and (v) are also cut out transversely, using Proposition \ref{ghostsub} and \cite[Lemma 3.4.2]{MS}.  Therefore the space of data $(J,H,V,\beta,\vec{t},\vec{s},\vec{z},\vec{n})$ obeying (i),(ii),(iii),(v), and the aforementioned $r$ equations of (vi) will be a Banach manifold, and we compute that the index of the projection to $(J,H,V,\beta)$ is \begin{align*} &(2n+2c_1(A)-6)+(2n+2c_1(B)-6)+2+(r+1)+2p+2q+2\sum_{\ell=1}^{p} d(i_{g_{\ell}})+2\sum_{\ell=1}^{q}d(i_{g'_{\ell}})\\ &\quad -\left(2n+2n+2np+2nq+r-2\right) \\ &=2c_1(A)+2c_1(B)-7-\delta(G)-\delta(I_{r+1}^{S})\end{align*}  In view of (\ref{mdim}), this quantity is negative if $2c_1(A)-\delta(G)\geq 0$, and so in this case the usual application of the Sard-Smale theorem shows that for generic $\frak{a}=(J,H,V,\beta)$ the substratum under consideration will not appear.  It remains to analyze the case that $2c_1(A)-\delta(G)<0$.  

This case is handled by an argument along the following lines.  Note first of all that, if the function $\beta$ were set equal to zero, then there would generically be no $J$-holomorphic representative of $A$ obeying the incidence conditions given by $G$ (and meeting the orbit $\gamma$), since the expected dimension of the space of such spheres is $2n+2c_1(A)-6-\delta(I)-(2n-3)=2c_1(A)-\delta(G)-3$.   However the presence of $\beta$ leads to the various $\tau_{\beta,j}$ changing as the $s_i$ vary, and for generic but fixed $J$ and $V$ the $\tau_{\beta,j}$ may occasionally attain exceptional values for which the sphere in question does occur.  Indeed if $2c_1(A)-\delta(G)=-\alpha$ then for generic $J$ and $V$ the presence of such a sphere imposes a condition of codimension $\alpha+3$ on the various $\tau_{\beta,j(1\ell)}(\vec{\eta})$ for $1\leq \ell\leq p$ (in particular if $p<\alpha+3$, or more generally if $(s_1,\ldots,s_{r+1})\mapsto (\tau_{\beta,j(11)},\ldots,\tau_{\beta,j(1p)})$ has rank less than $\alpha+3$, then the sphere will not arise for generic $J$ and $V$).  Now if $(s_1,\ldots,s_{r+1})\mapsto (\tau_{\beta,j(11)},\ldots,\tau_{\beta,j(1p)})$ has rank $c\geq \alpha+3$ one can see by taking advantage of the freedom to vary $\beta$ that at least $(r-1)c$ of the conditions in (vi) will be cut out transversely; this increases the codimension in the appropriate universal moduli space to $\alpha+3+(r-1)c\geq r(\alpha+3)=-r(2c_1(A)-\delta(G)-3)$.  Using this, the space of configurations of the form in question is found to be contained in a transversely-cut-out moduli space which, for generic $\frak{a}$, has dimension \[ 2n+2c_1(B)-6-\delta(I_{r+1}^{S})-(2n-3)+(r+1)+ r(2c_1(A)-\delta(G)-3),\] which by (\ref{mdim}) is at most $-2r$; thus the relevant space is empty for generic $\frak{a}$.

To sum up, using the methods that we have developed a sufficiently persistent reader may show that, for generic $\frak{a}$, if $\barmu(C)-\delta(I)\leq 2$ the only nonempty stratum of our compactified moduli space corresponding to just one cylindrical component is the main stratum $\mathcal{M}^{\frak{a},\ast}_{id}(\gamma^-,\gamma^+,C;N_I)$, and the only nonempty stratum corresponding to more than one cylindrical component is the usual space of two-stage broken trajectories (which arises only if $\barmu(C)-\delta(I)=2$).  This suffices to prove Proposition \ref{pertmain}.


\begin{thebibliography}{9999999}
\bibitem[Ban]{Ban}  A. Banyaga. \emph{Sur la structure du groupe des diff\'eomorphismes qui pr\'eservent une
forme symplectique}. Comment. Math. Helv. \textbf{53} (1978), 174--227. MR0490874.
\bibitem[Bay]{Bay} A. Bayer. \emph{Semisimple quantum cohomology and blow-ups}. Int. Math. Res. Not.  \textbf{2004},  no. 40, 2069--2083. MR2064316.  
\bibitem[BaM]{BaM} A. Bayer and Y. Manin. \emph{(Semi)simple exercises in quantum cohomology}. In \emph{The Fano Conference}. Univ. Torino, Turin, 2004, 143--173. MR2112573.
\bibitem[Bo]{Bo} N. Bourbaki. \emph{Algebra II: Chapters 4--7}. Elements of Mathematics. Springer-Verlag, 2003. MR1994218.
\bibitem[Ca]{Ca} D. Calegari. \emph{scl}. MSJ Memoirs, \textbf{20}.  Mathematical Society of Japan, Tokyo, 2009. MR2527432.
\bibitem[Cho]{Cho} C.-H. Cho \emph{Non-displaceable Lagrangian submanifolds and Floer cohomology with non-unitary line bundle}. J. Geom. Phys.  \textbf{58}  (2008), no. 11, 1465--1476. MR2463805. 
\bibitem[Cio]{Cio} G. Ciolli. \emph{On the quantum cohomology of some Fano threefolds and a conjecture of Dubrovin}. Internat. J. Math.  \textbf{16}  (2005), no. 8, 823--839. MR2168069.
\bibitem[Del]{Del} T. Delzant. \emph{Hamiltoniens p\'eriodiques et images convexes de l'application moment}. Bull. Soc. Math. France  \textbf{116}  (1988), no. 3, 315--339. MR0984900.
\bibitem[Dub]{Dub} B. Dubrovin. \emph{Geometry and analytic theory of Frobenius manifolds}. In \emph{Proceedings of the International Congress of Mathematicians, Vol. II (Berlin, 1998)}. Doc. Math.  \textbf{1998},  Extra Vol. II, 315--326. MR1648082. 
\bibitem[Eis]{Eis} D. Eisenbud. \emph{Commutative algebra with a view toward algebraic geometry}. Grad. Texts Math. \textbf{150}, Springer-Verlag, 1995. MR1322960.
\bibitem[ElP]{ElP} Y. Eliashberg and L. Polterovich. \emph{Symplectic quasi-states on the quadric surface and Lagrangian submanifolds}. arXiv:1006.2501.
\bibitem[EP03]{EP03} M. Entov and L. Polterovich. \emph{Calabi quasimorphism and quantum homology}. Int. Math. Res. Not.  \textbf{2003},  no. 30, 1635--1676. MR1979584.
\bibitem[EP06]{EP06} M. Entov and L. Polterovich.  \emph{Quasi-states and symplectic intersections}. Comment. Math. Helv. \textbf{81} (2006), 75--99. MR2208798.
\bibitem[EP08]{EP08} M. Entov and L. Polterovich. \emph{Symplectic quasi-states and semi-simplicity of quantum homology}. In \emph{Toric topology}   Contemp. Math. \textbf{460}, AMS, Providence, 2008, 47--70. MR2428348.
\bibitem[EP09]{EP09} M.  Entov and L. Polterovich. \emph{Rigid subsets of symplectic manifolds}. Compos. Math.  \textbf{145}  (2009), no. 3, 773--826. MR2507748.  
\bibitem[EPZ]{EPZ} M. Entov, L. Polterovich, and F. Zapolsky. \emph{Quasi-morphisms and the Poisson bracket}. Pure Appl. Math. Q.  \textbf{3}  (2007), no. 4, 1037--1055. MR2402596. 
\bibitem[Fl]{Fl} A. Floer. \emph{The unregularized gradient flow of the symplectic action}. Comm. Pure Appl. Math. \textbf{41} (1988), no. 6, 775--813. MR0948771.
\bibitem[FHS]{FHS} A. Floer, H. Hofer, and D. Salamon. \emph{Transversality in elliptic Morse theory for the symplectic action}. Duke Math. J. \textbf{80} (1995), no. 1, 251--292. MR1360618.
\bibitem[Fu]{Fu} K. Fukaya. \emph{Floer homology for families---a progress report}. In \emph{Integrable Systems, Topology, and Physics (Tokyo, 2000)}. Contemp. Math. \textbf{309}, AMS, Providence, 2002, 33--68. MR1953352.
\bibitem[FO]{FO} K. Fukaya and K. Ono. \emph{Arnold conjecture and Gromov--Witten invariants}.  Topology \textbf{38} (1999), 933--1048. MR1688434.
\bibitem[FOOO09]{FOOO09} K.Fukaya, Y.-G. Oh, H. Ohta, and K. Ono. \emph{Lagrangian Intersection Floer Theory: Anomaly and Obstruction}. 2 vols. AMS, Providence, 2009. MR2553465.
\bibitem[FOOO10]{FOOO10} K. Fukaya, Y.-G. Oh, H. Ohta, and K. Ono. \emph{Lagrangian Floer theory on compact toric manifolds. I.} Duke Math. J.  \textbf{151}  (2010), no. 1, 23--174. MR2573826. 
\bibitem[FOOO11]{FOOO} K. Fukaya, Y.-G. Oh, H. Ohta, and K. Ono. \emph{Spectral invariants with bulk, quasimorphisms and Lagrangian Floer theory}, arXiv:1105.5123.
\bibitem[Ga]{Ga} A. Gathmann. \emph{Gromov-Witten invariants of blow-ups}. J. Algebraic Geom.  \textbf{10}  (2001), no. 3, 399--432. MR1832328.
\bibitem[GIV1]{GIV1} A. Grothendieck with J. Dieudonn\'e.  \emph{\'El\'ements de g\'eom\'etrie alg\'ebrique. IV. \'Etude locale des sch\'emas et des morphismes de sch\'emas, Pr\`emiere partie.}  Inst. Hautes \'Etudes Sci. Publ. Math. \textbf{20} (1964), 5--259. MR0173675.
\bibitem[GIV4]{GIV4} A. Grothendieck with J. Dieudonn\'e.  \emph{\'El\'ements de g\'eom\'etrie alg\'ebrique. IV. \'Etude locale des sch\'emas et des morphismes de sch\'emas, Quatri\`eme partie.}  Inst. Hautes \'Etudes Sci. Publ. Math. \textbf{32} (1967), 5--361. MR0238860. 
\bibitem[Ha]{Ha} R. Hartshorne. \emph{Algebraic geometry}. Grad. Texts Math. \textbf{52}, Springer-Verlag, 1977. MR0463157.  
\bibitem[HMT]{HMT} C. Hertling, Y. Manin, and C. Teleman. \emph{An update on semisimple quantum cohomology and $F$-manifolds}. Tr. Mat. Inst. Steklova  \textbf{264}  (2009), Mnogomernaya Algebraicheskaya Geometriya, 69--76. MR2590836.
\bibitem[HS]{HS} H. Hofer and D. Salamon. Floer homology and Novikov rings. In \emph{The Floer memorial volume}. Progr. Math. \textbf{133}, Birkh\"auser, Basel, 1995, 483--524. MR1362838.
\bibitem[HV]{HV} H. Hofer and C. Viterbo. \emph{The Weinstein conjecture in the presence of holomorphic spheres}. Comm. Pure Appl. Math.  \textbf{45}  (1992), no. 5, 583--622. MR1162367.
\bibitem[HWZ]{HWZ} H. Hofer, K. Wysocki, and E. Zehnder. \emph{A general Fredholm theory I: a
splicing-based differential geometry}. J. Eur. Math. Soc. \textbf{9} (2007) no. 4,  841-–876. MR2341834.
\bibitem[Hu]{Hu} J. Hu. \emph{Gromov-Witten invariants of blow-ups along points and curves}. Math. Z.  \textbf{233}  (2000), no. 4, 709--739. MR1759269.  
\bibitem[Ir]{Ir} H. Iritani. \emph{Convergence of quantum cohomology by quantum Lefschetz}. J. Reine Angew. Math. \textbf{610} (2007), 29--69. MR2359850. 
\bibitem[KM]{KM} M. Kontsevich and Y.  Manin. \emph{Gromov-Witten classes, quantum cohomology, and enumerative geometry}. Comm. Math. Phys.  \textbf{164}  (1994), no. 3, 525--562. MR1291244.
\bibitem[LiT]{LiT} J. Li and G. Tian. \emph{Virtual moduli cycles and Gromov-Witten invariants of general symplectic manifolds}. In \emph{Topics in symplectic $4$-manifolds (Irvine, CA, 1996)},  47--83, First Int. Press Lect. Ser., I, Int. Press, Cambridge, MA,  1998. MR1635695.
\bibitem[LiuT98]{LiuT98} G. Liu and G. Tian. \emph{Floer homology and Arnold
conjecture}. J. Diff. Geom. \textbf{49} (1998), no. 1, 1--74. MR1642105.
\bibitem[LiuT00]{LiuT00} G. Liu and G. Tian. \emph{Weinstein conjecture and GW-invariants}. Commun. Contemp. Math.  \textbf{2}  (2000), no. 4, 405--459. MR1806943.  
\bibitem[LO]{LO} H.-V. Le and K. Ono. \emph{Cup-length estimates for symplectic fixed points}. In \emph{Contact and symplectic geometry (Cambridge, 1994)}, Publ. Newton Inst., \textbf{8}, Cambridge Univ. Press, Cambridge, 1996, 268--295. MR1432466.
\bibitem[Lu04]{Lu04} G. Lu. \emph{An explicit isomorphism between Floer homology and quantum homology}. Pacific J. Math.  \textbf{213}  (2004), no. 2, 319--363. MR2036923.
\bibitem[Lu06]{Lu06} G.Lu. \emph{Gromov-Witten invariants and pseudo symplectic capacities}. Israel J. Math. \textbf{156} (2006), 1--63. MR2282367.
\bibitem[Man]{Man} Y. Manin. \emph{Frobenius manifolds, quantum cohomology, and moduli spaces}. AMS Colloquium Publications, \textbf{47}, AMS, Providence, 1999. MR1702284.
\bibitem[Mat]{Ma} H. Matsumura. \emph{Commutative algebra. Second edition}. Mathematics Lecture Note Series, \textbf{56}. Benjamin/Cummings Publishing Co., Inc., Reading, Mass., 1980. MR0575344.
\bibitem[M]{M} D. McDuff. \emph{Hamiltonian $S^1$-manifolds are uniruled.} Duke Math. J. \textbf{146}  (2009), no. 3, 449--507. MR2484280.
\bibitem[MS]{MS} D. McDuff and D. Salamon. \emph{$J$-holomorphic curves and symplectic topology}. AMS Colloquium Publications, \textbf{52}, AMS, Providence, 2004. MR2045629.
\bibitem[Oh05]{Oh05} Y.-G. Oh. \emph{Spectral invariants and the length-minimizing property of Hamiltonian paths}. Asian J. Math. \textbf{9} (2005), no. 1, 1--18. MR2150687.
\bibitem[Oh06]{Oh06} Y.-G. Oh. \emph{Lectures on Floer theory and spectral invariants of Hamiltonian flows}. In \emph{Morse-theoretic methods in nonlinear analysis and in symplectic topology}. NATO Sci. Ser. II Math. Phys. Chem., 217, Springer, Dordrecht, 2006, 321--416.  MR2276955.
\bibitem[OZ]{OZ} Y.-G. Oh and K. Zhu. \emph{Floer trajectories with immersed nodes and scale-dependent gluing}. arXiv:0711.4187.
\bibitem[OT]{OT} Y. Ostrover and I. Tyomkin. \emph{On the quantum homology algebra of toric Fano manifolds}. Selecta Math. (N.S.)  \textbf{15}  (2009), no. 1, 121--149. MR2511201.
\bibitem[PSS]{PSS} S. Piunikhin, D. Salamon, and M. Schwarz. \emph{Symplectic Floer-Donaldson theory and quantum
cohomology}, in Publ. Newton. Inst. (Thomas, C. B., eds.), \textbf{8}, Cambridge University
Press, Cambridge, England, 1996, pp. 171-200. MR1432464.
\bibitem[PoSz]{PoSz} G. P\'olya and G. Szeg\"o. \emph{Problems and Theorems in Analysis. I. Series, integral calculus, theory of functions.}  Springer-Verlag, Berlin-New York, 1978. MR1492447.
\bibitem[Ri]{Ri} A. Ritter. \emph{Novikov-symplectic cohomology and exact Lagrangian embeddings}. Geom. Topol. \textbf{13} (2009), no. 2, 943--978. MR2470967.
\bibitem[RS]{RS} J. Robbin and D. Salamon. \emph{The Maslov index for paths}. Topology  \textbf{32}  (1993), no. 4, 827--844. MR1241874 
\bibitem[Ru]{R} Y. Ruan. \emph{Virtual neighborhoods and pseudo-holomorphic curves.} In \emph{Proceedings of the 6th G\"okova Geometry-Topology Conference}. Turkish J. Math. \textbf{23}  (1999), no. 1, 161--231. MR1701645.
\bibitem[RT]{RT} Y. Ruan and G. Tian. \emph{A mathematical theory of quantum cohomology}. J. Differential Geom.  \textbf{42}  (1995), no. 2, 259--367. MR1366548.
\bibitem[Sal]{Sal} D. Salamon. \emph{Lectures on Floer homology}. In
\emph{Symplectic geometry and topology (Park City, Utah, 1997)}.
AMS, Providence, 1999. MR1702944.
\bibitem[Sc95]{Sc95} M. Schwarz. \emph{Cohomology operations from $S^1$-cobordisms in Floer homology}. ETH thesis, 1995.
\bibitem[Sc00]{Sc00} M. Schwarz. \emph{On the action spectrum for closed symplectically aspherical manifolds}. Pacific J. Math. \textbf{193} (2000), 419--461. MR1755825.
\bibitem[Sta]{Sta} The Stacks Project, http://www.math.columbia.edu/algebraic\_geometry/stacks-git/
\bibitem[Th]{Th} R.  Thom. \emph{Quelques propri\'et\'es globales des vari\'et\'es diff\'erentiables}.  Comment. Math. Helv.  \textbf{28},  (1954). 17--86. MR0061823.  
\bibitem[U08]{U08} M. Usher. \emph{Spectral numbers in Floer theories}. Compos. Math.  \textbf{144}  (2008), no. 6, 1581--1592. MR2474322.
\bibitem[U09]{U09} M. Usher. \emph{Boundary depth in Hamiltonian Floer theory and its applications to Hamiltonian dynamics and coisotropic submanifolds}. arXiv:0903.0903, to appear in Israel J. Math.
\bibitem[U10a]{U10a} M. Usher. \emph{The sharp energy-capacity inequality}. Commun. Contemp. Math. \textbf{12} (2010), no. 3, 457--473.
\bibitem[U10b]{U10b} M. Usher. \emph{Duality in filtered Floer-Novikov complexes}. J. Topol. Anal. \textbf{2} (2010), no. 2, 233--258.
\bibitem[Wi90]{Wi90} E. Witten. \emph{On the structure of the topological phase of two-dimensional gravity}. Nuclear Phys. B. \textbf{340}  (1990), no. 2-3, 281--332. MR1068086. 
\bibitem[Wi91]{Wi91} E. Witten. \emph{Two-dimensional gravity and intersection theory on moduli space}. In \emph {Surveys in differential geometry (Cambridge, MA, 1990)}.   Lehigh Univ., Bethlehem, PA, 1991, 243--310. MR1144529. 

\end{thebibliography}
\end{document}